\documentclass[11pt,pdftex]{scrartcl}

\usepackage{comment}
\usepackage{amsmath,amsthm,amssymb}
\usepackage{mathtools}
\usepackage{stmaryrd}
\usepackage{bbm}
\usepackage[all]{xy}
\SelectTips{cm}{}
\usepackage{sfmath}
\usepackage[utf8]{inputenc}
\usepackage{graphics}
\usepackage{float}
\usepackage{graphicx}
\usepackage{booktabs}
\usepackage{hyperref}
\usepackage{quiver}
\RequirePackage{color}
\definecolor{VeryDarkGreen}{rgb}{0,0.22,0.12}
\definecolor{VeryDarkBrown}{rgb}{0.16,0.12,0.08}
\hypersetup{
bookmarks,
bookmarksdepth=2,
bookmarksopen,
bookmarksnumbered,
pdfstartview=FitH,
colorlinks,backref,hyperindex,
linkcolor=VeryDarkBrown,
citecolor=VeryDarkGreen,
urlcolor=VeryDarkGreen
}
\usepackage{todonotes}
\newcommand{\mydot}{\bullet}

\newcommand{\myR}{{\mathbf R}}
\newcommand{\myL}{{\mathbf L}}
\newcommand{\stens}{\mathbin{\mathop{\otimes}\limits^{!}}}
\newcommand{\bigstens}[1]{\displaystyle \mathop{\bigotimes}^{!}_{#1}}
\newcommand{\gr}{\mathrm{gr}}
\newcommand{\cosimp}[3]{\xymatrix@1{#1 \ar@<.4ex>[r] \ar@<-.4ex>[r] & {\ }#2 \ar@<0.8ex>[r] \ar[r] \ar@<-.8ex>[r] & {\ } #3 \ar@<1.2ex>[r] \ar@<.4ex>[r] \ar@<-.4ex>[r] \ar@<-1.2ex>[r] & \cdots }}

\newcommand{\tld}{\widetilde }

\reversemarginpar

\renewcommand{\mathbb}{\mathbbmss}
\newcommand{\bF}{{\mathbb{F}}}

    \makeatletter
    \@ifdefinable\equationname{\let\equationname\equationautorefname}
    \def\equationautorefname~#1\@empty\@empty\null{(#1\@empty\@empty\null)}%
    \@ifdefinable\AMSname{\let\AMSname\AMSautorefname}
    \def\AMSautorefname~#1\@empty\@empty\null{(#1\@empty\@empty\null)}
    \@ifdefinable\itemname{\let\itemname\itemautorefname}
    \def\itemautorefname~#1\@empty\@empty\null{(#1\@empty\@empty\null)%
    }%
    \makeatother
    \makeatletter
    \renewcommand{\theenumi}{\alph{enumi}}

    \renewcommand{\p@enumii}{\theenumi$\m@th\vert$}
    \renewcommand{\labelitemi}{$\m@th\circ$}
    \renewcommand{\labelitemii}{$\m@th\diamond$}
    \makeatother
    \renewcommand{\phi}{\varphi}
    
    \renewcommand{\theta}{\vartheta}
    
    \renewcommand{\epsilon}{\varepsilon}
\renewcommand{\to}[1][]{\xrightarrow{\ #1\ }}

\newcommand{\onto}[1][]{\protect{\xrightarrow{\ #1\ }\hspace{-0.8em}\rightarrow}}
\newcommand{\into}[1][]{\xhookrightarrow{\ #1\ }}

\RequirePackage{amsthm}
\RequirePackage{aliascnt}
\newcommand{\basetheorem}[3]{%
    \newtheorem{#1}{#2}[#3]
    \newtheorem*{#1*}{#2}
    \expandafter\def\csname #1autorefname\endcsname{#2}
}%
\newcommand{\maketheorem}[3]{%
    \newaliascnt{#1}{#3}
    \newtheorem{#1}[#1]{#2}
    \aliascntresetthe{#1}
    \expandafter\def\csname #1autorefname\endcsname{#2}
    \newtheorem*{#1*}{#2}
}%
\theoremstyle{plain}   %-------------------standard Style-------------------------
\basetheorem{theorem}{Theorem}{subsection}
\maketheorem{prop}{Proposition}{theorem}
\maketheorem{proposition}{Proposition}{theorem}
\maketheorem{cor}{Corollary}{theorem}
\maketheorem{corollary}{Corollary}{theorem}
\maketheorem{lem}{Lemma}{theorem}
\maketheorem{lemma}{Lemma}{theorem}
\maketheorem{claim}{Claim}{theorem}
\maketheorem{const}{Construction}{theorem}
\maketheorem{construction}{Construction}{theorem}

\theoremstyle{definition}    %------------text not italic style------------------

\maketheorem{defn}{Definition}{theorem}
\maketheorem{definition}{Definition}{theorem}

\theoremstyle{remark}    %----------------also text not italic ------------------

\maketheorem{exa}{Example}{theorem}
\maketheorem{example}{Example}{theorem}
\maketheorem{rem}{Remark}{theorem}
\maketheorem{remark}{Remark}{theorem}

\newcommand{\cO}{\mathcal{O}}
\newcommand{\id}{\operatorname{id}}

\newcommand{\Spec}{\operatorname{Spec}}

\newcommand{\Hom}{\operatorname{Hom}}
\newcommand{\Ext}{\operatorname{Ext}}
\newcommand{\RHom}{\myR\!\Hom}
\newcommand{\moduleomegacan}[1]{\omega_{#1}}
\newcommand{\omegacan}[1]{\omega^{\mydot}_{#1}}

\newcommand{\colim}{\operatorname{colim}}

\newcommand{\Perf}{\operatorname{Perf}}
\newcommand{\Ffinites}{\operatorname{Ring}^{p\text{-}\operatorname{gens}}}
\newcommand{\RegFfinites}{\operatorname{RegRing}^{p\text{-}\operatorname{basis}}}
\newcommand{\sHom}{\mathcal{H}om}
\newcommand{\RsHom}{\myR\!\mathcal{H}om}
\newcommand{\Wedge}{\bigwedge}
\begin{document}

% Bitte zunächst dies angemessen ausfüllen
\title{$F$-finite schemes have a dualizing complex}
\author{Bhargav Bhatt, Manuel Blickle, Karl Schwede, Kevin Tucker}
%\date{Aktuelles Datum Eintragen}

%\subtitle{}

\maketitle

\begin{abstract}
In this paper we show that any Noetherian $F$-finite scheme has a dualizing complex $\omegacan{X}$ with the property that for all finite type maps $f \colon X \to Y$ between $F$-finite Noetherian schemes there is a canonical isomorphism $\omegacan{X} \to[\cong] f^!\omegacan{Y}$  in $D^b_{coh}(X)$. This, in particular, applies to the Frobenius morphism $F \colon X \to X$ so that we obtain a canonical isomorphism $\omegacan{X} \to[\cong] F^!\omegacan{X}$.

To prove this, we rely on a result of Gabber that every Noetherian $F$-finite ring is a quotient of a regular ring, from which it follows that every $F$-finite Noetherian scheme has a (potentially non-canonical) dualizing complex \cite{Gabber_someTstructures}.  To make this canonical, we identify the dualizing complex of any $F$-finite Noetherian scheme as a unit of an alternate symmetric monoidal structure on $D^b_{coh}(X)$ we call the $!$-tensor product.  We also sketch an alternate approach to finding this canonical dualizing complex following the more classical     \cite{HartshorneResidues,ConradGDualityAndBaseChange}.  %This alternate approach requires checking numerous compatibilities not all of which we have verified.
%
%This is more elementary at the price of numerous compatibilities which have to be verified, and we only verify it in full detail in the case of nonschemes. The second, much more conceptual proof uses a new symmetric monoidal structure on $D^b_{coh}(X)$ called the $!$-tensor product, and identifies the canonical dualizing complex constructed above as the unit for this monoidal structure.
%
    % In this paper we build on an argument sketched for affine schemes by Gabber in \cite{Gabber_someTstructures} to show that every $F$-finite Noetherian scheme $X$ has a \emph{canonical dualizing complex} $\omegacan{X}$ with the property that for all essentially finite type maps $f \colon X \to Y$ between $F$-finite Noetherian schemes there is a canonical isomorphism $\omegacan{X} \to[\cong] f^!\omegacan{Y}$  in $D^b_{coh}(X)$. Applied to the absolute Frobenius morphism $F \colon X \to X$ we obtain a canonical isomorphism $\omegacan{X} \to[\cong] F^!\omegacan{X}$. 

    % In the second part of the paper we construct a symmetric monoidal structure on $D^b_{coh}(X)$ called the $!$-tensor product, and identify the canonical dualizing complex constructed above as the unit for this monoidal structure. 
\end{abstract}

\tableofcontents

\section{Introduction}

Among the most important rings in algebraic and arithmetic geometry in positive characteristic are the $F$-finite rings, namely commutative rings of characteristic $p>0$ for which the Frobenius endomorphism is finite. Examples include quotients and localizations of polynomial or power series rings over a perfect field. Such rings and schemes are known to avoid various pathologies. For example, every Noetherian $F$-finite ring is excellent of finite Krull dimension \cite{Kunz_noetherianCharP}; moreover, every such ring is a quotient of an $F$-finite regular ring \cite{Gabber_someTstructures}, and hence admits a dualizing complex (cf. \cite[Theorem 3.5.1]{LurieEllipticCohomologyII} and \cite{MaPolstra.FSingularitiesBook} for other expositions of Gabber's result).

Our purpose here is to leverage Gabber's construction further by showing that every Noetherian $F$-finite scheme has a \emph{canonical dualizing complex} compatible with all finite type morphisms. More precisely, we show:

\begin{theorem}[{\autoref{candualffin}, \autoref{FFinPullProper}, \autoref{cor.FiniteTypePullbacksViaMonoidal} {\itshape cf.} \autoref{thm.CanDualRegular}}]
\label{thmdualexistintro}
    For every Noetherian $F$-finite scheme $X$ there exists a canonical dualizing complex $\moduleomegacan{X}^{\mydot} \in D^b_{coh}(X)$ such that if $f : Y \to X$ is any finite type morphism of $F$-finite schemes, there is a canonical isomorphism% in the derived category
    \[ 
         \moduleomegacan{Y}^{\mydot} \to[\cong] f^! \moduleomegacan{X}^{\mydot}.
    \]
    Moreover, if $X$ is an $F$-finite regular integral scheme, we have
\[
    \moduleomegacan{X}^\mydot \cong \Wedge^n \Omega_{X}[n]   
\]
where $n$ is the generic rank of the finitely generated locally free sheaf of (absolute) Kähler differentials $\Omega_X := \Omega_{X/\bF_p}$.
\end{theorem}

Said differently, we aim to show the existence of dualizing complexes on Noetherian $F$-finite schemes which are $(-)^!$-crystals for finite type morphisms. Note in particular, for any $F$-finite Noetherian ring $R$ and any surjection $S \twoheadrightarrow R$ from an $F$-finite regular ring such that $\Omega_S$ has generic rank $n$, our canonical dualizing complex satisfies
\begin{equation}
    \label{eq.CanonicalDualizingComplexForFFiniteRing}
    \omegacan{R} \cong \myR\Hom_S(R, \wedge^n \Omega_S[n])
\end{equation}
and hence the right-hand side is independent of the choice of $S$. However, unlike in the classical settings of duality theory, we do not have the reference point afforded by working (essentially) of finite type over a (nice) fixed base scheme. This is especially relevant for the absolute Frobenius endomorphism, where Theorem~\ref{thmdualexistintro} immediately yields:

\begin{cor} 
\label{cor.UpperShriekDualizes}
If $X$ is an $F$-finite Noetherian scheme, the canonical dualizing complex $\omegacan{X}$ satisfies $F^! \omegacan{X} \cong \omegacan{X}$.
    In other words, there is a canonical $F_*\cO_X$-linear isomorphism 
    \[
        \RsHom_{\cO_X}(F_* \cO_X, \omegacan{X}) \cong F_* \omegacan{X}.
    \]
\end{cor}
In the special case of regular $X$, a more direct proof is contained \autoref{thm.CanDualRegular}. The result is nontrivial even in the regular domain case, as \autoref{exam.EllipticCurve} illustrates. Indeed, for a connected scheme, any two dualizing complexes differ by a shift and tensoring with an invertible sheaf, so any other dualizing complex is, up to shift, of the form $\omegacan{X} \otimes L$ for some invertible sheaf $L$. If we ask that $\omegacan{X} \otimes L$ also satisfy Frobenius compatibility, then we would need
\(
    F^!(\omegacan{X} \otimes L) \cong \omegacan{X} \otimes L.
\)
But since $F^!(\omegacan{X} \otimes L) \cong \omegacan{X} \otimes L^{\otimes p}$, this forces $L^{\otimes (p-1)} \cong \cO_X$, so $L$ must be a $(p-1)$-torsion line bundle. This is not true in general. Thus the desired compatibility with $(-)^!$ for finite type maps, and in particular with Frobenius, depends crucially on our choice of dualizing complex, which in the regular $F$-finite case is determined by the determinant of the K\"ahler differentials.

By applying the lowest nonzero cohomology functor to our dualizing complex we also obtain \emph{canonical} canonical modules. When $X$ is normal, these canonical modules admit an explicit elementary description. In particular, the following is an immediate consequence of the preceding results.
\begin{cor}%[\autoref{cor.CanonicalCanonicalModuleDefinition}]
Suppose that $X$ is a connected, locally equidimensional, $F$-finite Noetherian scheme.
 If $\moduleomegacan{X}$ denotes the lowest nonzero cohomology of $\moduleomegacan{X}^{\mydot}$, then 
$\moduleomegacan{X}$ is a canonical module for $X$ such that 
$
    \sHom_{\cO_X}(F_* \cO_X, \moduleomegacan{X}) \cong F_* \moduleomegacan{X}
$ as $F_*\cO_X$-modules. Moreover, if $X$ is normal and integral, we have
\[
    \moduleomegacan{X} \cong (\Wedge^n \Omega_{X})^{\vee\vee}    
\]
where $n$ is the generic rank of $\Omega_X := \Omega_{X/\bF_p}$  and $(-)^{\vee} = \sHom(-, \cO_X)$ denotes the $\cO_X$-dual.
\end{cor}
\noindent
While similar local statements hold for any choice of canonical module, the key point is that the canonicity of our $\moduleomegacan{X}$ ensures global compatibility with the Frobenius endomorphism. 

%\todo[inline]{Insert shriek tensor product stuff. \\ (Kevin: Here's a one paragraph attempt?) \\ BB: I edited it slightly}

%The above results give a {\em construction} of the canonical dualizing complex of a Noetherian $F$-finite scheme $X$ via classical duality theory. In the second part of the paper, we give an {\em intrinsic characterization} of this complex (restricting to the affine case for simplicity). The main observation is that even though the absolute product $X \times_{\mathrm{Spec}(\mathbb{F}_p)} X$ is not Noetherian in general, there is a well-behaved notion of the completion of $X \times_{\Spec(\mathbb{F}_p)} X$ along the diagonal $X \to X \times_{\Spec(\mathbb{F}_p)} X$ as a derived Noetherian scheme (\autoref{FfinCompDiag}). This enables us to construct a symmetric monoidal structure on $D^+_{coh}(X)$ that we call the $\stens$-product (\autoref{stensFFin}).  We then obtain the following. %We then identify the canonical dualizing complex (together with an additional datum) as the unit for the $\stens$-product.

The key ingredient to proving \autoref{thmdualexistintro} is the construction of a new and well-behaved completion $S \widehat{\otimes_{\mathbf{F}_p}} S$ of $S \otimes_{\mathbf{F}_p} S$ for a Noetherian $F$-finite ring $S$, see \autoref{compffin}. This completion allows us to construct a $\stens$-product operation on $D^+_{coh}(S)$ via the formula
\[ M \stens_S N = \RHom_{S \widehat{\otimes_{\mathbf{F}_p}} S}(S, M \widehat{\otimes_{\mathbf{F}_p}} N),\]
see \autoref{stensFFin}. Using this operation, we obtain the following characterization of the canonical dualizing complexes for $F$-finite rings:

\begin{theorem}[\autoref{mainthmdualcanffin} and \autoref{rem.UnitForShriekTesnorIsCanDual}]
    Let $S$ be a Noetherian $F$-finite ring. Then the $\stens$-product is a naturally symmetric monoidal unital product on $D^+_{coh}(S)$. 
    Furthermore, the unit of the $\stens$-product operation is identified as the canonical dualizing complex $\omegacan{S} \in D^b_{coh}(S)$.  \end{theorem}

\noindent
The reason this observation yields good behavior of the dualizing complex under finite type maps is that the upper shriek functor respects this monoidal structure; see \autoref{FFinPull} and \autoref{FFinPullProper}.

Instead of using the above approach via the $\stens$-monoidal structure, one may attempt to redo the classical approach to Grothendieck duality and dualizing complexes to obtain our main result \autoref{thmdualexistintro}, using the explicit descriptions of $f^!$ for finite and smooth maps as in \cite{HartshorneResidues,ConradGDualityAndBaseChange}.
  We explore this approach in  \autoref{sec.ClassicalDuality}.  Explicitly, in \autoref{thm.CanDualRegPullsback} we prove that $f^!$ yields an isomorphism of canonical dualizing complexes $\omega_S^{\mydot} \cong f^! \omega_R^{\mydot}$ for finite type maps $R \to S$ of $F$-finite regular rings, independent of the factorization of $f$ into polynomial and finite maps.  Here if $R$ is a regular ring such that $\Omega_{R/\bF_p}$ has generic rank $n$, then $\omegacan{R} := \bigwedge^n \Omega_{R/\bF_p}[n]$ as above.  Then, in \autoref{prop.DualizingComplexesIsomorphicForArbitraryFFiniteRing} for any $F$-finite ring $A$, we utilize this to provide a concrete isomorphism between the dualizing complexes $\myR\Hom_R(A, \omegacan{R})$ and $\myR\Hom_S(A, \omegacan{S})$ where $R, S$ are regular rings and there are finite maps $R, S \to A$ that make $A$ into an $R$ (or $S$) module.  However, while we believe this to be possible, we do not develop herein
  a full alternative proof of the main result as this would require lengthy verification that this isomorphism satisfies the necessary additional compatibilities.  For more about what we think would be required to complete this approach, see the discussion at the end of \autoref{ss:CanDualFFin}.

Gabber's construction enters so centrally into both the proofs and the formulation of our main results that it is worth indicating its role already here. In \autoref{sec:gabberconstr} we review the construction itself, and in \autoref{subsection:functorialityofGabber} we establish functoriality and uniqueness properties. % that will be used throughout the paper.
Recall that Gabber's construction associates to each $F$-finite ring $R$ with $p$-generating set $x := x_1, \dots, x_n$ (that is, $R = R^p[x_1, \dots, x_n]$) a regular ring $G(R;x)$ together with a surjection
\[
    G(R; x) \to R
\]
such that some $p$-basis of $G(R;x)$ maps onto $x$.
We also prove several simple functoriality and uniqueness results for this construction which we suspect are known to experts but for which we know of no reference. In particular, we explain how Gabber's construction relates to a surjection $S \to R$ from a regular ring with a $p$-basis (\autoref{lem.UniversalGabberViaCompletion}), and we describe functoriality properties of $G(R;x)$ (\autoref{thm.gabbifyIsRightAdjointCompletion}, \autoref{cor.GabberConstructionCompletedAndCompleted}, \autoref{cor.GabberConstructionAndLocalization}), including an explicit description of what happens when one enlarges the $p$-generating set (\autoref{cor.ExtendingThePGeneratingSet}).

\subsection*{Outline of the paper}

\begin{itemize}
\item \autoref{sec.Preliminaries} provides background on $p$-basis, differential bases and the fundamental local isomorphism.
\item In \autoref{sec.CanDualizingClassical} we recall Gabber's construction of regular rings surjecting onto $F$-finite rings, proving some basic results about it.
\item In \autoref{sec.DualizingComplexesAsUnits}, we prove our main results using the symmetric monoidal structure discussed above.
\item In \autoref{sec.ClassicalDuality}, we sketch an approach to our main results via classical Grothendieck duality.  %However, we do not prove all the necessary details in order to make this fully rigorous in general, see \autoref{ss:CanDualFFin}.
\end{itemize}

\subsection*{Acknowledgements}

The first three authors first considered this problem when visiting Oberwolfach in September 2017.  They also worked on it in MSRI in Spring 2019.  The final three authors worked on this at CIRM Luminy in January 2023.  The authors thank Ofer Gabber for the inspiration for this paper and a number of enlightening conversations on its topic. We further thank Gebhard Böckle, Andy Jiang, Alicia Lamarche,  and Jacob Lurie for numerous valuable discussions.  Finally we thank Shiji Lyu and Eamon Quinlan-Gallego for valuable comments on a previous draft.
\begin{itemize}
\item{} {Bhatt was partially supported by the NSF (\#1801689, \#1952399, \#1840234), the Packard Foundation, and grants from the Simons Foundation (MPS-SIM-00622511, MPS-PERF-00001529-02).}
\item{} {Blickle was partially supported by DFG Grants SFB/TRR45, CRC326 and AEI-DFG \#541528446}
\item{} {Schwede was partially supported by NSF Grants \#1801849  \#2101800, \#2501903, and NSF FRG Grant \#1952522, Simons Foundation Travel Support for Mathematicians SFI-MPS-TSM00013051 and a fellowship from the Simons Foundation.}
\item{} {Tucker was partially supported by
NSF Grants \#2200716, \#2501904 and Simons Foundation Travel Support for Mathematicians SFI-MPS-TSM-00014083.}
\end{itemize}

\section{Preliminaries}
\label{sec.Preliminaries}

\subsection{Preliminaries on $p$-bases and differential bases}\label{pbasisdiffbasis}

Throughout this paper, if $R$ is a ring of characteristic $p > 0$, we use $R^p \subseteq R$ to denote the image of the Frobenius map $F \colon R \to R$.

\begin{defn}[$p$-basis]
    Suppose $R$ is a Noetherian ring and $A \subseteq R$ is a subring.  A subset $\Gamma \subseteq R$ is called \emph{$p$-independent} over $A$ if the set of monomials 
    \[
        \{ x_1^{b_1} \cdot \dots \cdot x_m^{b_m} \;|\; x_i \in \Gamma, x_i \neq x_j \text{ for $i \neq j$}, 0 \leq b_i \leq p-1 \}
    \]
    is linearly independent over the subring $A[R^p]$ of $R$.  It is called a \emph{$p$-generating set} over $A$ if $R = A[R^p,\Gamma]$.  A $p$-independent $p$-generating set is called a \emph{$p$-basis} over $A$. If $A \subseteq R^p$ we may drop ``over $A$'' in the notation.  
\end{defn}

\begin{defn}[Differential basis]
    Suppose $R$ is a Noetherian ring and $A \subseteq R$ is a subring.  A subset $\Gamma \subseteq R$ is called a \emph{differential basis} of $R$ over $A$ if $\{ dx \;|\; x \in \Gamma\}$ is a free $R$-basis of  the module of Kähler differentials $\Omega_{R/A}^1$.
\end{defn}

% By \cite[38.A, page 269]{Matsumura_CommutativeAlgebra} every $p$-basis is a differential basis.  Conversely, %\cite[Theorem 1]{TycDifferentialBasisPBasisSmoothness} or \cite{KimuraNiitsumaDiffBasisPBasisInRLR1984} 
% as we are Noetherian, every differential basis is also a $p$-basis by \cite[Proposition 58]{Andre.HomologieDeFrobenius}, see also \cite[Theorem 1]{TycDifferentialBasisPBasisSmoothness}, \cite{KimuraNiitsumaDiffBasisPBasisInRLR1984}, and \cite[Theorem 11.5]{MaPolstra.FSingularitiesBook}.

By \cite[38.A, page 269]{Matsumura_CommutativeAlgebra} every $p$-basis is a differential basis.  
Conversely, if $R$ is a Noetherian ring of characteristic $p$, then every differential basis of $R$ is also a $p$-basis by \cite[Proposition 58]{Andre.HomologieDeFrobenius}; see also \cite[Corollary 11.6]{MaPolstra.FSingularitiesBook}.
For historical background, see also \cite[Theorem 1]{TycDifferentialBasisPBasisSmoothness} and \cite{KimuraNiitsumaDiffBasisPBasisInRLR1984}; however, as explained in \cite[Remark 11.1 and the discussion preceding Theorem 11.5]{MaPolstra.FSingularitiesBook}, the argument in \cite{TycDifferentialBasisPBasisSmoothness} relies on \cite{FogartyKahlerDifferentialsAndHilberts14th}, whose proof requires additional Noetherian hypotheses.

Let $R$ be $F$-finite. We denote by $\Omega_R$ the absolute Kähler differentials $\Omega^1_{R/\bF_p}$. The cotangent sequence \cite[\href{https://stacks.math.columbia.edu/tag/00RS}{Tag 00RS}]{stacks-project} for the ring maps $\bF_p \to R^p \to R$
\[ 
    \Omega^1_{R^p/\bF_p} \otimes_{R^p} R \to \Omega^1_{R/\bF_p} \to \Omega^1_{R/R^p} \to 0
\] 
shows that $\Omega_R=\Omega^1_{R/\bF_p} \cong \Omega_{R/R^p}^1$, since the image of the first map is generated by $dr^p=pr^{p-1}dr=0$. As $R$ is $F$-finite, $\Omega_{R/R^p}^1$ is a finitely generated $R$-module, hence so is $\Omega_{R}$.

If $R$ is regular and $F$-finite, then the map $\mathbb{F}_p \to R$ is a regular morphism (i.e. flat with geometrically regular fibers) and Popescu's theorem \cite[\href{https://stacks.math.columbia.edu/tag/07GC}{Tag 07GC}]{stacks-project} asserts that $R$ can be written as a directed colimit of smooth $\mathbb{F}_p$ algebras $R_i$. Hence $\Omega_{R} = \colim \Omega_{R_i}$ is a colimit of flat modules and hence $\Omega_R$ is a finitely generated flat $R$-module, i.e. a locally free $R$-module of finite rank. The equivalence of the notions of differential basis and $p$-basis together with Kunz's theorem \cite{Kunz_regular} now imply:

\begin{lem}\label{lem.regularFfiniteHasLocalPbasis}
    A Noetherian $F$-finite ring is regular if and only if it is reduced and locally has a $p$-basis (equivalently a differential basis). 
\end{lem}
Without the $F$-finiteness assumption, the result is false, as the example of a power series ring in one variable over a non-$F$-finite field shows \cite[Example 3.8]{KimuraNiitsumaRegularLocalOfCharpAndPBasis1980}.

We point out that there are regular $F$-finite rings $R$ where $F_* R$ is not free and so in particular $R$ does not have a $p$-basis, as the following example shows.

\begin{example}
\label{exam.EllipticCurve}
Consider $R = \bF_2[x,y]/(y^2+xy+y+x^3+x+1)$ an affine chart on a smooth ordinary elliptic curve.  $R$ is then regular and $F$-finite and it is known that $F_* R \cong R \oplus Q$ where $Q$ is the prime ideal whose corresponding point has order $p=2$ in the elliptic curve group, see for instance \cite{SannaiTanaka-ACharacterizationOfOrdinaryAbelian} or \cite[Exercise 2.19]{PatakfalviSchwedeTucker}.  A direct computation shows that $Q = (x+1,y+1)$ is the point of order $2$.  It is easy to verify that $Q$ is not principal.  We also observe that $F_* R \cong R \oplus Q$ is not free since its determinant is also equal to $Q$, which is still not principal.  In particular, the regular ring $R$ does not have a $p$-basis.  Note we used Macaulay2 and Sage to verify some aspects of this particular example \cite{M2,sagemath}.  

Similar examples can also be done with elliptic curves in other characteristics, indeed a similar computation is explored in \cite[Example 5.2]{deStefaniPolstraYao.GlobalizingFInvariants}.
\end{example}

If $R$ is reduced, $F$-finite and has a $p$-basis $x_1,\ldots,x_n$ (i.e. $R$ is in particular regular). Then by the above $dx_1,\dots,dx_n$ is a basis of $\Omega^1_R$ and hence $dx_1\wedge \dots \wedge dx_n$ is a free generator of $\omega_R = \wedge^n\Omega^1_R$, hence $\omega_R \cong R$. We recall the following observation which can also be found in work of Baudin: 

\begin{proposition}[{{\itshape cf.} \cite[Lemma 5.1.14]{Baudin.DualityBetweenCartierCrystalsAndPerverseFPSheaves}}]
    \label{prop.FrobeniusUpperShriekPBasis}
    Suppose $R$ is an $F$-finite regular ring with a $p$-basis $x_1,\ldots,x_n$.  Then projection $x_1^{p-1}\cdot \ldots\cdot x_n^{p-1} \mapsto dx_1 \wedge \dots dx_n$ gives an $F_*R$-module isomorphism
    \[
        F_* F^!\omega_R = Hom_R(F_* R, \omega_R) \cong F_* \omega_R.
    \]
\end{proposition}
\begin{proof}
    From the $p$-basis $x_1, \dots, x_n$ we obtain the $R$-basis $F_* x_1^{a_1} \cdots x_n^{a_n}$ with $0 \leq a_i \leq p-1$ of $F_*R$. A free generator of the Hom-set is the map which sends the $F_* x^{p-1}_1\cdot\ldots\cdot x^{p-1}_n$ to $dx_1\wedge\ldots\wedge dx_n$ and the other basis elements to zero. The other projections can be obtained by pre-composing this one by multiplication by basis elements.%Clearly, via evaluation at $x_1^{p-1}\cdot \ldots\cdot x_n^{p-1}$ this is mapped to the free generator $dx_1\wedge\ldots\wedge dx_n$ of $\omega_R$.
\end{proof}
This immediately implies that schemes $f \colon X \to \Spec R$ of finite type over $\Spec R$, where $R$ is regular with $p$-basis, have a dualizing complex $\omega_X^{\mydot} = f^!\omega_R$ that satisfies $F^! \omega_X^{\mydot} \cong \omega_X^{\mydot}$, also see \cite[Corollary 5.1.15]{Baudin.DualityBetweenCartierCrystalsAndPerverseFPSheaves} and \cite[\href{https://stacks.math.columbia.edu/tag/0AU5}{Tag 0AU5}]{stacks-project}. One thing we show in this paper is that this dualizing complex on $X$ is independent of choices.
% This was also observed in recently in \cite[Lemma 5.1.14]{Baudin}] actually we were using since many years that the property that the dualizing complex is unit Cartier transports via essentially finite type maps.   

\subsection{The fundamental local isomorphism -- a key tool for duality}
\label{subsec.FunLocalIso}

Grothendieck duality is essentially the story of the existence of a functor $f^!$ which is the right adjoint of $\myR f_* $ for $f : Y \to X$ a proper morphism of Noetherian schemes.  For non-proper maps, one can define $f^!$ via compactification and restriction to an open set, but there are alternate highly explicit descriptions as well which actually play a key role in the most classical proofs of Grothendieck duality.  The descriptions of $f^!$ for smooth and finite maps of rings are recorded below as they will be important in both the modern and classical approaches to our theorem.  Additionally, we recall the \emph{fundamental local isomorphism} which gives another characterization in the case of an LCI morphism.

Suppose $f : R \to S$ is a finite type map of Noetherian rings.
When $f \colon R \to S$ is smooth (and hence separable) of relative dimension $d$, we have
\[
f^!(\_) \simeq f^\sharp(\_) = \bigwedge^d \Omega_f[d] \otimes_{
S} \myL f^*(\_).
\]

For $f \colon R \to S$ finite, one has
\[
f^!(\_) \simeq f^\flat(\_) := \RHom_{R}(S, \_).
\]
Further assuming that $R \to S$ is an lci surjection with kernel $J$ of codimension $c$, the \emph{fundamental local isomorphism}, recalled immediately below, yields a functorial isomorphism 
\begin{equation}
\label{eq:etadef}
    \eta \colon f^\flat(\_) \to[\cong] \bigwedge^c(J/J^2)^\vee[-c] \otimes_S (S \otimes^{\myL}_R \_ )
\end{equation}
cf. \cite[III, Proposition 7.2]{HartshorneResidues}, \cite[I, Theorem 4.5]{AltmanKleimanGD},  \cite[Equation (2.5.2)]{ConradGDualityAndBaseChange}. We recall the construction of this isomorphism and its proof: 

%\todo[inline]{{\textbf{Karl:}}  This is also used in Bhargav's section.  I will try to move it.}

\begin{lemma}[Fundamental local isomorphism]
\label{dualregular}
If $J$ is a regular ideal of codimension $c$ in a commutative ring $S$, then for any $M \in D(S)$, there is a canonical isomorphism
\[ 
    \eta \colon \RHom_S(S/J,M) \to[\simeq] \bigwedge^c(J/J^2)^\vee[-c] \otimes_{S/J} (S/J \otimes_S^{\myL} M)
\]
in $D(S/J)$ which is functorial in $M$.
\end{lemma}
\begin{proof}
%When regarded as functors of $M \in D(S)$, both sides of the above expression are colimit preserving $S$-linear functors $D(S) \to D(S/J)$: this is clear for the RHS, and follows for the LHS as $S/J \in \Perf(S)$ by regularity of $J \subseteq S$. As any such functor is uniquely determined by its value on $M=S$, it
As $S/J$ is perfect, by tensoring with $M$, it suffices to prove the claim when $M = S$, that is
\[ \RHom_S(S/J,S) \simeq \bigwedge^c(J/J^2)^\vee[-c]\]
in $D(S/J)$. The vanishing in degrees $\neq c$ can be checked locally using a Koszul calculation. The same calculation also shows that the canonical map $\mathrm{Ext}^c_S(S/J,S) \to \mathrm{Ext}^c_S(S/J,S/J)$ is an isomorphism. The identification in top degree then results  from the sequence of natural isomorphisms 
\[ \mathrm{Ext}^c_S(S/J,S) \simeq \mathrm{Ext}^c_S(S/J,S/J) \simeq \mathrm{Tor}_c^S(S/J,S/J)^\vee \simeq \bigwedge^c \mathrm{Tor}_1^S(S/J,S/J)^\vee \simeq \bigwedge^c(J/J^2)^\vee\]
of finite projective $S/J$-modules.
\end{proof}

\section{Gabber's construction}
\label{sec.CanDualizingClassical}

Gabber proved that if $R$ is an $F$-finite ring with a $p$-generating set $x = x_1, \dots, x_n$ (that is, $R = R^p[x_1, \dots, x_n]$), then there exists a regular ring $S$ surjecting onto $R$ such that $S$ admits a $p$-basis mapping onto $x$.

In this section we review this construction and establish several additional properties that will be used later. These results will also play a role in our discussion of the classical approach to dualizing complexes for $F$-finite schemes in \autoref{sec.ClassicalDuality}.

%On a regular $F$-finite scheme the sheaf of Kähler differentials $\Omega_X^1$ is finitely generated and locally free. We take its top exterior power 
%\[ \omegacan{X}= \bigwedge^n \Omega_X^1[n]
%\]
%with $n = \operatorname{rank} \Omega_X^1$ placed in appropriate homological degree as our canonical dualizing complex. We observe that the explicit duality in terms of differential forms of \cite{HartshorneResidues,ConradGDualityAndBaseChange} applies in this context to show that these dualizing complexes for $F$-finite regular schemes are compatible with $f^!$-pullback for finite type maps as well as with localization and completion. 

%A construction of Gabber \cite{Gabber_someTstructures} shows that any $F$-finite Noetherian ring $R$ is the quotient of a regular $F$-finite ring $S$. If $\pi: S \to R$ denotes the quotient map we define 
%\[
%\omegacan{R} = \pi^! \omegacan{S}.
%\]
%By carefully analyzing Gabber's construction we show that this definition is independent of the chosen presentation. This does not follow from classical duality as the comparison map involved is not of finite type -- however it can be reduced to the case of a completion of a regular ring along some ideal. This is our key technical observation, see \autoref{lem.UniversalGabberViaCompletion}.

%Once one has a canonical dualizing complexes on Noetherian $F$-finite affine schemes it is an immediate consequence of the gluing Lemma \cite{BBD} that these local dualizing complexes glue to a unique global dualizing complex which is compatible with $( )^!$-pullback for all finite type maps. 

\subsection{A review of Gabber's Construction}
\label{sec:gabberconstr}

In this mostly expository subsection, we review a construction of Gabber \cite[Remark 13.6]{Gabber_someTstructures} which assigns to any $F$-finite ring $R$ with a tuple $x_1,\ldots,x_n$ of $p$-generators a regular $F$-finite ring $R_\infty$ with a $p$-basis and a surjection $R_\infty \to R$. By tracking the construction carefully, we obtain a more elementary presentation and also gain better control over the dependence on the initial choice of $p$-generators.

Let $R$ be an $F$-finite Noetherian ring of characteristic $p > 0$. We continue to let $R^p$ denote the image $F_R(R)$ of the Frobenius endomorphism $F_R \colon R \to R$.

We have two ring homomorphisms
\begin{equation*}
    \iota \colon R^p \into[r^p \mapsto r^p] R \qquad \mbox{ and } \qquad \phi \colon R \onto[r \mapsto r^p] R^p
\end{equation*}
where $\iota$ is given by the inclusion of $R^p \subseteq R$ and $\phi$ is induced by the Frobenius on $R$. By construction we have $\iota \circ \phi = F_R$, but observe that the same holds for the opposite composition as well $\phi \circ \iota = F_{R^p}$. Moreover, since $R$ is $F$-finite, the map $\iota$ is finite type.% (equivalently, $\phi$ is finite, since it is integral).

Gabber's construction consists of iterating and taking the limit over the following process:
\begin{const}\label{gabbify_first_step}
Let $(R;x)$ be an $F$-finite ring with $x=x_1,\ldots,x_n$ an $n$-tuple of elements of $R$ which is a $p$-generating set of $R$, i.e. $R^p[x] = R$. Define
\[
    R'= R[X_1,\ldots,X_n]/(X_i^p-x_i | i \in \{1,\ldots,n\})
\]
and maps 
\[
    \iota\colon R \into[r \mapsto r] R' \text{ and } \phi\colon R' \onto[rX_i \mapsto r^px_i] R.
\]
Then the following properties are immediate to verify:
\begin{enumerate}
\item $\iota$ is injective and $\phi$ is surjective.
\item $\iota\circ \phi = F_{R'}$ and $\phi \circ \iota = F_R$. In particular, $R'^{p} = \iota(R)$.
\item Write $x_1' =\overline{X_1}, \ldots ,x_n'=\overline{X_n} \in R'$.  Then $x' = x_1', \dots, x_n'$ form a $p$-basis of $R'$ which is mapped bijectively to the $p$-generating set $x_1,\ldots,x_n$ via $\phi$.
\end{enumerate}
\end{const}

%{\color{red}
%\begin{remark}  What happens if we don't choose a $p$-generating set, but instead a finite subset of $R$?
%      Let $(R;x)$ be an $F$-finite ring with $x=x_1,\ldots,x_n$ an $n$-tuple of elements of $R$.
%    % which is a $p$-generating set of $R$, i.e. $R^p[x] = R$. 
%    We may still define
%    \[
%        R'= R[X_1,\ldots,X_n]/(X_i^p-x_i | i \in \{1,\ldots,n\})
%    \]
%    and maps 
%    \[
%        \iota\colon R \into[r \mapsto r] R' \text{ and } \phi\colon R' \onto[rX_i \mapsto r^px_i] R^p[x_1, \dots, x_n] \hookrightarrow R.
%    \]
%   Then the following properties are immediate to verify:
%    \begin{enumerate}
%    \item $\iota$ is injective %and $\phi$ is surjective.
%    \item $\iota\circ \phi = F_{R'}$ and $\phi \circ \iota = F_R$. In particular, $R'^{p} = \iota(R)$.
%    \item Write $x_1' =\overline{X_1}, \ldots ,x_n'=\overline{X_n} \in R'$.  Then $x' = x_1', \dots, x_n'$ form a $p$-independent subset of $R'$ (even $p$-independent over $R$) which is mapped bijectively to $x_1,\ldots,x_n$ via $\phi$.
%    \end{enumerate}
%\end{remark}
%}

Starting from an $F$-finite ring with $p$-generating set $(R;x)$, this construction produces an $F$-finite ring with $p$-basis $(R';x')$, together with a surjection $\phi\colon R' \to R$ identifying the $p$-basis $x'$ of $R'$ with the given $p$-generators $x$ of $R$, and an injection $\iota\colon R \to R'$ whose image is $R'^p$. This characterization determines $(R',x')$ uniquely, although $R'$ is generally not reduced and hence not regular.

We make this precise in the following lemma.% which also provides the key technical statement needed later:
\begin{lem}\label{gabbify_first_step_keyFacts}
   Let $(R;x)$ be an $F$-finite ring with a $p$-generating set $x=x_1,\ldots,x_n$. 
    \begin{enumerate}
    \item $R$ is reduced and $x_1,\ldots,x_n$ is a $p$-basis of $R$ if and only if $\phi\colon R' \to R$ is an isomorphism (which is equivalent to $R'$ being reduced).
    \item If we have a further map
        \[
            (S;y) \to[\psi] (R';x') \to[\phi] (R;x)
        \]
    such that $y$ is a $p$-generating set of $S$ and $\psi$ maps $y$ bijectively to $x'$ then 
    $\ker\psi = \ker (\phi \circ \psi)^{[p]}$.
    \end{enumerate}
\end{lem}
\begin{proof}
(a) We use multi-index notation such that $x^i$ is short for $x_1^{i_1}x_2^{i_2}\cdots x_n^{i_n}$ with $0 \leq i \leq p-1$ meaning that $0 \leq i_j \leq p-1$ holds for all $i_j$ in the entries of the tuple $i=(i_1,\ldots,i_n)$. Let $a = \sum_{0 \leq i \leq p-1} \iota(r_{i})x'^{i}$ be an element of $\ker \phi$ written (uniquely) in the $p$-basis $x_1',\ldots,x_n'$ with $r_{i} \in R$. 
Then $0 = \phi(a)=\sum_{0 \leq i \leq p-1} r_{i}^px^{i}$ as $\phi \circ \iota = F_R$ and $\phi(x')=x$. If $x$ is a $p$-basis of $R$ we may conclude that all coefficients $r_{i}^p$ are also zero. If $R$ is reduced then $F_R$ is injective and it follows that all $r_{i}$ are zero, and hence $a=0$. This shows that $\phi$ is injective. Since $\phi$ is always surjective it follows that $\phi$ is an isomorphism.

Conversely, if $\phi$ is an isomorphism, then $F_R=\iota \circ \phi$ is injective. Hence $R$ is reduced and the induced map $R \to R^p$ is a ring isomorphism. Thus the $p$-independence of $x'$ over ${R'}^p=\iota(R)$ in $R'$ implies the $p$-independence of $x$ over $R^p$ in $R$, by construction of $\phi\colon R' \to R$.

(b) We first show that $(\ker (\phi \circ \psi))^{[p]} \subseteq \ker \psi$. Since $F_{R'} = \iota \circ \phi$ and $\iota$ is injective, we have
\[
    \ker(\phi \circ \psi)=\ker(\iota \circ \phi \circ \psi) = \ker(F_{R'} \circ \psi) = \ker(\psi \circ F_{S}),
\]
hence $\ker(\phi \circ \psi)^{[p]} = \ker(\psi \circ F_S)^{[p]} \subseteq \ker\psi$. Conversely, let $a \in \ker \psi$. As $y$ is a $p$-generating set of $S$, we may write $a = \sum_{0 \leq {i} \leq {p-1}} s_{{i}}^p {y}^{{i}}$ for suitable $s_i \in S$. Applying $\psi$ one obtains 
\[
    0 = \psi(a) = \sum_{0 \leq i \leq p-1} \psi(s_{{i}})^p {\psi(y)}^{{i}} = \sum_{0 \leq i \leq p-1} (\iota\circ\phi\circ\psi)(s_{{i}}) {x'}^{{i}}
\]
using $F_{R'}=\iota \circ \phi$ by \autoref{gabbify_first_step} (b), and the assumption that $\psi$ maps $y$ to $x'$ bijectively. Also by \autoref{gabbify_first_step} (c), the $x'$ form a $p$-basis of $R'$ so we conclude that all $(\iota\circ\phi\circ\psi)(s_{{i}})$ are zero. Since $\iota$ is injective it follows that $s_{{i}} \in \ker(\phi\circ\psi)$. But now $a=\sum_{0 \leq i \leq p-1} s_{{i}}^p {y}^{{i}} \in \ker(\phi \circ \psi)^{[p]}$ and the result follows.
\end{proof}

Iterating \autoref{gabbify_first_step} and taking the limit one obtains Gabber's construction of a reduced $F$-finite ring $R_\infty$ with a $p$-basis surjecting onto $R$. 
\begin{const}\label{gabbify}
    Let $(R;x)$ be an $F$-finite ring with $p$-generating set $x$. Set $R_0 \coloneqq R$ and $R_i \coloneqq R_{i-1}'$ with maps $\phi\colon R_i \to R_{i-1}$ from the preceding construction. Define:
    \[
        G(R;x) \coloneqq \varprojlim_\phi R_i = \varprojlim \big( \dots \to[\phi] R_i \to[\phi] R_{i-1} \to[\phi] \dots \to[\phi] R_1 \to[\phi] R \big) 
    \]
    Denote the natural maps $G(R;x) \to R_i$ by $\phi_i$ and the kernel of $\phi_0 \colon G(R;x) \to R$ by $I$. Then the following properties hold:
    \begin{enumerate}
        \item The system of maps $\phi \colon R_i \to R_{i-1}$ from \autoref{gabbify_first_step}  
        defines a morphism of inverse systems inducing the identity on $R_\infty = \varprojlim_\phi R_i$. It follows that the maps $\iota \colon R_i \to R_{i+1}$, viewed as a map between limit systems, induce the Frobenius on $R_\infty$. In particular, since the limit of injective maps is injective, the Frobenius $F_{R_\infty}$ is injective, and hence $G(R;x)$ is reduced. 
        \item For $i \geq 1$, the $p$-bases of the rings $R_i$ are compatible with the maps in the inverse system by \autoref{gabbify_first_step}, and hence determine a $p$-basis $x_\infty$ of $G(R;x)$. We suppress this $p$-basis from notation since it is in natural bijection with the $p$-generating set $x$ of $R$. 
        \item For all $i$ we have $\ker \phi_i = I^{[p^i]}$.
        \item $G(R;x) \cong \varprojlim G(R;x)/I^{[p^i]}$. In particular, $\bigcap I^{[p^i]} = 0$. 
        \item The ideal $I$ is finitely generated.
        \item $G(R;x)$ is a Noetherian, $I$-adically complete reduced ring with $p$-basis. Hence $G(R;x)$ is a regular ring. 
    \end{enumerate}
\end{const}
\begin{proof}
(a) and (b) are immediate. For (c), the case $i=0$ is tautological. Applying \autoref{gabbify_first_step_keyFacts} (b) to $R_\infty \to R_i \to R_{i-1}$ shows inductively that $\ker \phi_i = \ker(\phi_{i-1})^{[p]}=(I^{[p^{i-1}]})^{[p]}=I^{[p^i]}$. As $R_i = R_\infty/\ker\phi_i$ (d) is a direct consequence of (c). 

For (e) choose elements $g_1,\ldots,g_m \in I$ such that their images $\phi_1(g_j)$ generate $\phi_1(I) = \ker(\phi\colon R_1 \to R) \subseteq R_1$. This is possible since $\phi_1$ is surjective. Clearly, $I=(g_1,\ldots,g_m)+\ker \phi_1$. By part (c) $\ker \phi_1 = I^{[p]}$ so we conclude $I=(g_1,\ldots,g_m)+I^{[p]}$. Applying Frobenius powers to this equation yields
\begin{equation}\label{eq:IcontainedIp}
    I^{[p^e]} = (g_1^{p^e},\ldots,g_m^{p^e})+I^{[p^{e+1}]} \text{ for all $e \geq 0$.}
\end{equation}
Let $x \in I$ be arbitrary. We want to show that $x$ can be expressed as a combination of the elements $g_1,\ldots,g_m$. The equality $I=(g_1,\ldots,g_m)+I^{[p]}$ allows us to find $w_{1,0},\ldots,w_{m,0} \in R_\infty$ such that 
\[
    x - \sum_{i=1}^m w_{i,0}g_i \in I^{[p]}.
\]
As by \autoref{eq:IcontainedIp} $I^{[p]} = (g_1^p,\ldots,g_m^p)+I^{[p^2]}$ we find $w_{1,1},\ldots,w_{m,1} \in R_\infty$ such that
\[
    x - \sum_{i=1}^m w_{i,0}g_i - \sum_{i=1}^m w_{i,1}g_i^p \in I^{[p^2]}.
\]
Proceeding inductively suppose $w_{1,k},\ldots,w_{m,k} \in R_\infty$ for $0 \leq k \leq e-1$ have been chosen such that $x - \sum_{k=0}^{e-1} \sum_{i=1}^m w_{i,k}g_i^{p^k} \in I^{[p^{e}]}$. Then \autoref{eq:IcontainedIp} shows we may find $w_{1,e},\ldots,w_{m,e} \in R_\infty$ such that:
\[
    x - \sum_{k=0}^{e} \sum_{i=1}^m w_{i,k}g_i^{p^k} \in I^{[p^{e+1}]}
\]
By description of $G(R;x) = \varprojlim R_{\infty}/I^{[p^e]}$ in part (d) it follows that 
\[
    x = \sum_{k=0}^{\infty} \sum_{i=1}^m w_{i,k}g_i^{p^k} = g_1(\sum_{k=0}^\infty w_{1,k}g_1^{p^k-1})+\ldots+g_m(\sum_{k=0}^\infty w_{m,k}g_m^{p^k-1})
\]
Each of the series $b_i \coloneqq \sum_{k=0}^\infty w_{i,k}g_i^{p^k-1}$ converges in $G(R;x) = \varprojlim G(R;x)/I^{[p^e]}$ as $g_i \in I$ and $p^k-1 \geq p^{k-1}$ for $k \geq 1$. This shows that $x = b_1g_1 + \ldots + b_mg_m$ lies in $(g_1,\ldots, g_m)$ concluding our argument. 

Finally, since $I$ is finitely generated, its Frobenius powers $I^{[p^e]}$ are necessarily cofinal with its ordinary powers $I^n$. Hence $G(R;x)$ is $I$-adically complete. Since $G(R;x)/I = R$ is Noetherian, it follows \cite[\href{https://stacks.math.columbia.edu/tag/05GH}{Tag 05GH}]{stacks-project} that $G(R;x)$ is Noetherian as well. Since $G(R;x)$ is reduced (a) and has a $p$-basis (b) it follows by the result of Kunz \cite[\href{https://stacks.math.columbia.edu/tag/0EC0}{Tag 0EC0}]{stacks-project}, \cite{Kunz_regular} that $G(R;x)$ is regular.  
\end{proof}

%As an immediate consequence to this construction one obtains that every $F$-finite ring is a quotient of a regular $F$-finite ring and hence has a dualizing complex. 

%\begin{cor}[{\emph{c.f.} \cite[Corollary 5.1.1]{Baudin.DualityBetweenCartierCrystalsAndPerverseFPSheaves}}]
 %   \label{prop.DualizingComplexAffine}
  %  Let $R$ be a Noetherian and $F$-finite ring. Then $R$ has a dualizing complex $\omega^\bullet_{R}$ which is equipped with a quasi-isomorphism 
   % \[
    %    \kappa\colon \omega^\bullet_{R} \to[\cong] F^! \omega^\bullet_{R} = \RHom_R(F_*R, \omegacan{R})
%    \]
%\end{cor}
%\begin{proof}
 %   By \autoref{gabbify} let $S=G(R;x)$ be a regular $F$-finite ring with $p$-basis with a surjection $\iota: S \onto R$. Now apply \autoref{prop.FrobeniusUpperShriekPBasis} and the discussion immediately after.
  %\end{proof}

%The statement of this corollary, the existence of a dualizing complex equipped with such a quasi-isomorphism, is a standard assumption in many papers on birational geometry in positive characteristic. In order to globalize this result we need to show that this construction of the dualizing complex $\omega_R^\bullet$ can be made unique, i.e. independent of the surjection $S \to R$.  We also would like the  dualizing complex to be compatible with arbitrary finite type maps. 

%In order to tackle this in the following subsections w

\subsection{Functoriality and compatibility of Gabber's construction}
\label{subsection:functorialityofGabber}
We now record some functoriality properties of Gabber's construction.

\begin{defn}
Fix a prime $p$. Let $\Ffinites$ be the category of pairs $(R;x)$ where $R$ is an $F$-finite Noetherian ring and $x=(x_1,\ldots, x_n)$ is a finite tuple of $p$-generators of $R$, i.e. a tuple exhibiting the $F$-finiteness of $R$. A morphism 
\[
\phi: (R;x) \to (S;y)
\]
is a ring map $\phi\colon R \to S$ and a map of tuples $\phi'\colon x \to y$ such that $\phi'(x_i)$ coincides with $\phi(x_i)$ as elements of $S$ (warning, there may be repeated entries). Let $\RegFfinites$  denote the full subcategory consisting of those pairs $(R;x)$ such that $R$ is regular and $x$ is a $p$-basis of $R$. 
\end{defn}

\begin{prop}\label{thm.gabbifyIsRightAdjointCompletion}
The assignment $(R,x) \mapsto G(R;x) \big(= (G(R;x),x_\infty) \;\big)$ of \autoref{gabbify} is a functor 
\[
    G\colon \Ffinites \to \RegFfinites
\]
which is right adjoint to the inclusion $\iota \colon \RegFfinites \subseteq \Ffinites$. The unit of adjunction $\epsilon\colon \id \to G\circ\iota$ is an isomorphism and the counit $\delta\colon \iota \circ G \to \id$ is a surjection. In particular, $G$ is essentially surjective. 
\end{prop}
\begin{proof}
The key assertion of \autoref{gabbify} is that $G$ indeed lands in $\RegFfinites$. If $\psi \colon (R;x) \to (S;y)$ is a morphism in $\Ffinites$, then \autoref{gabbify_first_step} yields a commutative diagram
\[
\xymatrix{
(R;x) \ar[r]^\psi & (S;y) \\
(R_1;X) \ar[u]^{\phi_R}\ar[r]^{\psi_1} & (S_1;Y) \ar[u]_{\phi_S}
}
\]
where
\[
\psi_1 \colon R_1=\frac{R[X_1,\ldots,X_n]}{(X_i^p - x_i| i\in 1,\ldots n )} \to \frac{S[Y_1,\ldots,Y_m]}{(Y_i^p - y_i| i\in 1,\ldots m )}=S_1
\]
is the unique map which is $\psi$ on the coefficients and maps the class of $X_i$ to the class of $Y_j$ if $\psi'(x_i)=y_j$ in the map of tuples. Since the $x$'s and $X$'s are in bijection this is also the only map which lifts $\psi$ to a map $(R_1;X) \to (S_1;Y)$ in $\Ffinites$. By iterating and taking the limit as in \autoref{gabbify} we obtain an induced map $\psi_\infty \colon G(R;x) \to G(S;y)$ which makes the diagram
\[
\xymatrix{
(R;x) \ar[r]^\psi & (S;y) \\
G(R;x) \ar[u]^{\delta}\ar[r]^{\psi_\infty} & G(S;y) \ar[u]_{\delta}
}
\]
commute. Again, since the vertical maps induce bijections on the level of $p$-generating sets, the map $G(R;x) \to G(S;y)$ is unique. It follows that the construction is compatible with composition.

The counit of adjunction $\delta \colon \iota \circ G \to \id$ is given by the surjection $\phi_0\colon G(R;x) \to R$ of \autoref{gabbify}, which identifies the $p$-basis $x_\infty$ of $G(R;x)$ with the given $p$-generators $x$ of $R$. By \autoref{gabbify_first_step_keyFacts} (a), the restriction of the counit to $\RegFfinites$,
\(
    \delta\iota \colon \iota \circ G \circ \iota \to \iota,
\)
is an isomorphism. Since $\iota$ is fully faithful, this determines the unit
\(
    \epsilon \colon \id \to G \circ \iota
\)
as the corresponding inverse map. The required adjunction identities are then immediate.
\end{proof}

The following proposition computes the adjoint map of $(S;x) \to (R;y)$ in the case that the ring map is surjective and induces a bijection on the level of the $p$-generators.

\begin{prop}\label{lem.UniversalGabberViaCompletion}
    Suppose $S$ is a regular ring with a $p$-basis $x = x_1, \ldots, x_n$ and $\pi \colon S \onto R := S / J$ is a surjective ring map. Then the $y_i=\pi(x_i)$ are a $p$-generating set for $R$. The surjection $(S;x) \to (R;y)$ factors uniquely through 
    \[
       (S;x) \to[\pi_\infty] G(R;y) \to[\phi_0] R
    \]
    and the map $\pi_\infty \colon (S;x) \to G(R;y)$ is just the completion of $S$ along $J$.  In particular, $G(R;y) \cong \widehat{S}^J$.
\end{prop}

\begin{proof}
    The argument is essentially the same as in \autoref{gabbify} (d): Consider the commutative diagram
    \[
    \xymatrix{
     &&&&&& S \ar[d]^{\pi} \ar[dl]^{\pi_1}\ar[dlll]^{\pi_i} \ar[dllllll]_{\pi_\infty} \\
     G(R;y)  \ar[rrr]_{\phi_i} & & &  R_i\ar[r]_{\phi}  & \ldots  \ar[r]_\phi & R_1  \ar[r]_{\phi} & R
    }
    \]
    where the bottom row is \autoref{gabbify} for $(R;{y})$ noting that the system for $(S;{x})$ is constant due to \autoref{gabbify_first_step_keyFacts} (a).  By construction all maps $\pi_i$ are surjective, but $\pi_\infty$ need not be. We claim that $\ker \pi_i = (\ker \pi)^{[p^i]}=J^{[p^i]}$ for all $i \geq 0$. For $i=0$, this is clear. Proceeding by induction on $i$, we apply \autoref{gabbify_first_step_keyFacts} (b) to the maps $\pi_{i-1} \colon S \to[\pi_i] R_i \to[\phi] R_{i-1}$ to conclude
    \[
        \ker \pi_i = (\ker \pi_{i-1})^{[p]} = (J^{[p^{i-1}]})^{[p]} = J^{[p^i]}.
    \]
    Therefore 
    \[
        G(R;y) = \varprojlim R_i = \varprojlim S/\ker \pi_i = \varprojlim S/J^{[p^i]} = S^{\wedge_J}
    \]
    as $J$ is finitely generated since $S$ is Noetherian. 
\end{proof}

%\subsubsection*{Additional notes on and consequences of Gabber's construction}

We close this section with several compatibilities and examples of Gabber's construction that follow directly from \autoref{lem.UniversalGabberViaCompletion}.
 Although not strictly necessary for the rest of the paper, they are nevertheless useful for gaining a better understanding of the construction.

%When proving the above compatibilities, we noted a number of statements and computed several examples which were not strictly necessary in our above work, but perhaps are useful at least in terms of gaining intuition about Gabber's construction (they were for us).  We begin with two corollaries of \autoref{lem.UniversalGabberViaCompletion}.

%\todo[inline]{\emph{Karl:} Please read over the next statement carefully.}

\begin{cor}[Compatibility with completion]
    \label{cor.GabberConstructionCompletedAndCompleted}
    Suppose $R$ is an $F$-finite Noetherian ring with $p$-generating set ${x}=x_1,\ldots,x_n$ and induced $S := G(R;{x}) \xrightarrow{\pi_{\infty}^R} R$.  Suppose $J \subseteq R$ is an ideal with preimage $J' \subseteq G(R;\underline{x})$.  Then we have a commutative diagram
    \[
        \xymatrix{
            G(R; {x}) = S \ar[d]_{\pi_{\infty}^R} \ar[r]^-{\widehat{(-)}^{J'}} & \widehat{S}^{J'} \ar[r]^-{\sim} & G(\widehat{R}^J; {x}) \ar[d]^{\pi_{\infty}^{\widehat{R}^J}}\\
            R \ar[rr]_{\widehat{(-)}^{J}} & & \widehat{R}^J.
        }
    \]
    In particular, $G(R;{x})\,\widehat{\;}^{{}_{J'}} \cong G(\widehat{R}^J; \underline{x})$.
\end{cor}
\begin{proof}
    There is a surjection $\pi = (\pi_{\infty}^R)\,\widehat{\;}^{{}_{J'}}: \widehat{S}^{J'} \to \widehat{R}^J$ and the source is a regular ring with a $p$-basis mapping bijectively to ${x}$.  Note also that $J' \widehat{S}^{J'} \supseteq \ker(\pi) = (\ker(\pi_{\infty}^R))\,\widehat{\;}^{{}_{J'}}$ and so the source is already $\ker \pi$-adically complete (\cite[\href{https://stacks.math.columbia.edu/tag/090T}{Tag 090T}]{stacks-project}).  Hence we can apply \autoref{lem.UniversalGabberViaCompletion} and the result follows.
\end{proof}

\begin{cor}[Compatibility with localization]
\label{cor.GabberConstructionAndLocalization}
    With notation as above, suppose $W \subseteq R$ is a multiplicative set and $V \subseteq G(R; x)$ is a multiplicative subset surjecting onto $W$ by $\pi_{\infty} \colon G(R; x) \to R$.  Let $J = \ker \pi_{\infty}$.  We write $\frac{x}{1}$ to be the induced $p$-generating set of $W^{-1}R$.  Then we have a canonical isomorphism 
    \[
        G(W^{-1} R, \frac{x}{1}) \cong (V^{-1} G(R; x)){}^{\wedge_{V^{-1}J}}.
    \]
\end{cor}
\begin{proof}
    Apply \autoref{lem.UniversalGabberViaCompletion} to the map $V^{-1}G(R;x) \to (W^{-1}R;\frac{x}{1})$.
\end{proof}

\begin{cor}[Extending the $p$-generating set]
\label{cor.ExtendingThePGeneratingSet}
Suppose $R$ is an $F$-finite Noetherian ring with $p$-generating set $x = x_1, \dots, x_n$.  Suppose $y = y_1, \dots, y_m \in R$ are a sequence of additional elements.  Then we have a (non-canonical) isomorphism
\[
G(R; x , y) \cong G(R; x)\llbracket t_1, \dots, t_m \rrbracket
\]
so that the following diagram commutes:
\[
\xymatrix@C=60pt@R=30pt{
    G(R; x , y) \ar@{<->}[d]_{\sim} \ar@{->>}[r]^-{\pi_{\infty}^{x , y}} & R\\
    G(R; x)\llbracket t_1, \dots, t_m \rrbracket \ar@{->>}[ur]_-{\pi_{\infty}^{x}, t_i \mapsto 0}
}
\]
\end{cor}
\begin{proof}
Consider the regular ring $S = G(R; x)[Y_1, \dots, Y_m]$ with surjection to $R$,
\[
    S = G(R; x)[Y_1, \dots, Y_m] \to[\nu] R
\]
given by $\pi_{\infty}$ on $G(R;x)$ and so that $Y_i \mapsto y_i$.  Write $J = \ker \nu$ and note it contains $K = \ker \pi_{\infty}$.
By \autoref{lem.UniversalGabberViaCompletion}, we have an isomorphism 
\[
    G(R; x , y) \cong \widehat{S}^J.
\]
Now, for each $y_i \in R$, choose a lift $g_i \in G(R; x)$ and consider the elements $t_i = Y_i - g_i$.  We then have an equality
\[
S = G(R; x)[t_1, \dots, t_m]
\]
which is simply a change of coordinates and it follows that $J = \ker{\nu} = (t_1, \dots, t_m)S + KS$. The completion $\widehat{S}^J$ is then seen to be $G(R; x)\llbracket t_1, \dots, t_m\rrbracket$ and the result follows.
\end{proof}

We complete this section with a number of explicit examples of these instances.

\begin{example}
Let $(R;x)$ be an $F$-finite ring with $p$-generators $x$.
\begin{enumerate}
\item For $t \in \bF_p$ let $\psi\colon  (\bF_p[X], X) \to (\bF_p,t)$ given by sending $X$ to $t$. Then by \autoref{thm.gabbifyIsRightAdjointCompletion} we have $G(\bF_p,t)$ is the $(X-t)$-adic completion of $\bF_p[X]$. This shows that $G(\bF_p,t)$ is just equal to a power series ring $\bF_p\llbracket Y \rrbracket$ in the variable $Y = X-t$.
\item Let $Y=Y_1,\ldots,Y_m$ be some variables. Then $x , Y$ is a set of $p$-generators for the polynomial ring $R[Y]$. The natural map $G(R;x)[Y] \to G(R[Y];x , Y)$ is then the completion of $G(R;x)[Y]$ along the kernel of $G(R;x)[Y] \to R[Y]$. %\todo[inline]{check, and in case it is correct write down proof.  \\ \emph{Karl:} Ok, $S[Y] = G(R;x)[Y]$ is a regular ring with a surjection onto $R[Y]$.  However, I'm not sure I see that I see that $S[Y]$ is complete with respect to $\ker(S[Y] \to R[Y]) = (\ker(S \to R) S[Y])$.  For instance, if $R = \mathbb{F}_p$ and $x = 0$, then $G(R; x) = \mathbb{F}_p\llbracket X \rrbracket$.  Now, I think $G(R[Y], (x,Y)) = \bF_p[Y]\llbracket x \rrbracket$ which is not equal to $\bF_p\llbracket x \rrbracket[Y]$.}
\item Let $Y=Y_1,\ldots,Y_m$ be some variables. Then  $x_\infty , Y$ is a $p$-basis of the regular Noetherian ring  $G(R;x)[Y]$. Applying \autoref{thm.gabbifyIsRightAdjointCompletion} to any $G(R;x)$-linear surjection $\pi \colon G(R;x)[Y] \to R$ where $x_\infty \mapsto x$ and $Y \mapsto y$ for some $y \in R$ shows that $G(R;x , Y)$ is just the completion of the regular ring $G(R;x)[Y]$ along the ideal $J= \ker \pi$. That is, the map $G(( R;x) \to G(R; x , Y))$ factors as the composition of a polynomial extension $G(R;x) \to G(R;x)[Y]$ followed by the completion along an ideal $G(R;x)[Y] \to G(R; x , Y)$.
\end{enumerate}
\end{example}

\section{Dualizing complexes as units of the $\stens$-monoidal structure}
\label{sec.DualizingComplexesAsUnits}

In this section, we explain how to characterize a canonical dualizing complex intrinsically\footnote{The existence of surjections $R \to S$ is still critical to the arguments, but now only in checking the properties.}: it is the unit of a  symmetric monoidal structure on $D^+_{coh}(S)$ called the $\stens$-product. The corresponding uniqueness will guarantee the existence of a satisfactory theory of dualizing complexes on arbitrary $F$-finite Noetherian schemes.

Technically, the main new observation is that there is a good notion of the completion of the (typically non-Noetherian) ring $S \otimes_{\mathbb{F}_p} S$ along the (typically not finitely generated) kernel of the diagonal map $S \otimes_{\mathbb{F}_p} S \to S$; we used derived algebraic geometry to construct this completion, but it turns out to be a classical Noetherian ring.  After developing this theory, in \S \ref{DualCompStensClassical} we will explain the format of the argument for the $\stens$-product in the classical setting of finite type algebras over a field where there are no complications involving the diagonal.  Then, in \S \ref{ss:DualDiagFFin}, we will carry out the proof of our main result, showing that the canonical dualizing complex from \S \ref{ss:CanDualFFin} provides a unit for this symmetric monoidal structure (by imitating the arguments in the finite type case treated in \S \ref{DualCompStensClassical}).

\subsection{Completions revisited}
\label{sec:CompRev}

Our goal in this section to describe a version of the classical $I$-adic completion operation for Noetherian rings that works well for arbitrary rings without assuming the ideal is finitely generated. To make the arguments flow, we formulate statements in the generality of animated rings, so our operations (e.g., tensor products) are understood to be derived.  Our arguments are inspired by the characteristic $0$ story in \cite{BhattCompddR}. The starting point is the following convergence result, essentially due to Quillen in \cite{QuillenMITnotes} in different language:

%\begin{proposition}[Quillen]
%\label{quillenconv}
%Let $A \to B$ be a map of animated rings whose fibre $I$ is connected, i.e., $\pi_0(I) = 0$. Then there is a natural complete descending $\mathbf{N}$-indexed filtration $F^*$ on $A$ with an identification $\gr^*_F A \simeq \mathrm{Sym}^*_B(L_{B/A}[-1])$.
%\end{proposition}

\begin{proposition}[Animated completions and Quillen convergence]
\label{quillenconv}
Let $A \to B$ be a map of animated rings which is surjective on $\pi_0$. Then there is a natural descending $\mathbf{N}$-indexed filtration $F^*$ on $A$ with an identification $\gr^*_F A \simeq \mathrm{Sym}^*_B(L_{B/A}[-1])$. 

If we further assume that the fibre $I$ of $A \to B$ is connected, then $F^n \in D^{\leq -n}$, and hence the filtration $F^*$ is complete.
\end{proposition}

A version of the proposition in the global context of algebraic derived stacks was carefully explained in \cite[Appendix A]{DanHLDerivedThetaStrat} as well. The construction of $F^*$ given below naturally produces a filtered animated $A$-algebra in the sense of \cite[\S 4.3]{RaksitHH} or equivalently (via the Rees construction) an affine derived scheme over $\mathbf{A}^1/\mathbf{G}_m \times \mathrm{Spec}(A)$; for our applications, it suffices to consider $F^*$ as a filtered $E_\infty$-$A$-algebra, i.e., as a commutative algebra object of the filtered derived $\infty$-category $\mathcal{DF}(A) := \mathrm{Fun}(\mathbf{N}^{op}, \mathcal{D}(A))$ of $A$.

\begin{proof}[Proof sketch]
We proceed via animation to construct the filtration. 

First, for a surjection $A \to B$ of polynomial $\mathbf{Z}$-algebras with kernel $I$, we simply take $F^*$ to be the $I$-adic filtration on $A$. As $I$ is a regular ideal, we have $I/I^2 \simeq L_{B/A}[-1]$ and also that $\mathrm{Sym}^*_B(I/I^2) \simeq \oplus_n I^n/I^{n+1}$ as graded $B$-algebras, which gives the desired filtration in this case. 

Next, note that the construction of $F^*$ as well as the identification of the associated graded in the previous paragraph is clearly functorial in the map $A \to B$. Animating this construction then gives the filtration in general. 

The assertion $F^n \in D^{\leq -n}$ when $I$ is connected relies on certain Koszul calculations by Quillen, see \cite[Proposition 4.11]{BhattCompddR} or \cite[Proposition 8.8]{QuillenMITnotes}. Granting this property, the completeness is immediate: the Milnor sequence for homotopy groups of inverse limits shows that $\pi_i(\lim_n F^n) = 0$ for all $i$, so $\lim_n F^n = 0$.
\end{proof}

The assumption $\pi_0(I) = 0$ appearing in the last sentence of \autoref{quillenconv} is quite restrictive: it does not apply when $A$ and $B$ are classical unless $A=B$. To extend this to maps $A \to B$ where we only demand a completeness property of the kernel, as well as to make compatibility of completions with compositions transparent, we shall need the following notion of completion, which helps push classical problems into the homotopical realm so that \autoref{quillenconv} might apply; this construction is due to Carlsson and inspired by the construction of the Adams spectral sequence in stable homotopy theory.

\begin{definition}[The Adams completion]
Fix a map $A \to B$ of animated rings. We denote its Cech nerve by
\[ \mathrm{Cech}(A \to B) := \left( \cosimp{B}{B \otimes_A B}{B \otimes_A B \otimes_A B} \right), \]
regarded as a cosimplicial animated $A$-algebra; let 
\[ \mathrm{Comp}(A \to B) := \lim \mathrm{Cech}(A \to B),\]
where the limit is computed in all (possibly non-connective) derived rings or equivalently complexes; we call this the {\em Adams completion} of $A$ along $B$. We say that $A$ is {\em Adams complete} along $A \to B$ (or along $I = \mathrm{fib}(A \to B)$) if the map $A \to \mathrm{Comp}(A \to B)$ is an isomorphism.
\end{definition}

\begin{example}[Descendable maps are Adams complete]
\label{DescendableComp}
Let $A \to B$ be a descendable map of animated rings; for example, $A$ and $B$ can be discrete with $B=A/I$ for a nilpotent ideal $I$. Then $A$ is Adams complete along $B$.
\end{example}

\begin{remark}[The $\mathrm{Tot}$-tower filtration]
For a map $A \to B$ of animated rings, write $I$ for the fibre of $A \to B$. Then the partial totalization $\mathrm{Tot}_{\leq n} \mathrm{Cech}(A \to B)$ identifies with $A/I^{\otimes_A n}$, and hence we have
\[ \mathrm{Comp}(A \to B) = \lim_n \mathrm{Tot}_{\leq n} \mathrm{Cech}(A \to B) \simeq \lim_n A/I^{\otimes_A n}.\]
It follows that if $I$ is connective (i.e., $A \to B$ is surjective on $\pi_0$), then $\mathrm{Comp}(A \to B)$ is connective, and thus naturally an animated ring. This will be the case of primary interest for us; in this case, we shall also write $A^{\wedge}_I := \mathrm{Comp}(A \to B)$ instead, and refer to it as the $I$-adic Adams completion of $A$.
\end{remark}

When $A \to B$ is surjective map of (classical) Noetherian rings with kernel $I$, then Carlsson has identified $A^{\wedge}_I$ with the classical $I$-adic completion $\lim_n A/I^n$ of $A$, see \cite{CarlssonDerComp,BhattCompddR}. Let us sketch a proof of this fact using the (by now well-understood) theory of derived $I$-completions. We shall prove the following slightly more general statement for later use:

\begin{proposition}[The case of finitely generated ideals]
\label{ClassicalAdams}
Let $A \to B$ be a map of animated rings that is surjective on $\pi_0$. Assume that the kernel $I = \ker(\pi_0(A) \to \pi_0(B))$ is finitely generated. Then $A' := \mathrm{Comp}(A \to B)$ is naturally identified with the derived $I$-completion $\widehat{A}$ of $A$.
\end{proposition}

When $A$ is a classical Noetherian ring, it is known that derived completions of finitely generated $A$-modules identify with their classical completions (see \cite[Tag 0EEU]{stacks-project}), so the above proposition fulfils the promise made right before. 

\begin{proof}
By definition,  the animated $A$-algebra $A'$ is an inverse limit of derived $I$-complete $A$-complexes and thus itself derived $I$-complete. The map $A \to A'$ then extends uniquely to a map $\widehat{A} \to A'$ that we shall show is an isomorphism. Choose generators $f_1,...,f_r \in I$. As both $\widehat{A}$ and $A'$ are derived $I$-complete, the isomorphy of $\widehat{A} \to A'$ can be checked after tensoring with the animated $A$-algebra $C := \mathrm{Kos}(A; f_1,....,f_r)$ by the derived Nakayama lemma. But $C$ is perfect as an $A$-complex, so tensoring with $C$ commutes with limits. As the formation of $\widehat{A} \to A'$ is then compatible with base change and because our assumptions are compatible with this base change, we reduce to checking the theorem in the case that $I=0$ or equivalently $\pi_0(A) = \pi_0(B)$. As $I=0$, we have $A=\widehat{A}$, so we must show that $A \to A' := \mathrm{Comp}(A \to B)$ is an isomorphism, i.e., that $A$ is Adams complete along $B$. If $A$ has only finitely many homotopy groups (which would be the case if the original $A$ in the question had this property, and is the main case relevant for the sequel), then the map $A \to \pi_0(A)$ is descendable; the factorization $A \to B \to \pi_0(B) = \pi_0(A)$ then forces $A \to B$ to be descendable too, which implies $A \simeq \mathrm{Comp}(A \to B)$. In the general case, one uses \autoref{QuillenConvAdams} below.
\end{proof}

Our goal is to explain two properties of the Adams completion. First, we construct a derived $I$-adic filtration on the Adams completion with understandable graded pieces, essentially generalizing \autoref{quillenconv}:

\begin{proposition}[The derived $I$-adic filtration]
\label{derIadic}
Let $A \to B$ be a map of animated rings which is surjective on $\pi_0$, i.e., the fibre $I$ lies is connective. Then the $I$-adic Adams completion $A^{\wedge}_{I}$ admits a natural complete descending $\mathbf{N}$-indexed filtration with an identification $\gr^* A^{\wedge}_I \simeq \mathrm{Sym}^*_B(L_{B/A}[-1])$.
\end{proposition}

We refer to the filtration in the proposition as the {\em derived $I$-adic filtration} on $A$. 

\begin{proof}
There is a natural map $\mathrm{Cech}(A \to B) \to B$ of cosimplicial animated $A$-algebras. By the surjectivity assumption on $A \to B$, each level $B^i  \to B$ of this map satisfies the hypothesis of Quillen's theorem: the induced map on $\pi_0$ is an isomorphism, and the map on $\pi_1$ (and in fact all $\pi_i$) is surjective as there exist sections $B \to B^i$. Consequently, we learn that $B^i$ admits a natural complete descending $\mathbf{N}$-indexed filtration with an identification $\gr^* B^i \simeq \mathrm{Sym}^*_B(L_{B/B^i}[-1])$. Totalizing, we learn that $A^{\wedge}_I = \lim \mathrm{Cech}(A \to B)$ admits a natural complete descending $\mathbf{N}$-indexed filtration with an identification $\gr^* A^{\wedge}_I \simeq  \lim \mathrm{Sym}^*_B(L_{B/B^\bullet}[-1])$ of graded animated $B$-algebras. It is thus sufficient to show that the natural map $a_{B/A}:\mathrm{Sym}^*_B(L_{B/A}[-1]) \to \lim \mathrm{Sym}^*_B(L_{B/B^\bullet}[-1])$ of graded animated $B$-algebras is an isomorphism. As the formation of the filtration, the identification of its associated graded, and the formation of the map $a_{B/A}$ (as a map of graded animated $B$-algebras) commutes with base change on $A$, the desired assertion can be checked after base change along any map $A \to A'$ with the property that tensoring with $A'$ is conservative on the image of $D(B) \to D(A)$. After such a base change (e.g., with $A'=B$), we may assume that  $A \to B$ admits a section.  Our task is to show that the map $a_n:\mathrm{Sym}^n_B(L_{B/A}[-1]) \to \mathrm{Sym}^n_B(L_{B/B^\bullet}[-1])$ of cosimplicial $B$-modules (where the LHS is a constant cosimplicial object) induces an isomorphism on totalizations. For $n=1$, the fibre of this map is given by $L_{B^\bullet/A} \otimes_{B^\bullet} B[-1]$, which vanishes on totalization as $A \to B^\bullet$  is a cosimplicial homotopy equivalence of cosimplicial animated $A$-algebras by the choice of a section. For $n > 1$, the fibre $\mathrm{fib}(a_n)$ admits a finite filtration with graded pieces of the form $\mathrm{Sym}^i_B(L_{B^\bullet/A} \otimes_{B^\bullet} B[-1]) \otimes_B \mathrm{Sym}^{n-i}(L_{B/B^\bullet}[-1])$ $i > 0$; as $i > 0$, each of these pieces is cosimplicial homotopic to $0$ (as the term $\mathrm{Sym}^i_B(L_{B^\bullet/A} \otimes_{B^\bullet} B[-1]) $ is so) whence the totalization vanishes.
\end{proof}

The filtration in \autoref{derIadic} agrees with that in \autoref{quillenconv} in case of overlap:

\begin{corollary}[Quillen's filtration as a derived $I$-adic filtration]
\label{QuillenConvAdams}
Let $A \to B$ be a map of animated rings whose fibre $I$ is connected, i.e., $\pi_0(I) = 0$. Then $A \simeq A^{\wedge}_I$, with the filtration in \autoref{quillenconv} naturally identifying with the derived $I$-adic filtration.
\end{corollary}
\begin{proof}
By construction of the derived $I$-adic filtration, the natural map $A \to A^{\wedge}_I = \lim \mathrm{Cech}(A \to B)$ refines  to a filtered map carrying the filtration $F^*$ from \autoref{quillenconv} to the derived $I$-adic filtration on $A^{\wedge}_I$. As the graded pieces match up, the map is a filtered isomorphism, proving the corollary.
\end{proof}

Our second main result about Adams completions is its compatibility with composition; this was much easier (for us) to prove using the definition of the Adams completion given above, and is the main reason we did not adopt the animated $I$-adic completion from \autoref{quillenconv} as our definition.

\begin{proposition}[Composition of completions]
\label{AdamsComptwo}
Let $A \to B \to C$ be a composition of maps of animated rings that are both surjective on $\pi_0$, with fibres $I = \mathrm{fib}(A \to B)$, $J = \mathrm{fib}(B \to C)$ and $K = \mathrm{fib}(A \to C)$.  Assume that $B$ is Adams complete along $J$ and that the same holds true after every base change on $B$. Then $A^{\wedge}_I \simeq A^{\wedge}_K$, and similarly after any base change on $A$.
\end{proposition}
\begin{proof}
The assumptions are stable under base change, so it suffices to show that $A^{\wedge}_I \simeq A^{\wedge}_K$. For this, we shall use an abstract argument, so we explain it in that generality, following the universal descent terminology of Liu--Zheng \cite[\S 3.1]{LiuZheng}.

Let $\mathcal{C}$ be a small $\infty$-category admitting finite limits. 

First, let us introduce a notion of completion. For a map $f:Y \to Z$ in $\mathcal{C}$, write $Z^{\wedge}_Y := \colim Y^{\bullet/Z}$ for the geometric realization of the Cech nerve $Y^{\bullet/Z}$ of $f$, regarded as a subobject of $Z$ in the $\infty$-topos $\mathrm{PShv}(\mathcal{C})$ of presheaves of spaces on $\mathcal{C}$; explicitly, for any test object $T \in \mathcal{C}$, the fibre of $\mathrm{Map}(T,Y^{\bullet/Z}) \to \mathrm{Map}(T,Z)$ over a map $g:T \to Z$ is empty if $g$ does not factor over $f$, and a singleton $\{\ast\}$ otherwise. These definitions extend naturally to arbitrary maps $f:Y \to Z$ in $\mathrm{PShv}(\mathcal{C})$, not merely those lying in the image of the Yoneda embedding $\mathcal{C} \to \mathrm{PShv}(\mathcal{C})$.

Next, we introduce the notion of descent. Fix a presheaf $F:\mathcal{C}^{op} \to \mathcal{D}$ valued in a presentable $\infty$-category $\mathcal{D}$; recall that $F$ extends uniquely to a functor $\mathrm{PShv}(\mathcal{C})^{op} \to \mathcal{D}$ (also denoted $F$) carrying  colimits in $\mathrm{PShv}(\mathcal{C})$ to limits in $\mathcal{D}$, and we shall identify $F$ with such an extension. We say that a map $f:Y \to Z$ in $\mathcal{C}$ (or in $\mathrm{PShv}(\mathcal{C})$) is of {\em $F$-descent} if $F(Z) \simeq F(Z^{\wedge}_Y) = \lim F(Y^{\bullet/Z})$, and we say that $f$ is {\em universally of $F$-descent} if the same holds true for every base change of $f$. Maps that are universally of $F$-descent are stable under base change, composition, retractions, as well as passage to the second composition factor (see \cite[Lemma 3.1.2]{LiuZheng}).

In the above notation, the key lemma is the following:

\begin{lemma}
Fix a composition $U \to V \to X$ of maps in $\mathcal{C}$ such that $U \to V$ is universally of $F$-descent. Then $F$ carries the map natural map $X^{\wedge}_U \to X^{\wedge}_V$ to an isomorphism, i.e., $F(X^{\wedge}_V) \simeq F(X^{\wedge}_U)$.
\end{lemma}
\begin{proof}
We shall show that the composition $U \to X^{\wedge}_U \to X^{\wedge}_V$ is universally of $F$-descent. Granting this statement, it follows that $F(X^{\wedge}_V) \simeq F((X^{\wedge}_V)^{\wedge}_U)$. But we also have $U^{\bullet/X^{\wedge}_V} \simeq U^{\bullet/X}$ since $X^{\wedge}_V \to X$ is a monomorphism, whence $(X^{\wedge}_V)^{\wedge}_U \simeq X^{\wedge}_U$ via the natural map. Combining the last two sentences gives the desired identification $F(X^{\wedge}_V) \simeq F(X^{\wedge}_U)$.

It remains to show that the composition $U \to X^{\wedge}_U \to X^{\wedge}_V$ is universally of $F$-descent. Since this map factors as $U \to V \to X^{\wedge}_V$, using our assumptions on $U \to V$, it suffices to show that $V \to X^{\wedge}_V$ is universally of $F$-descent. But this statement is clear and independent of $F$: the map $V \to X^{\wedge}_V$ is an effective epimorphism in $\mathrm{PShv}(\mathcal{C})$, so $V^{\bullet/X^{\wedge}_V} \to X^{\wedge}_V$ is a colimit diagram in $\mathrm{PShv}(\mathcal{C})$, and is thus carried to a limit diagram by any presheaf on $\mathcal{C}$ (regarded as a colimit preserving functor on $\mathrm{PShv}(\mathcal{C})$).
\end{proof}

To prove the proposition, one applies the preceding lemma to $\mathcal{C}^{op} = \mathrm{CAlg}^{an}$, $\mathcal{D} = \mathcal{D}(\mathbf{Z})$, $F = R\Gamma(-, \mathcal{O})$, and $U \to V \to X$ being the opposite of $A \to B \to C$. 
\end{proof}

The following consequence will be particularly useful for maps of the form $S \otimes_R S \to S$, where $R \to S$ is a map of animated rings surjective on $\pi_0$.

\begin{corollary}[Completeness of $\pi_0$-isomorphisms]
\label{AdamsComppi0}
Say $A \to B$ is a map of animated rings with fibre $I$ such that $\pi_0(A) \simeq \pi_0(B)$. Then $A$ is $I$-adically Adams complete and similarly after any base change. 
\end{corollary}
\begin{proof}
The special case where $B=\pi_0(A)$ (or more generally when $I$ is connected) follows from \autoref{QuillenConvAdams}. In the general case, consider the composition $A \to B \to \pi_0(B)$. Then $B$ universally Adams complete along $B \to \pi_0(B)$ by the first statement applied to this map. Since $\pi_0(A) \simeq \pi_0(B)$ by assumption, the claim follows from \autoref{AdamsComptwo} and the first statement applied to $A \to \pi_0(A)$. 
\end{proof}

In the rest of this section, we record the interaction of the Frobenius with the derived $I$-adic filtration, and use it to give a model-free proof of a recent result from \cite{BallardIyengarLankMukhopadhyayPollitzHighFrobeniusPushforwards}; this material is not relevant to the application to dualizing complexes, and only needs \autoref{quillenconv}.

\begin{remark}[The Frobenius and the $I$-adic filtration]
\label{FrobIAdic}
Say $A \to B$ is a map of animated $\mathbb{F}_p$-algebra with fibre $I$; assume $\pi_0(A) \to \pi_0(B)$ is surjective. \autoref{quillenconv} constructed, via animation, a filtration $F^*$ on $A$. Examining the construction shows that the Frobenius $\phi:A \to A$ is naturally $p$-speed filtered, i.e., it extends to a filtered map $F^* \to F^{p*}$. Thus, for any integer $e \geq 1$, the $e$-fold Frobenius $\phi^e:A \to A$ refines to a filtered map $F^* \to F^{p^e *}$; this compatibility can in fact be seen for any filtered animated $\mathbf{F}_p$-algebra via the Rees construction by observing that the Frobenius on an affine derived scheme over $\mathbf{A}^1/\mathbf{G}_m \times \mathrm{Spec}(\mathbf{F}_p)$ lies over the Frobenius on $B\mathbf{G}_m$, and pullback along the latter corresponds to rescaling the grading by $p$ when we identify quasi-coherent sheaves on $B\mathbf{G}_m$ with $\mathbf{Z}$-graded vector spaces.
\end{remark}

\autoref{FrobIAdic} has the following consequence:

\begin{corollary}[Factoring the Frobenius through $\pi_0$]
\label{FactorFrob}
Let $A$ be an animated $\mathbb{F}_p$-algebra with $\pi_i(A) = 0$ for $i \notin [0,n]$. For $e > \log_p(n)$, the $e$-fold Frobenius map $\phi^e:A \to A$ factors over $A \to \pi_0(A)$. In particular, $\phi^e_* A \in D(A)$ comes from $D(\pi_0(A))$.
\end{corollary}
\begin{proof}
We have a commutative square
\[ \xymatrix{ A \ar[d] \ar[r]^{\phi^e} & A \ar[d] \\ 
A/F^1 \simeq \pi_0(A) \ar[r] & A/F^{p^e} }\]
of animated $\mathbb{F}_p$-algebras, where the vertical maps are the canonical ones, the top horizontal map is $\phi^e$, and the bottom horizontal map is induced by $\phi^e$ thanks to \autoref{FrobIAdic}. By \autoref{quillenconv} and the assumption $p^e > n$, we have $F^{p^e} \in D^{\leq -p^e} \subset D^{< -n}$, so the right vertical map $A \to A/F^{p^e}$ induces an isomorphism on $\tau_{\leq n}$. But $\pi_i(A) = 0$ for $i \notin [0,n]$, so the right vertical map admits a unique section $s:A/F^{p^e} \to A$ (given by the truncation map $A/F^{p^e} \to \tau_{\leq n} (A/F^{p^e})$). Composing the bottom horizontal map with $s$ then provides the desired factorization.
\end{proof}

A concrete consequence of this corollary (with a similar proof, albeit phrased in terms of explicit models) appeared recently as an essential technical input in \cite[Proposition 2.7]{BallardIyengarLankMukhopadhyayPollitzHighFrobeniusPushforwards}; we explain the statement for completeness. 

\begin{theorem}[Ballard--Iyengar--Lank--Mukhopadhyay--Pollitz]
Given a Noetherian local $F$-finite $\mathbb{F}_p$-algebra $(R,\mathfrak{m},k)$, the smallest thick subcategory of $D(R)$ generated by $\phi^e_* R$ contains $k$ for $e > \log_p(c)$, where $c:=\mathrm{codepth}_{\mathfrak{m}}(R) = \mathrm{edim}(\widehat{R}) - \mathrm{depth}_{\mathfrak{m}}(R)$.
\end{theorem}

\begin{proof}
Fix a minimal system of generators $f_1,...,f_d \in \mathfrak{m}$. Consider the animated $R$-algebra $A=\mathrm{Kos}(R;f_1,...,f_d)$ given by the Koszul construction. By choosing a splitting of $R \to k$, we obtain a surjection $P := k\llbracket x_1,...,x_d \rrbracket \to \widehat{R}$ via $x_i \mapsto f_i$. Auslander--Buchscbaum over $P$ shows that the amplitude of $A$ (which is the projective dimension of $\widehat{R}$ over $P$) equals the codepth $c$, whence $\pi_i(A) = 0$ for $i \notin [0,c]$. Applying \autoref{FactorFrob}, we learn that for $e > \log_p(c)$, the $e$-fold Frobenius $A \to A$ factors over $A \to \pi_0(A) = k$. Thus, the nonzero object $\phi^e_* A \in D(R)$ admits a $k$-complex structure, i.e., it lies in the image of $D(k) \to D(R)$. As $k$ is a field, we can then split a copy of $k$ from $\phi^e_* A$, whence the smallest thick subcategory of $D(R)$ generated by $\phi^e_* R$ certainly contains $k$. 
\end{proof}

\subsubsection*{Formally completing Noetherian $F$-finite rings along the diagonal}
\label{compffin}

Our main object of study is the following notion in the case of an $F$-finite $\mathbb{F}_p$-algebra $R$:

\begin{definition}[The completion along the diagonal]
\label{AdamsCompFFinDef}
For an animated $\mathbb{F}_p$-algebra $R$, we define 
\[ R \widehat{\otimes_{\mathbb{F}_p}} R  := \mathrm{Comp}(R \otimes_{\mathbb{F}_p} R \xrightarrow{\mu} R),\]
and call it the completion of $R \otimes_{\mathbb{F}_p} R$ along the diagonal. 
\end{definition}

As defined,  $R \widehat{\otimes_{\mathbb{F}_p}} R$ is an animated ring with potentially unbounded homotopy. The main goal of this section is to show that, in fact,  $R \widehat{\otimes_{\mathbb{F}_p}} R$ is a classical Noetherian ring when $R$ is $F$-finite (\autoref{FfinCompDiag}). One of our main tools to understanding $ R \widehat{\otimes_{\mathbb{F}_p}} R$ is the following:

\begin{remark}[The derived $I_\Delta$-adic filtration]
For an animated $\mathbb{F}_p$-algebra $R$, the animated ring $R \widehat{\otimes_{\mathbb{F}_p}} R$ comes endowed with the complete derived $I_\Delta$-adic filtration, where $I_\Delta = \mathrm{fib}(R \otimes_{\mathbb{F}_p} R \to R)$; see \autoref{derIadic}. The associated graded of this filtration is identified with 
\[ \mathrm{Sym}^*_R(L_{R/R \otimes_{\mathbb{F}_p} R}[-1]) \simeq \mathrm{Sym}^*_R(L_{R/\mathbb{F}_p}),\]
where we identified $L_{R/R \otimes_{\mathbb{F}_p} R}[-1] \simeq L_{R/\mathbb{F}_p}$ via (say) the first inclusion.
\end{remark}

\begin{remark}[The cotangent complex of $F$-finite rings]
Let $S$ be a Noetherian $F$-finite ring. The Frobenius $F_S:S \to S$ induces the $0$ map $L_{S/\mathbb{F}_p} \to L_{S/\mathbb{F}_p}$, so it follows from the transitivity triangle that $L_{S/\mathbb{F}_p}$ is a direct summand of the relative cotangent complex $L_{F_S}$. By the assumption on $S$, the map $F_S$ is a finite map of Noetherian rings, so $L_{F_S}$ is a pseudocoherent $S$-complex (see for instance  \cite[6.11]{Iyengar.AndreQuillenHomologyOfCommutativeAlgebras}, \cite[Proposition 4.12]{Quillen.OnThe(Co-)homologyOfCommRings}, or \cite[\href{https://stacks.math.columbia.edu/tag/08PZ}{Tag 08PZ}]{stacks-project} and \cite[\href{https://stacks.math.columbia.edu/tag/066E}{Tag 066E}]{stacks-project})
%\todo{Karl: @Bhargav, hopefully these are the references you meant, I didn't track it down in Illusie's book}%*** reference Quillen, Iyenger or Illusie ****), 
whence the same holds true for the summand $L_{S/\mathbb{F}_p}$. In particular, the associated graded of the derived $I_\Delta$-adic filtration on $S \widehat{\otimes_{\mathbb{F}_p}} S$ identifies with $\mathrm{Sym}^*_S(K)$ for a pseudocoherent complex $K$, and thus has good finiteness properties. 
\end{remark}

Let us give some important examples:

\begin{example}[The geometric case]
\label{GeomCompDiag}
Say $R$ is a finite type algebra over a perfect field $k$ of characteristic $p$. Then $R \widehat{\otimes_{\mathbb{F}_p}} R$ coincides with the classical completion $R \widehat{\otimes_k} R$ of $R \otimes_k R$ along the ideal $I_\Delta$ of the diagonal via the natural map $R \widehat{\otimes_{\mathbb{F}_p}} R \to R \widehat{\otimes_k} R$; in particular, $R \widehat{\otimes_{\mathbb{F}_p}} R$ is an ordinary ring. To prove this, by identifying $R \widehat{\otimes_{\mathbb{F}_p}} R$ as the $I_\Delta$-adic Adams completion of $R \otimes_k R$ using Carlsson's theorem (\autoref{ClassicalAdams}) and comparing the corresponding complete derived $I_\Delta$-adic filtrations on the two sides, the claim follows from the fact that $L_{R/\mathbb{F}_p} \simeq L_{R/k}$ as $L_{k/\mathbb{F}_p} = 0$ by perfectness. 
\end{example}

\begin{example}[The regular $F$-finite case]
\label{regularFfindiag}
Say $R$ is a Noetherian $F$-finite regular ring. Then we claim the following:
\begin{enumerate}
\item $R \widehat{\otimes_{\mathbb{F}_p}} R$ identifies with the classical $J_\Delta$-adic completion $A = \lim_n (R \otimes_{\mathbb{F}_p} R)/J_\Delta^n$ of $R \otimes R$ along the diagonal. This identification matches the derived $I_\Delta$-adic filtration to the classical $J_\Delta$-adic filtration on $A$.\label{regularFfindiag.itm.1}

\item $R \widehat{\otimes_{\mathbb{F}_p}} R$ is regular, and the kernel $J$ of $R \widehat{\otimes_{\mathbb{F}_p}} R \to R$ is a regular ideal with $J/J^2 \simeq \Omega^1_{R/\mathbb{F}_p} \simeq L_{R/\mathbb{F}_p}$.  \label{regularFfindiag.itm.2}
\end{enumerate}

For \autoref{regularFfindiag.itm.1}: by comparing the complete derived $I_\Delta$-adic filtration on $R \widehat{\otimes_{\mathbb{F}_p}} R$ with the evident complete $J_\Delta$-adic filtration on $A$, we are reduced to showing that the natural map induces an isomorphism $\mathrm{Sym}^n_R(L_{R/\mathbb{F}_p}) \simeq \oplus J_\Delta^n/J_\Delta^{n+1}$ for all $n$. This is classical when $R$ is a smooth $\mathbb{F}_p$-algebra. In general, both sides are compatible with filtered colimits in $R$, so the claim follows by writing $R$ as a filtered colimit of smooth $\mathbb{F}_p$-algebras using Popescu's theorem.

For \autoref{regularFfindiag.itm.2}: as $R$ is regular, Popescu's theorem implies that the $R$-complex  $L_{R/\mathbb{F}_p}$ identifies with $\Omega^1_{R/\mathbb{F}_p}$ and is $R$-flat. Moreover, as $R$ is $F$-finite, the $R$-module $\Omega^1_{R/\mathbb{F}_p} \simeq \Omega^1_{R/R^p}$ is also finitely generated. Thus, we learn that $L_{R/\mathbb{F}_p}$ is a finite projective $R$-module. The first factor inclusion $i_1:R \to R \widehat{\otimes_{\mathbb{F}_p}} R$ can then be extended to a map $\mathrm{Sym}_R^*(L_{R/\mathbb{F}_p}) \to R \widehat{\otimes_{\mathbb{F}_p}} R$ splitting the derived $I_\Delta$-adic filtration, i.e., identifying the target with the completion of the source along the ideal of positive degree elements. As the latter is the completion of the regular ring $\mathrm{Sym}^*_R(\Omega^1_{R/\mathbb{F}_p})$, it must be regular itself; this immediately proves the desired statements in (2).
\end{example}

\begin{example}[The Artinian case]
\label{ArtinDiagComp}
Assume $(S,\mathfrak{m})$ is an Artinian local $F$-finite ring with residue field $k$. We claim that $S \widehat{\otimes_{\mathbb{F}_p}} S$ is discrete. To see this, choose a section $k \to S$ of $S \to S/\mathfrak{m} = k$, so $k \to S$ is a finite flat universal homeomorphism. By the upcoming \autoref{PushoutFFin}, we learn that
\[ \xymatrix{ k \otimes_{\mathbb{F}_p} k \ar[r] \ar[d] & k \widehat{\otimes_{\mathbb{F}_p}} k \ar[d]  \\
S \otimes_{\mathbb{F}_p} S \ar[r] & S \widehat{\otimes_{\mathbb{F}_p}} S  } \]
is a pushout square of animated rings. But the left vertical map is obviously finite flat, whence the same is true for the right vertical map. Now $k \widehat{\otimes_{\mathbb{F}_p}} k$ is discrete by \autoref{regularFfindiag}, so the flat algebra $S \widehat{\otimes_{\mathbb{F}_p}} S$ is also discrete.
\end{example}

The following lemma was used above and will be quite useful later in reducing to the regular case, where one has better homological control:

\begin{lemma}
\label{PushoutFFin}
Let $R \to S$ be a map of Noetherian $F$-finite rings which is surjective up to powers of Frobenius. Then the square 
\[ \xymatrix{ R \otimes_{\mathbb{F}_p} R \ar[r]^c \ar[d]^a & R \widehat{\otimes_{\mathbb{F}_p}} R \ar[d]^b \\
S \otimes_{\mathbb{F}_p} S \ar[r]^d & S \widehat{\otimes_{\mathbb{F}_p}} S  }\]
is a pushout square of animated rings.
\end{lemma}

\begin{proof}
By Gabber's theorem and the $2$-out-of-$3$ property for pushout squares, it is enough to show this claim when $R$ is regular, so let us assume $R$ is regular. Let $A$ denote the relevant pushout in the category of animated rings, so we have a pushout square
\[ \xymatrix{ R \otimes_{\mathbb{F}_p} R \ar[r] \ar[d] & R \widehat{\otimes_{\mathbb{F}_p}} R \ar[d] \\
S \otimes_{\mathbb{F}_p} S \ar[r] & A  }\]
as well as a natural map $A \to S \widehat{\otimes_{\mathbb{F}_p}} S$ that we aim to prove is an isomorphism. Now the top horizontal map is the Adams completion $\mathrm{Comp}(R \otimes_{\mathbb{F}_p} R \to R)$. As $R$ is regular, the ring $S \otimes_{\mathbb{F}_p} S$ is a perfect complex over $R \otimes_{\mathbb{F}_p} R$, so tensoring with $S \otimes_{\mathbb{F}_p} S$ commutes with limits, whence the bottom horizontal map is identifies $A$ with the Adams completion 
\[ \mathrm{Comp}\left(S \otimes_{\mathbb{F}_p} S \to (S \otimes_{\mathbb{F}_p} S) \otimes_{R \otimes_{\mathbb{F}_p} R} R \simeq S \otimes_R S\right).\] 
But, by our assumption on $R \to S$, the relative multiplication map $S \otimes_R S \to S$ is a quotient by a nilpotent ideal on $\pi_0$, so we have
\[ \mathrm{Comp}\left(S \otimes_{\mathbb{F}_p} S \to S \otimes_R S\right) \simeq \mathrm{Comp}\left(S \otimes_{\mathbb{F}_p} S \to S\right) =: S \widehat{\otimes_{\mathbb{F}_p}} S\]
by \autoref{DescendableComp} and \autoref{AdamsComptwo}, so we win.
\end{proof}

The following lemma is standard but we record it for convenience:

\begin{lemma}
\label{cotcompleteclass}
Let $k$ be a base ring, let $S$ be an animated $k$-algebra, and let $J \subset \pi_0(S)$ be a finitely generated ideal; write $S^{\wedge}_J$ for the derived $J$-adic completion of $S$. Then $L_{S/k} \to L_{S^{\wedge}_J/k}$ becomes an isomorphism after application of $- \otimes_S^L \pi_0(S)/J$.
\end{lemma}
\begin{proof}
By the transitivity triangle, it suffices that $L_{S^{\wedge}_J/S} \otimes_S \pi_0(S)/J = 0$. Now the formation of the cotangent complex always commutes with derived base change. As $S \to S^{\wedge}_J$ becomes an isomorphism after $- \otimes_S^L \pi_0(S)/J$, the claim follows.
\end{proof}

The next lemma shows that the map $S \otimes_{\mathbf{F}_p} S \to  S \widehat{\otimes_{\mathbb{F}_p}} S$ behaves like the completion map of a Noetherian ring at the level of cotangent complexes: 

\begin{corollary}
\label{cotcomplete}
Let $S$ be a Noetherian $F$-finite ring. Then $L_{(S \widehat{\otimes_{\mathbf{F}_p}} S)/ (S \otimes_{\mathbf{F}_p} S)} \otimes_{ S \widehat{\otimes_{\mathbb{F}_p}} S} S \simeq 0$.
\end{corollary}
\begin{proof}
Thanks to Gabber's theorem and \autoref{PushoutFFin}, it suffices to prove the claim when $S=R$ is a regular Noetherian $F$-finite ring. In this case, \autoref{regularFfindiag} shows that  the map $L_{S/(S \otimes_{\mathbb{F}_p} S)} \to L_{S/(S \widehat{\otimes_{\mathbb{F}_p}} S)}$ is an isomorphism (as both sides identify with $\Omega^1_{S/\mathbb{F}_p}[1]$), so the claim follows by the transitivity triangle.
\end{proof}

The promised result is the following:

\begin{proposition}
\label{FfinCompDiag}
Let $S$ be a Noetherian $F$-finite ring. Then $S \widehat{\otimes_{\mathbb{F}_p}} S$ is a classical Noetherian ring.

\end{proposition}

Note that the classicality assertion in the proposition is quite false for the associated graded of derived $I_\Delta$-adic filtration on $S \widehat{\otimes_{\mathbb{F}_p}} S$: in fact, the cotangent complex $L_{R/\mathbb{F}_p}$ can be unbounded to the left if $R$ is not lci. Nonetheless, the derived $I_\Delta$-adic filtration will be our main tool; the strategy is to use completeness with respect to this filtration to reduce the statement to the Artinian case treated in \autoref{ArtinDiagComp}.

\begin{proof}
First, let us show that $S \widehat{\otimes_{\mathbb{F}_p}} S$ is a Noetherian\footnote{An animated ring $R$ is called Noetherian if $\pi_0(R)$ is Noetherian and each $\pi_i(R)$ is a finitely generated $\pi_0(R)$-module.} animated $\mathbb{F}_p$-algebra and has only finitely many nonzero homotopy groups. By Gabber's theorem, we can write $S = R/I$, where $R$ is a Noetherian $F$-finite regular ring and $I \subset R$ is an ideal such that $R$ is $I$-adically complete. Consider the diagram
\[ \xymatrix{ R \otimes_{\mathbb{F}_p} R \ar[r] \ar[d] & R \widehat{\otimes_{\mathbb{F}_p}} R \ar[d] \\
S \otimes_{\mathbb{F}_p} S \ar[r] & S \widehat{\otimes_{\mathbb{F}_p}} S  }\]
of animated rings. By \autoref{PushoutFFin}, this square is a pushout square of animated rings. As $R$ is regular, $S$ is a perfect $R$-complex, and thus $S \otimes_{\mathbb{F}_p} S$ is a perfect $(R \otimes_{\mathbb{F}_p} R)$-complex. By the pushout property above, it follows that $S \widehat{\otimes_{\mathbb{F}_p}} S$ is a perfect $(R \widehat{\otimes_{\mathbb{F}_p}} R)$-complex. But $R \widehat{\otimes_{\mathbb{F}_p}} R$ is a regular ring by \autoref{regularFfindiag}, so the claim follows: any perfect complex over a regular ring has only finitely many cohomology groups, each of which is finitely generated.

Thanks to this Noetherianness, it suffices to show that  $S \widehat{\otimes_{\mathbb{F}_p}} S$ becomes classical after (derived or equivalently Adams) $\mathfrak{m}$-adic completion for every maximal ideal $\mathfrak{m} \subset \mathrm{Spec}(\pi_0(S \widehat{\otimes_{\mathbb{F}_p}} S))$. Now $ \pi_0(S \widehat{\otimes_{\mathbb{F}_p}} S)$ is complete along the multiplication map to $S$, so $\mathfrak{m}$ in fact comes from a maximal ideal of $S$ (also called $\mathfrak{m}$) under the inclusion $\mathrm{Spec}(S) \subset \mathrm{Spec}(\pi_0(S \widehat{\otimes_{\mathbb{F}_p}} S))$. To check the desired classicality, we shall show the following assertions:
\begin{enumerate}
    \item The natural map gives an isomorphism $(S \widehat{\otimes_{\mathbb{F}_p}} S)_{\mathfrak{m}}^{\wedge} \simeq S_{\mathfrak{m}}^{\wedge} \widehat{\otimes_{\mathbb{F}_p}} S_{\mathfrak{m}}^{\wedge}$.
    \item If $S$ is complete with respect to some ideal $J$, then the natural map induces an isomorphism
$S \widehat{\otimes_{\mathbb{F}_p}} S \simeq \lim_n (S/J^n \widehat{\otimes_{\mathbb{F}_p}} S/J^n)$.
    \end{enumerate}

The proposition follows from these assertions: (a) allows us to assume $(S,\mathfrak{m})$ is a complete Noetherian local ring, and then (b) together with the calculation in \autoref{ArtinDiagComp} shows the desired discreteness (as limits of discrete rings cannot have higher homotopy). Thus, it remains to show (a) and (b).

To show (a), by identifying Adams completions and derived completions for Noetherian rings, we learn that both terms in (a) are Adams complete along the map down to $k=S/\mathfrak{m}$.  Comparing graded pieces of the resulting complete filtrations, we are reduced to checking that 
\[ L_{k/(S \widehat{\otimes_{\mathbb{F}_p}} S)_{\mathfrak{m}}^{\wedge}} \simeq L_{k/S_{\mathfrak{m}}^{\wedge} \widehat{\otimes_{\mathbb{F}_p}} S_{\mathfrak{m}}^{\wedge}} \]
via the natural map. Using \autoref{cotcompleteclass} and \autoref{cotcomplete} multiple times, both are identified with $L_{k/(S \otimes_{\mathbb{F}_p} S)}$, so the claim follows.

For (b), we shall compare the derived $I_\Delta$-adic filtrations on both sides (where the limit is endowed with the limiting filtration). As the map is naturally filtered and both filtrations are complete, it suffices to show we get an isomorphism on graded pieces, i.e., for each $k \geq 0$, we must show that
\[ \mathrm{Sym}^k_S(L_{S/\mathbb{F}_p}) \simeq \lim_n \mathrm{Sym}^k_{S/J^n}(L_{(S/J^n)/\mathbb{F}_p}).\]

For $k=1$, consider the triangle
\[ \{L_{S/\mathbb{F}_p} \otimes_S S/J^n\}_n \to \{L_{(S/J^n)/\mathbb{F}_p}\}_n \to Q_n\]
of projective systems of complexes over the projective system $\{S/J^n\}_n$ of rings. By \cite[Theorem 1.2 (ii)]{MorrowProcdh}, the pro-system $\{H^i(Q_n)\}_n$ is essentially $0$ for all $i$, so the system $\{Q_n\}$ has vanishing inverse limit. Thus, we learn that
\[ \lim_n L_{S/\mathbb{F}_p} \otimes_S S/J^n \simeq \lim_n L_{(S/J^n)/\mathbb{F}_p},  \]
so it suffices to show that the natural map $K \to \lim_n K \otimes_S S/J^n$ is an isomorphism for $K=L_{S/\mathbb{F}_p}$; but this holds true for any pseudocoherent complex by general nonsense on derived completions, so we win. 

For $k > 1$, by taking symmetric powers of the above exact triangle, it suffices to show that for each $i > 0$ and $M \in D^{\leq 0}(S)$, the system $\{ \mathrm{Sym}^i_{S/J^n}(Q_n) \otimes_S M\}_n$ has essentially $0$ cohomology groups. In the spectral sequence computing $\mathrm{Sym}^i_{S/J^n}(Q_n) \otimes_S M$ induced by the Postnikov filtration of $Q_n$, the $E_\infty^j$-term $H_j(\mathrm{Sym}^i_{S/J^n}(Q_n) \otimes_S M)$ only sees contributions the finitely many $E_2$ terms of the form $H_{j-k}(\mathrm{Sym}^i_{S/J^n}(H_k(Q_n)) \otimes_S M)$, so the claim again follows from Morrow's theorem. 
\end{proof}

\begin{remark}[A semi-classical description of $S \widehat{\otimes_{\mathbb{F}_p}} S$]
For any $\mathbb{F}_p$-algebra $S$, the Tot-tower filtration on $S \widehat{\otimes_{\mathbb{F}_p}} S$ shows that
\[ S \widehat{\otimes_{\mathbb{F}_p}} S = \lim_n \left( (S \otimes_{\mathbb{F}_p} S)/I_\Delta^{\otimes n}\right), \]
where the tensor on the last term takes place in $D(S \otimes_{\mathbb{F}_p} S)$. Now if $S$ is Noetherian and $F$-finite, then \autoref{FfinCompDiag} shows that this object is discrete, i.e., concentrated in degree $0$. The Milnor sequence for cohomology groups of inverse limits then gives a short exact sequence
\[ 0 \to R^1 \lim \ker(\alpha_n) \to S \widehat{\otimes_{\mathbb{F}_p}} S \to \lim_n \left( (S \otimes_{\mathbb{F}_p} S)/I_\Delta^n\right) \to 0,\]
where the term on the right is the classical $I_\Delta$-adic completion of the commutative ring $S \otimes_{\mathbb{F}_p} S$, while $\alpha_n$ is the natural map $I_\Delta^{\otimes n} \to S \otimes_{\mathbb{F}_p} S$. When $S$ is regular, the term on the left vanishes by the calculation in \autoref{regularFfindiag}; we do not know if it vanishes in general.
\end{remark}

%%%%%%%%%%%%%%%%%%%%%%%%%%%%%%%%%%%%%%%%%%%%%%%%%%%%%%%%%%%%%%%%%%%%%%%%
%%%%%%%%%%%%%%%%%%%%%%%%%%%%%%%%%%%%%%%%%%%%%%%%%%%%%%%%%%%%%%%%%%%%%%%%
%%%%%%%%%%%%%%%%SUBSECTION:  FINITE TYPE OVER A FIELD%%%%%%%%%%%%%%%%%%%
%%%%%%%%%%%%%%%%%%%%%%%%%%%%%%%%%%%%%%%%%%%%%%%%%%%%%%%%%%%%%%%%%%%%%%%%
%%%%%%%%%%%%%%%%%%%%%%%%%%%%%%%%%%%%%%%%%%%%%%%%%%%%%%%%%%%%%%%%%%%%%%%%

\subsection{Recap of the finite type over a field case}
\label{DualCompStensClassical}

We fix a field $k$ and work with finite type $k$-algebras.  In classical approaches to Grothendieck duality, the dualizing complex $\omegacan{X} \in D^b_{coh}(X)$ of a variety $X$ is defined in a somewhat ad hoc fashion. Indeed, when $X$ is proper, one can define $\omegacan{X}$ quite conceptually in terms of the right adjoint $D^b_{coh}(k) \to D^b_{coh}(X)$ to pushforward; but then in general one sets $\omegacan X := \omegacan{\overline{X}}|_X$, where $X \subset \overline{X}$ is a compactification, and then checks independence of $\overline{X}$. 
This approach via compactifications  makes the intrinsic meaning of the dualizing complex quite mysterious. There are a few solutions to this problem: one can  develop a full $6$ functor formalism for quasi-coherent sheaves at the expense of passing to fancier frameworks (e.g., to the solid quasi-coherent sheaves \cite{ClausenScholzeCondensedMathAndComplex}) to extend the simple description that worked for proper $X$, or one can use rigidified dualizing complexes \cite{VandenBerghExistence,YekutieliDualizingComplexesOverNCGraded} (or closely related Hochschild cohomology techniques \cite{AvramovIyengarLipmanNayak,NeemanTheRelationBetweenGDandHH}), also see \cite{JiangGrothendieckDualityViaDiagonallySupported}. In this section we give an exposition of a slight variant of the latter: we explain that the dualizing complex of an affine variety $X$ can be characterized as the unit for a symmetric monoidal structure on $D^+_{coh}(X)$ called the $\stens$-product; the main ingredients for constructing this monoidal structure are the notion of exterior tensor products as well as $!$-pullback along closed immersions. Our main purpose for writing this exposition is to have a convenient reference for our later treatment (in \S \ref{ss:DualDiagFFin}) of the analogous constructions for arbitrary $F$-finite Noetherian rings.

Recall that given a $k$-algebra $S$ and $M,N \in D(S)$, we can construct the tensor product $M \otimes_S N \in D(S)$ as the pullback of $M \otimes_k N \in D(S \otimes_k S)$ along the diagonal $\mathrm{Spec}(S) \to \mathrm{Spec}(S \otimes_k S)$. The $\stens$-product is obtained by changing the meaning of pullback in the last sentence:

\begin{construction}[The $\stens$-product]
Fix a finite type $k$-algebra $S$. Let $p,q:S \to S \otimes_k S$ be the two tautological inclusions, and let $\Delta: S \otimes_k S \to S$ be the multiplication map. The pushforward $\Delta_*:D(S) \to D(S \otimes_k S)$ has a right adjoint $\Delta^!$ given concretely by $\RHom_{S \otimes_k S}(S,-)$. We then obtain a functor
\[ D(S) \times D(S) \xrightarrow{\stens} D(S) \quad \text{via} \quad M \stens_S N := \Delta^!(p^* M \otimes_{S \otimes_k S} q^* N) := \RHom_{S \otimes_k S}(S, M \otimes_k N).\]
For $M,N \in D^+_{coh}(S)$, we have $M \stens_S N \in D^+_{coh}(S)$ as $S \in D^b_{coh}(S \otimes_k S)$.
\end{construction}

\begin{remark}[Interpreting $\stens$-products via the completion]
Fix a finite type $k$-algebra $S$. For future use, let us remark that the $\stens$-product can be computed using the completion $S \widehat{\otimes_k} S$ of $S \otimes_k S$ along the diagonal, i.e., for $M,N \in D^+_{coh}(S)$, we have
%\todo{{Karl:} I changed an $R$ to a {\color{blue} $S$}.}
\[ 
M \stens_S N \simeq \RHom_{S \widehat{\otimes_k} S}({S}, M \widehat{\otimes_k} N),
\]
where $M \widehat{\otimes_k} N$ is the base change of $M \otimes_k N \in D(S \otimes_k S)$ along the completion map $S \otimes_k S \to S \widehat{\otimes_k} S$ (or equivalently $M \widehat{\otimes_k} N$ is the derived completion of $M \otimes_k N \in D(S \otimes_k S)$ along the ideal $I_\Delta$ of the diagonal). To see this equivalence, we use the following: if $A$ is any Noetherian ring with an ideal $J$ and $J$-adic completion $\widehat{A}$, for any $K,L \in D(A)$ with $K$ being $J$-torsion (i.e., $K$ in the image of $D(A/J) \to D(A)$) and $L \in D^+_{coh}(A)$, we claim that the natural maps give isomorphisms
\[ \RHom_A(K,L) \simeq \RHom_A(K, L \otimes_A \widehat{A}) \simeq \RHom_{\widehat{A}}(K \otimes_A \widehat{A},L \otimes_A \widehat{A}) \simeq \RHom_{\widehat{A}}(K,  L \otimes_A \widehat{A}). \]
For the first isomorphism, we use the identification of $L \otimes_A \widehat{A}$ with the derived $J$-completion of $L$ (by the coherence assumption on $L$ as well as the finite cohomological dimensionality of the derived $J$-completion functor) as well as the fact that $\RHom_A(K,-)$ annihilates the cone of the derived $J$-completion map as $K$ is $J$-torsion. The second isomorphism comes from adjunction, and the third comes from the isomorphism $K \otimes_A \widehat{A} \simeq K$ as $K$ is $J$-torsion.
\end{remark}

Let us explain why the $\stens$-product is symmetric monoidal.

\begin{lemma}[Monoidality of the $\stens$-product]
\label{MonoidalClassical}
The functor $D^+_{coh}(S) \times D^+_{coh}(S) \xrightarrow{\stens} D^+_{coh}(S)$ naturally defines a (a priori non-unital) $S$-linear symmetric monoidal structure on $D^+_{coh}(S)$.
\end{lemma}
\begin{proof}
Given a finite set $T$, we define an $S$-linear functor
\[ 
D^+_{coh}(S)^T \to D^+_{coh}(S) \quad \text{via} \quad \{M_t\}_{t \in T} \mapsto \bigstens{S, t \in T} M_t  := \RHom_{S^{\otimes_k T}}(S, \bigotimes_{k, t \in T} M_t), 
\]
where $S^{\otimes_k T}$ is the tensor product (over $k$) of copies of $S$ parametrized by $T$; note that this construction is clearly symmetric in the input data. Given any partition $T = I \sqcup J$, we shall construct a natural map
\[ \eta:\left(\bigstens{S, i \in I} M_i\right) \stens_S \left(\bigstens{S, j \in J} M_j \right) \to \bigstens{S, t \in T} M_t \]
which is an isomorphism. The desired monoidal properties will be immediate from the construction. To define this map, we first observe that taking external products of homomorphisms gives a natural map
\[ \eta':\RHom_{S^{\otimes_k I}}(S, \bigotimes_{k, i \in I} M_i) \otimes_k \RHom_{S^{\otimes_k J}}(S, \bigotimes_{k, j \in J} M_j)  \to  \RHom_{S^{\otimes_k I} \otimes_k S^{\otimes_k J}}(S \otimes_k S, \bigotimes_{k, t \in T} M_t) \]
in $D(S \otimes_k S)$, where $S \otimes_k S$ acts on the left with the first (resp. second) copy of $S$ acting on the first (resp. second) tensor factor. Moreover, this map is an isomorphism by \autoref{ExtProductRHom} below. Applying $\Delta^! = \RHom_{S \otimes_k S}(S,-)$ then gives the desired map $\eta$ once we use the functoriality of $!$-pullback along the composition of surjections $ S^{\otimes_k I} \otimes_k S^{\otimes_k J} \to S \otimes_k S \to S$ to simplify the target.
\end{proof}

The key lemma powering the above calculation (as well as some subsequent ones) is the following:

\begin{lemma}[Exterior products of $\RHom$'s]
\label{ExtProductRHom}
Let $R,S$ be finite type $k$-algebras. Fix $M,N \in D^+_{coh}(R)$ and $M',N' \in D^+_{coh}(S)$ with both $M$ and $M'$ bounded. Then the natural map
\[ c_{M,N,M',N'}:\RHom_R(M,N) \otimes_k \RHom_S(M',N') \to \RHom_{R \otimes_k S}(M \otimes_k M', N \otimes_k N') \]
is an isomorphism.
\end{lemma}
\begin{proof}
By shifting, we may assume $N,N' \in D^{\geq 0}$, and that $M,M' \in D^{\leq 0}$. Also, the claim is clear if $M$ and $M'$ are perfect as this reduces to the case $M=R$ and $M'=S$. In general, for any positive integer $n$, by truncating a complex of finite free $R$-modules representing $M$, we can find a perfect $R$-complex $P \in D^{\leq 0}$ and a map $P \to M$ whose cone lies in $D^{\leq -n}$; similarly we can find a perfect $S$-complex $P' \in D^{\leq 0}$ and a map $P' \to M'$ whose cone lies in $D^{\leq -n}$. Using the claim for $P$ and $P'$ and using the natural map $c_{M,N,M',N'} \to c_{P,N,P',N'}$, we learn that the fibre of $c_{M,N,M',N'}$ lies in $D^{\geq n}$. As this holds true for all $n$, the cone must be $0$, so we win.
\end{proof}

\begin{example}[$\stens$-products for smooth $k$-algebras]
\label{smoothstens}
Say $S$ is a smooth $k$-algebra. Then $\Perf(S) = D^b_{coh}(S)$ by regularity of $S$. Moreover, by smoothness, $S$ is perfect over $S \otimes_k S$, and hence the functor $\RHom_{S \otimes_k S}(S,-)$ preserves perfectness. Consequently the $\stens$-product restricts to an $S$-linear symmetric monoidal structure on $\Perf(S)$. By $S$-linearity, we must have a natural isomorphism
\[ M \stens_S N \simeq M \otimes_S (S \stens_S S) \otimes_S N.\]
In other words, the $\stens$-product on $\Perf(S)$ is determined by the object $S \stens_S S \in \Perf(S)$. Using \autoref{dualregular} and the standard identification $I_\Delta/I_\Delta^2 \simeq \Omega^1_{S/k}$ where $I_\Delta = \ker(S \otimes_k S \to S)$, we can also compute that
\[ S \stens_S S = \RHom_{S \otimes_k S}(S, S \otimes_k S) \simeq \det(\Omega^1_{S/k})^\vee[-\dim(S)], \]
thus obtaining a complete description of the $\stens$-product in this case: it differs from the usual $\otimes$-product by translating by an invertible object $\omegacan S := \det(\Omega^1_{S/k})[\dim(S)]$. In particular, the $\stens$-product is unital: $\omegacan S$ with the isomorphism $\tau:\omegacan S \stens_S \omegacan S \simeq \omegacan S$ coming from the preceding considerations provides a unit. 
\end{example}

The following shall be useful in reducing general statements about $\stens$-products to the smooth case:

\begin{proposition}[Finite pullbacks]
\label{FinPull}
Let $f:R \to S$ be a finite map of finite type $k$-algebras. The functor $f^!:D^+_{coh}(R) \to D^+_{coh}(S)$ is naturally (non-unitally) symmetric monoidal for the $\stens$-product. 

Moreover, if $D^+_{coh}(R)$ admits a $\stens$-unit $\omega_R^\bullet$, then $f^! \omega_R^\bullet$ is a $\stens$-unit in $D^+_{coh}(S)$.
\end{proposition}

We shall show later that $D^+_{coh}(R)$ always admits a $\stens$-unit (\autoref{DualClassicalExists}). However, the proof of existence uses a special case of the second assertion above, which is the reason for the present formulation.

\begin{proof}
For the first part, given $M,N \in D^+_{coh}(R)$, we must show that there is a natural isomorphism
\[ f^! M \stens_S f^! N \simeq f^!(M \stens_R N)\]
compatible with the associativity and commutativity constraints. Unwinding definitions, we must construct a natural isomorphism
\[ \eta:\RHom_{S \otimes_k S}(S, \RHom_R(S,M) \otimes_k \RHom_R(S,N)) \simeq \RHom_R(S, \RHom_{R \otimes_k R}(R, M \otimes_k N)) \]
in $D(S)$ compatible with the associativity and commutativity constraints. Using \autoref{ExtProductRHom}, we have
\[ \RHom_R(S,M) \otimes_k \RHom_R(S,N) \simeq \RHom_{R \otimes_k R}(S \otimes_k S, M \otimes_k N) \]
in $D(S \otimes_k S)$. The source of $\eta$ is thus the $!$-pullback of $M \otimes_k N \in D^+_{coh}(R \otimes_k R)$ along the composition $R \otimes_k R \to S \otimes_k S \to S$. On the other hand, the target of $\eta$ is the $!$-pullback of $M \otimes_k N$ along the composition $R \otimes_k R \to R\to S$ by definition. As these two compositions give the same composite map $R \otimes_k R \to S$, we obtain the desired identification $\eta$; we leave it to the reader to check the desired compatibilities.

For the second part, assume $D^+_{coh}(R)$ admits a $\stens$-unit $\omega_R^\bullet$. Write $\omega_S^\bullet = f^! \omega_R^\bullet$ and fix $M \in D^+_{coh}(S)$. We must show that there is a natural isomorphism $M \stens \omega_S^\bullet \simeq M$. Consider the natural maps
\[ 
    \begin{array}{rl}
    M \stens_S \omegacan S := & \RHom_{S \otimes_k S}(S, M \otimes_k \RHom_R(S,\omegacan R)) \\ \xrightarrow{a} & \RHom_{S \otimes_k R}(S, M \otimes_k \omegacan R) \\
        \xrightarrow{b} & \RHom_{R \otimes_k R}(R, M \otimes_k \omegacan R) 
         =: M \stens_R \omegacan R \simeq M,
    \end{array}
\]
where the map $a$ comes from restricting scalars along $S \otimes_k R \to S \otimes_k S$ and post-composing with the evaluation at $1$ map $\RHom_R(S,\omega_R^\bullet) \to \omega_R^\bullet$,  the map $b$ comes from restricting scalars along $R \otimes_k R \to S \otimes_k R$ and pre-composing with $R \to S$, and the last isomorphism comes from $\omega_R^\bullet$ being a $\stens$-unit in $D^+_{coh}(R)$. We claim both $a$ and $b$ are isomorphisms. In fact, the isomorphy of $b$ is the standard base change identification: the map $S \otimes_k R \to S$ is the pushout of $R \otimes_k R \to R$ along the natural map $R \otimes_k R \to S \otimes_k R$. For $a$, observe that the middle term $\RHom_{S \otimes_k R}(S, M \otimes_k \omegacan R)$ above is the $!$-pullback of $M \otimes_k \omegacan R \in D^+_{coh}(S \otimes_k R)$ along $S \otimes_k R \to S$. Factoring this map as $S \otimes_k R \to S \otimes_k S \to S$ then shows that
\[ \RHom_{S \otimes_k R}(S, M \otimes_k \omegacan R) \simeq \RHom_{S \otimes_k S}(S, \RHom_{S \otimes_k R}(S \otimes_k S, M \otimes_k \omegacan R)),\]
which simplifies to $\RHom_{S \otimes_k S}(S, M \otimes_k \RHom_R(S,\omegacan R))$ using \autoref{ExtProductRHom}, as wanted.
\end{proof}

We can now prove the promised theorem: the $\stens$-product is unital.

\begin{theorem}[Existence of $\stens$-units]
\label{DualClassicalExists}
Let $S$ be a finite type $k$-algebra. Then the $\stens$-product on $D^+_{coh}(S)$ admits a unit $\omegacan S$. The unit $\omegacan S$ lies in $D^b_{coh}(S)$ and in fact has finite injective dimension.
\end{theorem}
\begin{proof}
When $S$ is smooth, everything follows from \autoref{smoothstens}. In the general case, since the data of a unit in a symmetric monoidal category is unique up to unique isomorphism\footnote{A unit in a symmetric monoidal category $(\mathcal{C},\otimes)$ is given by an object $u \in \mathcal{C}$ and an isomorphism $\tau:u \otimes u \simeq u$ with the additional property that $- \otimes u$ is fully faithful. Indeed, it is clear that any unit comes with a unique such $\tau$. Conversely, given $(u,\tau)$ as before, any $x \in \mathcal{C}$ comes equipped with a natural isomorphism $a_{x,u}:x \simeq x \otimes u$ obtained by applying the full faithfulness of $- \otimes u$ to the isomorphism $\mathrm{id}_x \otimes \tau: x \otimes u \simeq x \otimes (u \otimes u) \simeq (x \otimes u) \otimes u$. Moreover, the pair $(u,\tau)$ is unique up to unique isomorphism when it exists: applying the previous reasoning to both $u$ and $u'$ gives an isomorphism $a_{u,u'}^{-1} \circ a_{u',u}:u' \simeq u$ that carries $\tau$ to $\tau'$.}, we may make choices to construct a unit. In this case, if we choose a surjection $f:R \to S$ with $R$ smooth over $k$, then \autoref{FinPull} shows that $\omegacan S := f^! \omegacan R$ is a unit for $(D^+_{coh}(S), \stens_S)$. It remains to show that the unit $\omega_S^\bullet$ has finite injective dimension. But this follows from its definition as $f^! \omega_R^\bullet$ combined with the observation that $f^!$ preserves injectivity (being right adjoint to the exact functor $f_*$). 
\end{proof}

As an application, we obtain the following well-known characterization of dualizing complexes (due to Van den Bergh and Yekutieli):

\begin{corollary}[Uniqueness of rigidified dualizing complexes]
\label{CharactDual}
Let $S$ be a finite type $k$-algebra. There is a unique (up to unique isomorphism) pair $(\omegacan S, \tau)$ where:
\begin{itemize}
\item $\omegacan S \in D^b_{coh}(S)$ has finite injective dimension such that the natural map gives an isomorphism $S \simeq \RHom_S(\omegacan S,\omegacan S)$.
\item $\tau$ is an isomorphism $\tau:\omegacan S \simeq \omegacan S \stens_S \omegacan S := \RHom_{S \otimes_k S}(S,\omegacan S \otimes_k \omegacan S)$.
\end{itemize}
\end{corollary}
\begin{proof}
We already know such a pair $(\omegacan S,\tau)$ exists thanks to \autoref{DualClassicalExists}. Fix one such object once and for all. We must show its uniqueness. This is a standard argument but we sketch a proof as we could not find a reference; our arguments heavily rely on those in the Stacks Project, especially \cite[Tag 0A7E, 0A7F]{stacks-project}. 

Say $(\omega',\tau')$ is another pair as in the corollary. Our goal is to construct a unique isomorphism $\omega' \simeq \omegacan S$ carrying $\tau$ to $\tau'$. Write $F(-) = \RHom_S(-,\omegacan S)$ and  $G(-):= \RHom_S(-,\omega')$ for the displayed functors. 
    
First, let us explain some general properties of $G(-)$ (and thus also $F(-)$). As $\omega'$ has finite injective dimension, the functor $G(-)$ preserves $D^b_{coh}(-)$ and moreover we have a natural map $\eta_M:M \to G^2(M)$. We claim that the latter is an equivalence; this implies that $G$ is a contravariant autoequivalence of $D^b_{coh}(S)$. The isomorphy of $\eta_M$ for $M$ perfect reduces to the case of $M=S$, which holds true by assumption. To extend to arbitrary objects in $D^b_{coh}$, we use approximation. Given $M \in D^b_{coh}(S)$ and an integer $n$, we can find $P \in \Perf(S)$ and a map $P \to M$ whose cone is $n$-connected. By the finite injective dimension assumption on $\omega' \in D^b_{coh}(S)$, it follows that $G^2(P) \to G^2(M)$ has an $(n-c)$-connected cone for some $c$ depending only on $\omega'$ and independent of $n$. Comparing $\eta_M$ with $\eta_P$ then shows that the cone of $\eta_M$ is $(n-c-1)$-connected for all $n$, and thus $0$.

The functor $F(-)$ is a contravariant autoequivalence by the previous paragraph as well. Consider the covariant autoequivalence $GF:D^b_{coh}(S) \to D^b_{coh}(S)$. As this functor is an $S$-linear autoequivalence, we must have $GF(M) \simeq  M \otimes_S GF(S)$ (see \cite[Tag 0A7E]{stacks-project}). Moreover, again because this functor is an equivalence, the object $GF(S)$ must have the form $L[n]$ for $L \in \mathrm{Pic}(S)$ and some integer $n$, so we learn that $\RHom_S(\omega',\omegacan S) \simeq L[n]$. As $F(-) = \RHom_S(-,\omegacan S)$ is an equivalence with $F^2 = \mathrm{id}$, we learn that 
\[ \omega' \simeq \RHom_S(L[n], \omegacan S) \simeq L^\vee[-n] \otimes_S \omegacan S.\]
Write $L^\vee[-n] = M$ for simplicity. We claim that the rigidifying data forces $M=S$. Indeed, the $\stens$-product is $\Perf(S)$-linear in each factor, 
so we have a natural isomorphism.
\[ \omega' \stens_S \omega' \simeq M \otimes_S (\omegacan S \stens_S \omegacan S) \otimes_S M.\]
Using the rigidifying data $\tau'$ and $\tau$, this collapses to give
\[ \omega' \simeq  M^{\otimes 2}  \otimes_S \omegacan S,\]
thus giving
\[ M \otimes_S \omegacan S \simeq M^{\otimes 2} \otimes_S \omegacan S.\]
Applying $F(-) = \RHom_S(-,\omegacan S)$ and using $F(\omegacan S) = S$ as well as the $\Perf(S)$-linearity, we learn that
\[ M^\vee \simeq (M^{\otimes 2})^\vee.\]
As $M$ is an invertible object of $\Perf(S)$, this forces $M \simeq S$, as wanted.
\end{proof}

%%%%%%%%%%%%%%%%%%%%%%%%%%%%%%%%%%%%%%%%%%%%%%%%%%%%%%%%%%%%%%%%%%%%%%%%
%%%%%%%%%%%%%%%%%%%%%%%%%%%%%%%%%%%%%%%%%%%%%%%%%%%%%%%%%%%%%%%%%%%%%%%%
%%%%%%%%%%%%%%%%SUBSECTION:  F-Finite dualizing via diagonal%%%%%%%%%%%%
%%%%%%%%%%%%%%%%%%%%%%%%%%%%%%%%%%%%%%%%%%%%%%%%%%%%%%%%%%%%%%%%%%%%%%%%
%%%%%%%%%%%%%%%%%%%%%%%%%%%%%%%%%%%%%%%%%%%%%%%%%%%%%%%%%%%%%%%%%%%%%%%%

\subsection{Dualizing complexes over $F$-finite rings via the diagonal}
\label{ss:DualDiagFFin}

In this section, we finally complete the proof of our main result.  To do this, we explain how to construct the monoidal $\stens$-product for complexes over an arbitrary $F$-finite Noetherian ring $S$; the essential ingredient is a good notion of completion of $S \otimes_{\mathbb{F}_p} S$ along the diagonal, which we obtain via the material on completions in \S \ref{sec:CompRev} as well as Gabber's theorem \cite{Gabber_someTstructures}, \S \ref{sec:gabberconstr}. Once we have the $\stens$-product, we redo the arguments in \S \ref{DualCompStensClassical} to construct the dualizing complex as the unit.

\subsubsection*{The $\stens$-product for $F$-finite rings}

We now mimic the material in \S \ref{DualCompStensClassical} to construct the $\stens$-product on the derived category of a Noetherian $F$-finite ring, and also to prove it is unital. First, let us define a variant of the external tensor product, resulting from the notion of completion explained in \S \ref{sec:CompRev}.

\begin{construction}[The completed external product]
\label{stensffindef}
Let $S$ be a Noetherian $F$-finite ring. Let $p,q:S \to S \widehat{\otimes_{\mathbb{F}_p}} S$ be the two tautological inclusions. For $M,N \in D(S)$, we define
\[ M \widehat{\otimes_{\mathbb{F}_p}} N = p^* M  \otimes_{S \widehat{\otimes_{\mathbb{F}_p}} S} q^* N \in D( S \widehat{\otimes_{\mathbb{F}_p}} S),\]
i.e.,  $M \widehat{\otimes_{\mathbb{F}_p}} N \in D(S \widehat{\otimes_{\mathbb{F}_p}} S)$ is the base change of $M \otimes_{\mathbb{F}_p} N \in D(S \otimes_{\mathbb{F}_p} S)$ along $S \otimes_{\mathbb{F}_p} S \to S \widehat{\otimes_{\mathbb{F}_p}} S$. (We warn the reader that, despite the notation, the object $M \widehat{\otimes_{\mathbb{F}_p}} N$ is not complete for the derived $J_\Delta$-adic filtration on $S \widehat{\otimes_{\mathbb{F}_p}} S \to S$ in general; however, this is the case when $M$ and $N$ lie in $D^+_{coh}(S)$, which is the only relevant case for the sequel.)
\end{construction}

To ensure that the arguments in \S \ref{DualCompStensClassical} go through in our setting, we need the following:

\begin{lemma}[Boundedness of completed external products]
\label{BoundExtProduct}
Let $S$ be a Noetherian $F$-finite ring. There exists some integer $c > 0$ such that if $M,N \in D^{\geq 0}_{coh}(S)$, then $M \widehat{\otimes_{\mathbb{F}_p}} N \in D^{\geq -c}_{coh}(S \widehat{\otimes_{\mathbb{F}_p}} S)$. 
 \end{lemma}

In the classical case where $S$ is finite type over a perfect field $k$, the constant $c$ can be taken to be $0$: indeed, in this case, we have $M \widehat{\otimes_{\mathbb{F}_p}} N \simeq (M \otimes_k N) \otimes_{(S \otimes_k S)} (S \widehat{\otimes_{\mathbb{F}_p}} S)$, so the claim follows as $S \widehat{\otimes_{\mathbb{F}_p}} S$ is flat over $S \otimes_k S$ (\autoref{GeomCompDiag}).
But we do not know if $c$ can be taken to be $0$ in general; the proof below gives a much weaker bound.

\begin{proof}
By Gabber's theorem, we can write $S = R/I$, where $R$ is a Noetherian $F$-finite regular ring and $I \subset R$ is an ideal such that $R$ is $I$-adically complete. We shall check that $c=2\dim(R)$ works. Fix $M,N \in D^{\geq 0}_{coh}(S)$. By \autoref{PushoutFFin}, the  square
\[ \xymatrix{ R \otimes_{\mathbb{F}_p} R \ar[r]^c \ar[d]^a & R \widehat{\otimes_{\mathbb{F}_p}} R \ar[d]^b \\
S \otimes_{\mathbb{F}_p} S \ar[r]^d & S \widehat{\otimes_{\mathbb{F}_p}} S  }\]
is a pushout square of animated rings with the vertical maps realizing the target as perfect complexes over the source. The condition of lying in $D^{\geq -c}_{coh}(-)$ can be checked after restricting scalars along such maps, so it suffices to show that $b_*\left(M \widehat{\otimes_{\mathbb{F}_p}} N\right)  \in D(R \widehat{\otimes_{\mathbb{F}_p}} R)$ lies in $D^{\geq -c}_{coh}$. But we also have $M \widehat{\otimes_{\mathbb{F}_p}} N = d^*(M \otimes_{\mathbb{F}_p} N)$. By base change, we are then reduced to checking that $c^* a_*(M \otimes_{\mathbb{F}_p} N) \in D^{\geq -c}_{coh}$. Now $a_*(M \otimes_{\mathbb{F}_p} N) \simeq a_* M \otimes_{\mathbb{F}_p} a_* N$ lies in $D^{\geq 0}_{coh}(R \otimes_{\mathbb{F}_p} R)$ (as $M$ and $N$ lie in $D^+_{coh}(R)$), so it suffices to prove the more general statement that $c^*(M' \otimes_{\mathbb{F}_p} N') \in D^{\geq -c}_{coh}$ for $M',N' \in D^{\geq 0}_{coh}$. As the construction $M',N' \mapsto M' \otimes_{\mathbb{F}_p} N'$ is $t$-exact in each variable separately, by consideration of the canonical filtration, it suffices to show that if $M'$ and $N'$ are finitely presented $R$-modules, then $c^*(M \otimes_{\mathbb{F}_p} N) \in D^{\geq -c}$. But this is clear from the regularity of $R$: the global dimension of $M'$ and $N'$ individually is bounded above by $\dim(R)$, and hence $M' \otimes_{\mathbb{F}_p} N'$ has global dimension $\leq c=2\dim(R)$.
\end{proof}

Using the above, we obtain:

\begin{lemma}[Exterior products of $\RHom$'s]
\label{ExtProductRHomFFin}
Let $R$ be a Noetherian $F$-finite ring. Fix $M,N,M',N' \in D^+_{coh}(R)$ with both $M$ and $M'$ bounded. Then the natural map
\[ c_{M,N,M',N'}:\RHom_R(M,N) \widehat{\otimes_{\mathbb{F}_p}} \RHom_R(M',N') \to \RHom_{R  \widehat{\otimes_{\mathbb{F}_p}} R}(M  \widehat{\otimes_{\mathbb{F}_p}} M', N  \widehat{\otimes_{\mathbb{F}_p}} N') \]
is an isomorphism.
\end{lemma}
\begin{proof}
This follows by the same approximation argument used in \autoref{ExtProductRHom}, using \autoref{BoundExtProduct} to control the left term.
\end{proof}

The $\stens$-product is defined just as in the classical case, replacing the abstract tensor product with a completed one:

\begin{construction}[The $\stens$-product]
\label{stensFFin}
Let $S$ be a Noetherian $F$-finite ring. Let $p,q:S \to S \widehat{\otimes_{\mathbb{F}_p}} S$ be the two tautological inclusions, and let $\Delta: S \widehat{\otimes_{\mathbb{F}_p}} S \to S$ be the multiplication map. The pushforward $\Delta_*:D(S) \to D(S \widehat{\otimes_{\mathbb{F}_p}} S)$ has a right adjoint $\Delta^!$ given concretely by $\RHom_{S \widehat{\otimes_{\mathbb{F}_p}} S}(S,-)$. We then obtain a functor
\[ D(S) \times D(S) \xrightarrow{\stens} D(S) \quad \text{via} \quad M \stens_S N := \Delta^!(p^* M \otimes_{S \widehat{\otimes_{\mathbb{F}_p}} S} q^* N) = \RHom_{S \widehat{\otimes_{\mathbb{F}_p}} S}(S, M \widehat{\otimes_{\mathbb{F}_p}} N),\]
This functor is $\Perf(S)$-linear in each variable separately by construction.
\end{construction}
Using the control on the completion provided by \autoref{FfinCompDiag} as well as \autoref{BoundExtProduct}, we learn that the $\stens$-product preserves $D^+_{coh}$, as in the classical case:

\begin{proposition}
\label{stenscoh}
Let $S$ be a Noetherian $F$-finite ring. For $M,N \in D^+_{coh}(S)$, we have $M \stens_S N \in D^+_{coh}(S)$. 
\end{proposition}
\begin{proof}
By \autoref{FfinCompDiag}, the animated ring $S \widehat{\otimes_{\mathbb{F}_p}} S$ is classical and Noetherian, so the quotient $S$ is pseudocoherent, whence $\RHom_{S \widehat{\otimes_{\mathbb{F}_p}} S}(S,-)$ preserves $D^+_{coh}(-)$. It is thus enough to check that for $M,N \in D^+_{coh}(S)$, we have $M \widehat{\otimes_{\mathbb{F}_p}} N \in D^+_{coh}(S \widehat{\otimes_{\mathbb{F}_p}} S)$, which follows from \autoref{BoundExtProduct}.
\end{proof}

This proposition implies that $\stens_S$ defines a functor $D^+_{coh}(S)^{\times 2} \to D^+_{coh}(S)$. This underlies a symmetric monoidal structure:

\begin{lemma}[Monoidality of the $\stens$-product]
\label{MonoidalFFin}
For a Noetherian $F$-finite ring $S$, the functor $D^+_{coh}(S) \times D^+_{coh}(S) \xrightarrow{\stens} D^+_{coh}(S)$ naturally defines a (a priori non-unital) $S$-linear symmetric monoidal structure on $D^+_{coh}(S)$.
\end{lemma}
\begin{proof}
This follows by the same argument as in \autoref{MonoidalClassical} using (a multivariable version of) \autoref{ExtProductRHomFFin} instead of \autoref{ExtProductRHom}.
\end{proof}

This symmetric monoidal structure interacts well with $!$-pullback, as in \autoref{FinPull} in the classical case:

\begin{proposition}[Finite pullbacks]
\label{FFinPull}
Let $R \to S$ be a finite map of Noetherian $F$-finite rings. Then the functor $f^!:D^+_{coh}(R) \to D^+_{coh}(S)$ is naturally (non-unitally) symmetric monoidal for the $\stens$-product. 

Moreover, if there exists a $\stens$-unit $\omega_R^\bullet \in D^+_{coh}(R)$, then $f^! \omega_R^\bullet \in D^+_{coh}(S)$ is also a $\stens$-unit.
\end{proposition}

The proof below simplifies significantly in the case where $R \to S$ is surjective (which is the case relevant for \autoref{mainthmdualcanffin}) or more generally when $S \otimes_R S \to S$ is a universal homeomorphism (which is the case for the Frobenius in \autoref{FrobCompatCanDual}): the map labelled $c$ in the proof below is an isomorphism by \autoref{PushoutFFin}.

\begin{proof}
For the first part, fix $M,N \in D^+_{coh}(R)$. Our goal is to construct a natural isomorphism $f^! M \stens_S f^! N \simeq f^!(M \stens_R N)$. We shall do so imitating the argument in \autoref{FinPull}, though the completions create some extra complications. Consider the following commutative diagram of animated rings
\[ \xymatrix{ R \otimes_{\mathbb{F}_p} R \ar[r] \ar[d] & R \widehat{\otimes_{\mathbb{F}_p}} R \ar[d] \ar[r] & R \ar[d] & \\
S \otimes_{\mathbb{F}_p} S \ar[r] & A \ar[r]^a \ar[rd]^c & S \otimes^L_R S \ar[rd]^b & \\
& & S \widehat{\otimes_{\mathbb{F}_p}} S \ar[r]^d & S} \]
where the square on the top left is a pushout square defining $A$, and the outer square in the top row a pushout square for general reasons (whence the second square in the top row is also a pushout square). 
%\todo{\emph{Karl:}  Could we say what category these push    out squares are in?  Animated rings?  }
Consider the object 
\[ K :=  \RHom_{R \widehat{\otimes_{\mathbb{F}_p}} R}(A, M \widehat{\otimes_{\mathbb{F}_p}} N) \simeq \RHom_R(S,M) \widehat{\otimes_{\mathbb{F}_p}} \RHom_R(S,N) \in D(A),\]
where the isomorphism results from the analog of \autoref{ExtProductRHom} over $R$. Following the proof of \autoref{FinPull}, our task is to identify $b^! a^! K$ with $d^! c^* K$ (and not $d^! c^! K$, which would be obvious); this is clear when $c$ is an isomorphism (e.g., if $R \to S$ surjective, by \autoref{PushoutFFin}). In general, we shall use Noetherianness. Specifically, by first writing $R$ as a quotient of a regular ring $R_0$, we learn that $A$ is a perfect complex over the Noetherian regular ring $R_0 \widehat{\otimes_{\mathbb{F}_p}} R_0$, and thus $A$ is a derived Noetherian ring with only finitely many homotopy groups. Next, the natural map $L_{S/\mathbb{F}_p}[1] = L_{S/S \otimes_{\mathbb{F}_p} S} \to L_{S/A}$ is an isomorphism: this follows by the transitivity triangle for $A \to S \otimes^L_R S \to S$ and the calculations $L_{(S \otimes^L_R S)/A} \simeq L_{R/R \widehat{\otimes_{\mathbb{F}_p}} R} \simeq L_{R/\mathbb{F}_p}[1]$, $L_{S/S \otimes_R S} \simeq L_{S/R}[1]$. It follows from contemplating the derived $I$-adic filtrations that 
\[ S \widehat{\otimes_{\mathbb{F}_p}} S := \mathrm{Comp}(S \otimes_{\mathbb{F}_p} S \to S) \to \mathrm{Comp}(A \to S) \]
is an isomorphism, i.e., the map $c$ realizes the target as the Adams completion of $A$ along $A \to S$. Since $A$ is bounded, \autoref{ClassicalAdams} also identifies $A \to \mathrm{Comp}(A \to S)$ with the derived $J$-completion of $A$, where $J \subset \pi_0(A) = \pi_0(S \otimes^L_R S) $ is the (necessarily finitely generated, by Noetherianness) kernel of $\pi_0(A) \to S$. But for any $T \in D^b(A)$ with $H^*(T)$ annihilated by $J$, the functor $\RHom_A(T,-)$ annihilates the cone of the derived $J$-completion map; applying this with $T=S$ shows that the natural map
\[ \RHom_A(S,K) \to \RHom_A(S, K^{\wedge}_J) \simeq \RHom_A(S, K \otimes_A (S \widehat{\otimes_{\mathbb{F}_p}} S))\]
is an isomorphism, where $K^{\wedge}_J$ denotes the derived $J$-completion of $K$ and the second isomorphism results from the pseudocoherence of $K$. Unwinding definitions, this exactly gives that $b^! a^! K \simeq d^! c^* K$, as wanted.

Now assume we are given a $\stens$-unit $\omega_R^\bullet \in D^+_{coh}(R)$; write $\omega_S^\bullet = f^! \omega_R^\bullet$. For $M \in D^+_{coh}(S)$, we must construct a natural isomorphism $M \simeq M \stens \omega_S^\bullet$. This will follow by the argument in \autoref{FinPull}. More precisely, we have the following sequence of isomorphisms
\begin{align*}
M \simeq M \stens \omega_R^\bullet  := \RHom_{R \widehat{\otimes_{\mathbb{F}_p}} R}(R, M \widehat{\otimes_{\mathbb{F}_p}} \omega_R^\bullet) \\
\simeq \RHom_{S \widehat{\otimes_{\bF_p}} R}(S, M \widehat{\otimes_{\mathbb{F}_p}} \omega_R^\bullet) \\
\simeq \RHom_{S \widehat{\otimes_{\bF_p}} S}(S, \RHom_{S \widehat{\otimes_{\bF_p}} R}(S \widehat{\otimes_{\bF_p}} S, M \widehat{\otimes_{\mathbb{F}_p}} \omega_R^\bullet)) \\
\simeq \RHom_{S \widehat{\otimes_{\bF_p}} S}(S, M \widehat{\otimes_{\mathbb{F}_p}} \omega_S^\bullet) =: M \stens \omega_S^\bullet
\end{align*}
where the isomorphism in the first line comes from $\omega_R^\bullet$ being a $\stens$-unit in $D^+_{coh}(R)$, the one in the second line comes from the standard base change adjunction for the second square in the diagram 
\[ \xymatrix{
R \ar[r] \ar[d] & R \widehat{\otimes_{\mathbb{F}_p}} R \ar[r] \ar[d] & R \ar[d] \\
S \ar[r] & S \widehat{\otimes_{\mathbb{F}_p}} R \ar[r] & S}
\]
where both squares are (by definition) pushout squares of animated rings, the isomorphism in the third line comes from writing $!$-pullback along the composition
\[ S \widehat{\otimes_{\mathbb{F}_p}} R \to S \widehat{\otimes_{\mathbb{F}_p}} S \to S \]
as the composition of the corresponding $!$-pullbacks (where $!$-pullback always means the right adjoint to pushforward), and the one in the last line comes from the exterior product compatibility from \autoref{ExtProductRHomFFin}.
\end{proof}

Our main theorem is that this operation is naturally a unital symmetric monoidal structure:

\begin{theorem}
\label{mainthmdualcanffin}
Let $S$ be a Noetherian $F$-finite ring. Then the $\stens$-product on $D^+_{coh}(S)$ is naturally symmetric monoidal and unital. The unit $\omega_S^\bullet$ lies in $D^b_{coh}(S)$ and has finite injective dimension.
\end{theorem}
\begin{proof}
First, \autoref{MonoidalFFin} gives the (a priori non-unital) symmetric monoidal structure. 

Next, we find the unit when $S$ is further assumed to be regular. In this case, note that Popescu's theorem and $F$-finiteness show that $\Omega^1_{S/\mathbb{F}_p} \simeq L_{S/\mathbb{F}_p}$ is a finite projective $S$-module. One then argues as in \autoref{smoothstens} to conclude that $\omegacan S := \mathrm{det}(\Omega^1_{S/\mathbb{F}_p})[\operatorname{rk} \Omega^1_{S/\mathbb{F}_p}]$ is the unit for $(\Perf(S),\stens_S)$; the essential calculation of the Fundamental Local Isomorphism (see \cite[III, Proposition 7.2]{HartshorneResidues}, \cite[I, Theorem 4.5]{AltmanKleimanGD},  \cite[Equation (2.5.2)]{ConradGDualityAndBaseChange}, or \autoref{dualregular}) applies directly in our case thanks to \autoref{regularFfindiag}.%\todo{{@Bhargav, do you mean the fundamental local isomorphism here?}}

Next, the general $F$-finite case follows from the regular case thanks to Gabber's theorem and the second part of \autoref{FFinPull}.

Finally, the assertions about the unit are deduced by exactly the same argument as in the classical case in \autoref{DualClassicalExists}.
\end{proof}

For completeness, let us record some other features of the canonical dualizing complex mimicking those in the classical case. First, we characterize the canonical dualizing complex:

\begin{corollary}[Characterizing the canonical dualizing complex]
\label{candualffin}
Let $S$ be a Noetherian $F$-finite ring. There is a unique (up to unique isomorphism) pair $(\omegacan S, \tau)$ where:
\begin{itemize}
\item $\omegacan S \in D^b_{coh}(S)$ has finite injective dimension and $S \simeq \RHom_S(\omegacan S,\omegacan S)$ via the natural map.
\item $\tau$ is a rigidifying isomorphism $\tau:\omegacan S \simeq \omegacan S \stens_S \omegacan S$.
\end{itemize}
\end{corollary}
\begin{proof}
This follows exactly like the corresponding classical analog in \autoref{CharactDual}.
\end{proof}

\begin{remark}
\label{rem.UnitForShriekTesnorIsCanDual}
The proof of \autoref{mainthmdualcanffin} makes it clear that the dualizing complex provided by \autoref{candualffin} agrees with the one we mentioned in the introduction \autoref{eq.CanonicalDualizingComplexForFFiniteRing} and the one we will see later in  \autoref{def.GeneralDefinitionOfCanonicalForFFiniteRings}.

\end{remark}

\begin{corollary}
\label{FrobCompatCanDual}
Let $S$ be a Noetherian $F$-finite ring with Frobenius ${F}:S \to S$. There is a canonical identification
\[  F^! \omegacan S \simeq \omegacan S \in D^b_{coh}(S)\]
%\todo{Karl:  I made "Frob" into "$F$" to be compatible with earlier.}
compatible with the rigidifying data (where $\omegacan S$ is the canonical dualizing complex from \autoref{candualffin}).
\end{corollary}

\subsection{The $\stens$-product for schemes}

We briefly sketch how to mimic the material in the previous section for not necessarily affine Noetherian $F$-finite schemes.

For a Noetherian $F$-finite ring $R$, we have defined its completed self-tensor product 
$R \widehat{\otimes_{\mathbb{F}_p}} R$ and shown (\autoref{FfinCompDiag}) it is a Noetherian $F$-finite ring that comes equipped with a surjection $R \widehat{\otimes_{\mathbb{F}_p}} R \to R$ along which it is complete; the relative cotangent complex of this surjection is $L_{R/\mathbb{F}_p}[1]$ by \autoref{cotcomplete}. In particular, as the formation of the cotangent complex is Zariski (or even \'etale) local, we can glue this construction together:

\begin{construction}[Completed self-product of Noetherian $F$-finite schemes]
Let $X$ be a Noetherian $F$-finite scheme. Write $X \widehat{\times_{\mathbf{F}_p}} X$ for the Noetherian formal scheme with underlying topological space $X$ obtained by gluing together $\mathrm{Spf}(R \widehat{\otimes_{\mathbb{F}_p}} R)$ for varying affine opens $\mathrm{Spec}(R) \subset X$. We then have a tautological closed immersion $i:X \to X \widehat{\times_{\mathbb{F}_p}} X$, the projections $p,q:X {\times_{\mathbb{F}_p}} X \to X$, and the completion map $\pi:X \widehat{\times_{\mathbb{F}_p}} X \to X \times_{\mathbb{F}_p} X$. Following \autoref{stensFFin} and \autoref{stenscoh}, the formula
\[ (M,N) \mapsto  M \stens N := i^! \pi^* (p^* M \otimes q^*N)\]
underlies a symmetric monoidal structure on $D^+_{coh}(X)$ with unit $\omega_X^\bullet$.
\end{construction}

We then have the following global version of \autoref{FFinPull}

\begin{proposition}[Proper pullbacks]
\label{FFinPullProper}
    Suppose $X \to Y$ is a proper map with $Y$ a Noetherian $F$-finite scheme.  Then the functor $f^! : D^+_{coh}(Y) \to D^+_{coh}(X)$ is naturally symmetric monoidal for the $\stens$-product.
\end{proposition}
\begin{proof}
The proof is essentially the same as \autoref{FFinPull}, except that certain evident statements about $!$-pullback for finite maps need slightly more justification in the proper case.

We shall first show that $f^!$ is non-unitally symmetric monoidal, and then check the preservation of the unit. As our constructions are natural, we may assume $Y=\mathrm{Spec}(R)$ is affine. Fix $M,N \in D^+_{coh}(R)$. Consider the following diagram:
\[ \xymatrix{
X \ar[r]^d \ar[dr]^b & X \widehat{\times_{\mathbf{F}_p}} X \ar[dr]^c & & \\
& X \times_R X \ar[r]^a \ar[d] & \widetilde{X \times_{\mathbf{F}_p} X} \ar[d]^-{\widetilde{f \times f}} \ar[r]^{\pi} & X \times_{\mathbf{F}_p} X \ar[d]^-{f \times f} \\
& \mathrm{Spec}(R) \ar[r] & \mathrm{Spec}(R \widehat{\otimes_{\mathbf{F}_p}} R) \ar[r]^-{\pi} & \mathrm{Spec}(R \otimes_{\mathbf{F}_p} R) 
}
\]
where the bottom two horizontal squares are cartesian in derived schemes, while the top left parallelogram is given by the natural maps. In particular, $\widetilde{X \times_{\mathbf{F}_p} X}$ is a proper derived scheme over the Noetherian $F$-finite ring $R \widehat{\otimes_{\mathbf{F}_p}} R$ and hence Noetherian.

Consider the object
\[ K := \widetilde{f \times f}^!\pi^* (M \otimes_{\mathbb{F}_p} N) \in D^+_{coh}(\widetilde{X \times_{\mathbf{F}_p} X}). \]
We first claim that there are natural isomorphisms
\begin{equation}
    \label{stens!pullproper}
\widetilde{f \times f}^!\pi^* (M \otimes_{\mathbb{F}_p} N) \simeq \pi^*(f \times f)^!(M \otimes_{\mathbb{F}_p} N) \simeq \pi^*(f^! M \otimes_{\mathbb{F}_p} f^! N).
\end{equation}
For the first isomorphism, we shall use the following base change result: if
\[ \xymatrix{ U \ar[r]^\alpha \ar[d]^q & V \ar[d]^q \\ W \ar[r]^\alpha & Z }\]
is a Cartesian diagram of derived schemes with $q$ proper and all structure sheaves being bounded, then we have a natural isomorphism
\[ \alpha^* q^! \simeq q^! \alpha^* \]
of functors $D^+_{coh}(Z) \to D^+_{coh}(U)$; this can be deduced from the fact that $q^!:D^{+}_{qc}(Z) \to D^+_{qc}(W)$ preserves colimits as its left adjoint $q_*$ preserves pseudocoherence by properness. For the second isomorphism in \autoref{stens!pullproper}, we claim more generally that $(f \times f)^! (M \otimes_{\mathbb{F}_p} N) \simeq f^! M \otimes_{\mathbb{F}_p} f^! N$, which follows by the same base change result applied twice by writing $f \times f$ as a composition of $(f \times \id) \circ (\id \times f)$.

To finish proving non-unital symmetric monoidality, it now suffices to show that $d^! c^! L \simeq d^! c^* L$ for any $L \in D^+_{coh}(\widetilde{X \times_{\mathbf{F}_p} X})$: indeed, applying $d^! c^!$ to the leftmost term in \autoref{stens!pullproper} gives $f^!(M \stens N)$, while applying $d^! c^*$ to the rightmost term of \autoref{stens!pullproper} gives $f^!M \stens f^!N$. By formal GAGA, we can replace $\widetilde{X \times_{\mathbf{F}_p} X}$ with its formal completion along $X \times_R X$. After making this replacement, the claim becomes local and follows by the same argument given in the affine case.

It remains to show that $f^!$ is unital, i.e., for $M \in D^+_{coh}(X)$, there is a natural isomorphism $M \stens f^! \omega_R^\bullet \simeq M$ (in fact, it suffices to do this for $M=\omega_X^\bullet$). For this, one checks that the argument in the affine case in \autoref{FFinPull}, translated to schemes, gives a natural isomorphism 
\[ R\Gamma(X,M) \simeq R\Gamma(X, f^! \omega_R^\bullet \stens M) \]
for all $M \in D^+_{coh}(X)$. Using the $\mathrm{Perf}(X)$-linearity of all functors involved, we obtain a natural isomorphism
\[ \mathrm{RHom}_X(K,M) \simeq \mathrm{RHom}_X(K, f^! \omega_R^\bullet \stens M)\]
for $K \in \mathrm{Perf}(X)$ by applying the previous isomorphism to $M \otimes_{\mathcal{O}_X} K^\vee$. Since $D_{qc}(X)$ is generated by $\mathrm{Perf}(X)$ under colimits, this yields\footnote{Given $N,N' \in D_{qc}(X)$, any natural isomorphism $\mathrm{RHom}_X(-,N) \simeq \mathrm{RHom}_X(-,N')$ of functors on $\mathrm{Perf}(X)$ comes uniquely from an isomorphism $N \simeq N'$: indeed, as $D_{qc}(X) = \mathrm{Ind}(\mathrm{Perf}(X))$, the claim follows by noting that for any small $\infty$-category $\mathcal{C}$, we can realize $\mathrm{Ind}(\mathcal{C})$ as the full subcategory of $\mathrm{PShv}(\mathcal{C}) =\mathrm{Fun}(\mathcal{C}^{op},\mathcal{S})$ consisting of filtered colimits of representable functors, so specifying a map between ind-objects is equivalent to specifying a map between the corresponding functors on $\mathcal{C}$.} a natural isomorphism $M \simeq f^! \omega_R^\bullet \stens M$, as wanted.
\end{proof}

As $f^!$ for a finite type map is defined by compactifying a separated morphism (say affine), by breaking up a general finite type morphism into separated pieces, this yields the following, which finishes the proof of \autoref{thmdualexistintro}.

\begin{corollary}
\label{cor.FiniteTypePullbacksViaMonoidal}
    Suppose $f : X \to Y$ is a finite type map of Noetherian $F$-finite schemes.   Then the functor $f^! : D^+_{\mathrm{coh}}(Y) \to D^{+}_{\mathrm{coh}}(X)$ satisfies 
    \[
        f^! \omega_Y^{\mydot} \simeq \omega_X^{\mydot}.
    \]
\end{corollary}
\begin{proof}
    If $f$ is separated, we may compactify the map via Nagata \cite[\href{https://stacks.math.columbia.edu/tag/0F3T}{Tag 0F3T}]{stacks-project} and then simply apply \autoref{FFinPullProper}.  In general, we can cover $X$ by an open cover $\mathcal{U} = \{U_i\}$ such that each $U_i \to Y$ is separated.  Following \cite[\href{https://stacks.math.columbia.edu/tag/0AU5}{Tag 0AU5}]{stacks-project}, using the naturality of the isomorphisms in \autoref{FFinPullProper}, the dualizing complexes on the $U_i$ glue into a unique dualizing complex on $X$ which necessarily is the one constructed in this section.
    %\todo[inline]{@Bhargav can you look at this.}
\end{proof}

\section{A classical approach to duality for $F$-finite schemes}
\label{sec.ClassicalDuality}

In the previous section, we saw how to construct a canonical dualizing complex on arbitrary $F$-finite schemes using a modern perspective of identifying the dualizing complex as a unit in a certain monoidal structure on the derived category.  

One could instead ask if the classical approaches to duality as laid out in \cite{HartshorneResidues}, \cite{AltmanKleimanGD}, \cite{ConradGDualityAndBaseChange}, \cite{LipmanHashimotoFoundationsOfGDualityForDiagrams}, or \cite[\href{https://stacks.math.columbia.edu/tag/0DWE}{Tag 0DWE}]{stacks-project} can be utilized instead.    In this section, we carry this out in detail for $F$-finite regular rings and schemes, and then explain what we believe are the necessary steps to proving it for arbitrary $F$-finite rings and schemes.

\subsection{Duality for $F$-finite regular rings}
\label{subsec.DualityForFFiniteRegular}
In this subsection and its sequel, we define a canonical dualizing complex $\omegacan{-}$ for regular $F$-finite rings and show that it is compatible with $(-)^!$-pullback for all (essentially) finite type maps among these. This compatibility allows us to define a canonical dualizing complex for regular $F$-finite schemes, again compatible with $(-)^!$-pullback for finite type maps, see \autoref{thm.CanDualRegular}. 
\begin{defn}[Canonical dualizing complex, regular F-finite case]
    \label{defn.CanDualizingComplexRegular}
    Suppose $R$ is a regular $F$-finite domain.
    We define the \emph{canonical dualizing complex} to be the top exterior power of the Kähler differentials placed in homological degree $-n$ 
    \[
        \omegacan{R} = \Wedge^n \Omega_R [n]
    \]
    where $n$ is the rank of the locally free module $\Omega_R$.
    Notice if $R$ has a $p$-basis then $\Wedge^n \Omega_R [n]$ is free of rank 1.  If $R$ is not a domain, then the dualizing complex is defined as the direct sum of the complexes induced by working modulo each minimal prime.
\end{defn}
\begin{rem}
We point out that the normalization of our dualizing complex is not by Krull dimension, but instead by $p$-dimension \cite{KerzStrunkTamme_VorstConj}, which in the case of the presence of a $p$-basis is just equal to the length of that $p$-basis, or equivalently, to the rank of the Kähler differentials. For example, if $R=\mathbb{F}_p(x)$ is the function field in one variable over $\mathbb{F}_p$ then $x$ is a $p$-basis and hence the dualizing complex of this field is $\omegacan{\mathbb{F}_p(x)}=\Omega_{\mathbb{F}_p(x)}[1]$ which is concentrated in homological degree $-1$. The advantage of our normalization is that this dualizing complex will localize, or equivalently, will be compatible with finite type maps.
\end{rem}

\begin{rem}[finite type over $k = k^p$ case]
    Suppose $R$ is a regular domain and finite type over a perfect field $k$.  Notice that we have surjections:
    \[
        \Omega_R = \Omega_{R/\bF_p} \twoheadrightarrow \Omega_{R/k} \twoheadrightarrow \Omega_{R/R^p}.
    \]
    However, as noted, $\Omega_R = \Omega_{R/R^p}$ and so all notions coincide.  It follows that 
    \[
        \omegacan{R} = \Wedge^{\dim R} \Omega_{R/k}[\dim R]
    \]
    which is how we would have classically defined the dualizing complex.
\end{rem}
We need to show that for any finite type morphism $f \colon R \to S$ of $F$-finite regular rings, there is an isomorphism
\begin{equation}
    \label{eq.xif.MapForFiniteTypeMapsOfRegular}
    \xi_f: \omegacan{S} \to f^! \omegacan{R},
\end{equation}
and the formation of these isomorphisms is compatible with composition.  In this section, we will do this by factoring $f$ into a smooth and finite map and then using the explicit descriptions of $f^!$ for those as defined in \autoref{subsec.FunLocalIso}.
%As pointed out in the introduction, we will do this in two different ways in this paper, by 
%\begin{enumerate}
 %   \item a classical approach via differential forms, which we carry out in this section (although relegating the tedious checking of the crucial compatibilities to \autoref{app.ProofsOfCompatibilities}).
    %\item a modern approach which we pursue in \autoref{ss:DualDiagFFin}, by characterizing $\omegacan{}$ uniquely as the unit of $\stens$ symmetric monoidal structure on $D_{\textrm{coh}}^b(R)$.
%\end{enumerate}

In particular, we use the explicit version of Grothendieck-Serre duality in terms of differential forms for lci immersions and smooth maps as set up in \cite[Chapter III]{HartshorneResidues} with compatible sign convention laid out in \cite[Chapter 2]{ConradGDualityAndBaseChange} and its relation to duality for finite maps. Explicitly, this package gives a $2$-functor $(\_)^!$ on (essentially) finite type morphisms, which is the datum of a functor $f^! \colon D_{\textrm{coh}}(R) \to D_{\textrm{coh}}(S)$ for all essentially of finite type maps $f \colon R \to S$ together with natural compatible isomorphisms
\begin{equation}
\label{eq:psidef}
    \psi_{g,f} 
\colon  (g \circ f)^!\to[\cong]
g^! \circ f^!
\end{equation}
given another essentially of finite type map $g \colon S \to T$. Via factorization, the construction of $(\_)^!$ is reduced to the case of smooth and lci morphisms, where it can be expressed explicitly.

If $f \colon R \to S$ and $g \colon S \to T$ are either smooth or lci, the maps $\psi_{g,f}$ of \autoref{eq:psidef} are explicitly given and ultimately determine the general case after verification of the requisite compatibilities as in \cite[Chapter 2]{ConradGDualityAndBaseChange}.

We now define the isomorphisms $\xi_f : \omegacan{S} \to f^! \omegacan{R}$ for the smooth and lci case using the explicit versions of $(-)^!$.

\begin{defn}[Smooth morphisms]
    \label{defn:SmoothMapsUpperShriekPullbackForRegularRings}
For a smooth morphism $f \colon R \to S$ of relative dimension $d$, define
    \[\begin{tikzcd}[row sep=-.6em]
	{\xi_f \colon \omegacan{S}} && {f^!\omegacan{R} \simeq f^\sharp \omegacan{R} = \bigwedge^d \Omega_{S/R}[d] \otimes_R \omegacan{R}} \\
	{\alpha \wedge f^*\beta} && {\alpha \otimes \beta}
	\arrow[from=1-1, to=1-3]
	\arrow[maps to, from=2-3, to=2-1]
\end{tikzcd}\]
as the isomorphism induced by taking the determinant of the sequence of K\"ahler differentials
\[
0 \to S \otimes_R \Omega_R  \to \Omega_S \to \Omega_{S/R} \to 0.
\]
\end{defn}
For the lci case consider $f \colon R \twoheadrightarrow S$ a surjective map of $F$-finite regular rings.  In particular, it is locally a complete intersection map with $J := \ker f$ locally generated by $c$ elements $r_1,\ldots,r_c$, whose dual basis in $J/J^2$ we denote by $r_1^*, \ldots, r_c^*$. Suppose $m$ is the rank of $\Omega_S$ so that $m + c$ is the rank of $\Omega_R$.

We have the short exact sequence
    \begin{equation}\label{eq:lcidifferentials}
        0 \to J/J^2 \to S \otimes_R \Omega_{R}  \xrightarrow{\phi} \Omega_{S} \to 0
    \end{equation}
where exactness on the left is verified by considering $S$-module rank and the fact that $\Omega_{S}$ is a locally free $S$-module.  In particular, this sequence is split, with a splitting map $\vartheta : \Omega_{S} \to S \otimes_R \Omega_{R}$ which then induces $\Wedge^m \Omega_S \to \Wedge^m S \otimes_R \Omega_{R}$, a map we also call $\vartheta$.  

With this notation and setup we define:
\begin{defn}[Local complete intersections]
    \label{defn.lciCase} Let $f \colon R \to S$ be a surjection of regular $F$-finite rings inducing a codimension-$c$-lci map of schemes. In the notation of the preceding paragraph, the isomorphism
    $\xi_f \colon \omegacan{S} \to f^! \omegacan{R}$ is defined such that the composition with the isomorphism 
    \[
        \eta = \eta_f \colon f^\flat(\_) \to[\cong] \bigwedge^c(J/J^2)^\vee[-c] \otimes_S (S \otimes^{\myL}_R \_ )
    \]   
    \autoref{eq:etadef} explained in \autoref{dualregular} is locally given by
\[\begin{tikzcd}[row sep=-.6em]
	{ \omegacan{S}} &&& {f^!\omegacan{R}} & {\bigwedge^c(J/J^2)^\vee[-c] \otimes_S (S \otimes_R \omegacan{R})} \\
	\alpha &&&& {r_1^* \wedge \cdots \wedge r_c^* \otimes dr_c \wedge \cdots \wedge dr_1 {\wedge} \vartheta_f(\alpha)}
	\arrow["\eta", from=1-4, to=1-5]
	\arrow["{\xi_f}", from=1-1, to=1-4]
	\arrow[maps to, from=2-1, to=2-5]
\end{tikzcd}
\]
It can be verified that this is an isomorphism independent of the choices (the $r_i's$ and $\theta$).
\end{defn}

As any finite type morphism between regular rings can be factored into a smooth morphism (adjoining variables) followed by an lci morphism, we can use the above cases to define $\xi_f$ in general.
\begin{theorem}\label{thm.CanDualRegPullsback}
    Let $h \colon R \to[f] S \to[g] T$ be a finite type morphism of $F$-finite regular rings, factored as $f$ a polynomial extension and $g$ surjective. Define
    $\xi_h \colon \omegacan{T} \to h^! \omegacan{R}$ to be the composition
    \[
    \xi_h \colon \omegacan{T} \to[\xi_g] g^!\omegacan{S} \to[g^!\xi_f] g^!f^!\omegacan{R} \to[\psi_{g,f}^{-1}] h^!\omegacan{R}.
    \]
    This is independent of the factorization.

    If $e : T \to U$ is a further finite type map of $F$-finite regular rings then the following diagram commutes:
    \[
    \begin{tikzcd}
	{\omegacan{U}} && {e^! \omegacan{T}} \\
	{(e \circ h)^!\omegacan{R}} && {e^!(h^!\omegacan{R})}
	\arrow["{\xi_{e\circ h}}"', from=1-1, to=2-1]
	\arrow["{\xi_e}", from=1-1, to=1-3]
	\arrow["{e^!\xi_h}", from=1-3, to=2-3]
	\arrow["{\psi_{e,h}}", from=2-1, to=2-3]
\end{tikzcd}.\]
\end{theorem}
The proof of the independence of the factorization and the compatibility with composition is straightforward from \cite{ConradGDualityAndBaseChange} but somewhat tedious and hence relegated to \autoref{app.ProofsOfCompatibilities}.

As a simple direct consequence of this applied to the finite Frobenius morphism we obtain:
\begin{cor}
Let $R$ be a regular $F$-finite ring. Then there is a canonical isomorphism $\omegacan{R} \to F^!\omegacan{R}$.
\end{cor}

We also note the following compatibility of the canonical dualizing complex with localization and completion.

\begin{prop}\label{prop.OmegaAndLocCompl}
Let $R$ be a regular $F$-finite ring with $p$-basis. Let $f \colon R \to S$ be either a localization at a multiplicative set $W \subseteq R$ or the completion along an ideal $I \subseteq R$. Then $S$ also has a $p$-basis and one has a canonical isomorphism 
\[
    f^*\omegacan{R} \to[\cong] \omegacan{S}.
\]
More generally, the result also holds for any flat map $f : R \to S$ such that $\Omega_R \otimes_R S \to \Omega_{S}$ is an isomorphism.
\end{prop}
\begin{proof}
    If $S = W^{-1}R$ is localization one has $W^{-1}\Omega_R \cong \Omega_{W^{-1}R}$ by \cite[\href{https://stacks.math.columbia.edu/tag/00RT}{Tag 00RT}]{stacks-project}.  If $S = \widehat{R}$ is the completion of $R$ along an ideal $I$, then as $R$ is $F$-finite we have that $\widehat{R^p} \otimes_{R^p} R = \widehat{R}$.  Using \cite[\href{https://stacks.math.columbia.edu/tag/00RV}{Tag 00RV}]{stacks-project} that 
\[
    \Omega_{R/\bF_p} \otimes_{R} \widehat{R} = \Omega_{R/R^p} \otimes_{R^p} \widehat{R^p} = \Omega_{\widehat{R} / \widehat{R^p}} = \Omega_{\widehat{R} / \bF_p}.
\]
In either case we have an isomorphism $\Omega_R \otimes_R S \to \Omega_{S}$, the result follows by taking top exterior powers.
%Let $f \colon R \to W^{-1}R$ be the localization map. The image of any $p$-basis of $R$ is also a $p$-basis of $W^{-1}R$. For the K\"ahler differentials one has by \cite[\href{https://stacks.math.columbia.edu/tag/00RT}{Tag 00RT}]{stacks-project} the isomorphism $W^{-1}\Omega_R \cong \Omega_{W^{-1}R}$. Taking top exterior powers (and shifting) shows the claim.
%
%Next let $\widehat{R}$ denote the completion of $R$ along $I$.  Recall that we have a canonical isomorphism $\widehat{R^p} \otimes_{R^p} R = \widehat{R}$ since $R$ is $F$-finite.  We then have by \cite[\href{https://stacks.math.columbia.edu/tag/00RV}{Tag 00RV}]{stacks-project} that 
%\[
    %\Omega_{R/\bF_p} \otimes_{R} \widehat{R} = \Omega_{R/R^p} \otimes_{R^p} \widehat{R^p} = \Omega_{\widehat{R} / \widehat{R^p}} = \Omega_{\widehat{R} / \bF_p}.
%\]
%In particular, a differential basis/$p$-basis for $R$ (over $\bF_p$) base changes to become a differential basis/$p$-basis for $\widehat{R}$ (over $\bF_p$).  The result follows by taking top exterior powers.
%
%\todo{mnl: Karl you wanted to write out the completion case here.}
%Now if $\phi\colon R \to \widehat{R}$ is  the completion along some ideal of $R$, then again the image of a $p$-basis $x_1,\ldots,x_n$ of $R$ is a $p$-basis $\widehat{x}_1,\ldots,\widehat{x}_n$ of $\widehat{R}$. In particular $\Omega_{\widehat{R}}$ is a finite free $\widehat{R}$-module. Hence the natural map
%\[
%    \Omega_R \otimes_R \widehat{R} \to \Omega_{\widehat{R}}
%\]
%is an isomorphism as it identifies the bases $dx_1 \otimes 1, \ldots, dx_n \otimes 1$ and $d\widehat{x}_1,\ldots,d\widehat{x}_n$. Taking top exterior powers implies the claim.
\end{proof}

%\todo[inline]{There is still the "fun" exercise to be done that shows that this canonical Cartier module structure is indeed given by the classical Cartier operator as recalled above. Since the isomorphism is explicit, one needs to factor $F: R \to R[Y] \to R$ and use the explicit maps from \cite{ConradGDualityAndBaseChange} to fit everything together.}

The compatibilities shown in this section allow us to conclude that the theory developed here in the affine case will globalize to all regular Noetherian $F$-finite schemes and finite type morphisms. In accordance with \autoref{defn.CanDualizingComplexRegular} we define: 
\begin{defn}[Canonical dualizing complex, regular F-finite global case]
    Suppose $X$ is regular Noetherian and $F$-finite scheme. We define the \emph{canonical dualizing complex} to be the top exterior power of the Kähler differentials placed, on each irreducible component $X_i$, in homological degree $-n$ where $n = \operatorname{rk}\Omega_{X_i}$.  That is, on an irreducible $X$,
    \[
        \omegacan{X} := (\Wedge^n \Omega_X) [n].
    \]
\end{defn}

\begin{prop}\label{thm.CanDualRegular}
Let $f \colon X \to Y$ be a finite type morphism of regular Noetherian $F$-finite schemes. Then there is a canonical isomorphism \[\xi_f \colon \omegacan{X} \to f^!\omegacan{Y}\] which is compatible with composition in the sense of \autoref{thm.CanDualRegPullsback}. In particular for the Frobenius map $F \colon X \to X$ we have a canonical isomorphism $\xi_F \colon \omegacan{X} \to[\cong] F^!\omegacan{X}$.
\end{prop}
\begin{proof}
We define $\xi_f$ by gluing the morphism obtained in the affine case (\autoref{thm.CanDualRegPullsback}). Let $U \subseteq X$, $U' \subseteq Y$ be affine open subschemes with $f(U) \subseteq U'$. Then by \cite[\href{https://stacks.math.columbia.edu/tag/01T2}{Lemma 01T2}]{stacks-project} the map
\[
    f_U \colon U \to U', 
\]
whose composition with the inclusion $U' \subseteq Y$ is $f|_U$, is a finite type map of affine regular $F$-finite schemes. Hence we use \autoref{thm.CanDualRegPullsback} to define the isomorphism $\xi_U \colon \omegacan{X}|_U \to (f^! \omegacan{Y})|_U$ as the composition
\[
    \omegacan{X}|_U \cong \omegacan{U} \to[\xi_{f_U}] f_U^!\omegacan{U'} \cong f_U^!(\omegacan{Y}|_U') \cong (f^!\omegacan{Y})|_U 
\]
By \autoref{thm.CanDualRegPullsback} applied to $j \circ f_U$ for some open immersion of open affines $j: U' \hookrightarrow U''$ of $Y$ this is independent of the auxiliary choice of $U'$ in the definition of $\xi_U$. 

Since $\omegacan{U}$ and $f^!\omegacan{Y}$ are both shifted line bundles, it remains to check that for any open affine subset $V \subseteq U$ the restriction $\xi_U|_V$ is equal to $\xi_V$. Using the independence of $\xi_V$ from the auxiliary open $V'$ (which we hence may choose to be equal to $U'$) this comes down to the commutativity of the diagram 
\[\begin{tikzcd}
	{j^!\omegacan{U}} && {j^!f_U^!\omegacan{U'}} \\
	{\omegacan{V}}
	\arrow["{j^!\xi_{f_U}}", from=1-1, to=1-3]
	\arrow["{\xi_j}", from=2-1, to=1-1]
	\arrow["{\xi_{f_V}}"', from=2-1, to=1-3]
\end{tikzcd}\]
which holds by \autoref{thm.CanDualRegPullsback} since all the maps $f_V : V \to[j] U \to[f_U] U'$ are finite type maps of affine regular $F$-finite schemes.

From this definition of the isomorphism $\xi_f$ it follows directly from \autoref{thm.CanDualRegPullsback} again, that it is compatible with composition.
\end{proof}

\subsection{Proof of \autoref{thm.CanDualRegPullsback}}\label{app.ProofsOfCompatibilities}

We now verify the compatibilities appearing in the statement of \autoref{thm.CanDualRegPullsback}. The arguments, though somewhat tedious, are elementary and mainly involve diving into the explicit formulas for the maps defined in \cite{ConradGDualityAndBaseChange} on pages 30--31 and \cite[Chapter III]{HartshorneResidues}.

As indicated in \autoref{subsec.DualityForFFiniteRegular} we need to show that for any finite type morphism $f \colon R \to S$ of $F$-finite regular rings, there is a well defined isomorphism 
\[ 
\xi_f: \omegacan{S} \to f^! \omegacan{R}.
\]
such that if $g : S \to T$ is another finite type morphism between $F$-finite regular rings, then the following diagram commutes
\begin{equation}
\label{eq:compatdiag}
\tag{$\dagger$}
\begin{tikzcd}
	{\omegacan{T}} && {g^! \omegacan{S}} \\
	{(g \circ f)^!\omegacan{R}} && {g^!(f^!\omegacan{R})}
	\arrow["{\xi_{g\circ f}}"', from=1-1, to=2-1]
	\arrow["{\xi_g}", from=1-1, to=1-3]
	\arrow["{g^!\xi_f}", from=1-3, to=2-3]
	\arrow["{\psi_{g,f}}", from=2-1, to=2-3]
\end{tikzcd}
\end{equation}
where $\psi_{g,f}$ is the 2-functoriality morphism from \autoref{eq:psidef}. 

As with the construction of $(\_)^!$ itself, the construction of $\xi_f$ is given by factoring $f$ into a smooth map followed by an lci surjection and treating these cases individually. Let us start by recalling the definition and verify \autoref{eq:compatdiag} in these basic cases.

\begin{lemma}[See \autoref{defn:SmoothMapsUpperShriekPullbackForRegularRings}]
\label{lem:smoothlemma}
    For a smooth morphism $f \colon R \to S$ of relative dimension $d$ (for instance, adjoining variables), recall we define
    \[\begin{tikzcd}[row sep=-.6em]
	{\xi_f \colon \omegacan{S}} && {f^!\omegacan{R} \simeq f^\sharp \omegacan{R} = \bigwedge^d \Omega_{S/R}[d] \otimes_R \omegacan{R}} \\
	{\alpha \wedge f^*\beta} && {\alpha \otimes \beta}
	\arrow[from=1-1, to=1-3]
	\arrow[maps to, from=2-3, to=2-1]
\end{tikzcd}\]
which is the isomorphism determined by taking the determinant of the sequence of K\"ahler differentials
\[
0 \to S \otimes_R \Omega_R  \to \Omega_S \to \Omega_{S/R} \to 0.
\]
If $g \colon S \to T$ is a further smooth map of relative dimension $e$,
the diagram \autoref{eq:compatdiag} commutes.
\end{lemma}

\begin{proof}
Note that in \cite[p. 29--31]{ConradGDualityAndBaseChange} the isomorphism map $\xi_f$ is explicitly constructed (denoted $\zeta'$ there). In our context, diagram \autoref{eq:compatdiag} is given by
\[\begin{tikzcd}
	{\omegacan{T}} && {\bigwedge^e \Omega_{T/S}[e] \otimes_S \omegacan{S}} \\
	{\bigwedge^{e+d} \Omega_{T/R}[e+d] \otimes_R \omegacan{R}} && {\bigwedge^e \Omega_{T/S}[e] \otimes_S \bigwedge^d \Omega_{S/R}[d] \otimes \omegacan{R}.}
	\arrow["{\xi_g}", from=1-1, to=1-3]
	\arrow["{\xi_{g\circ f}}", from=1-1, to=2-1]
	\arrow["{\psi_{g,f}}"', from=2-1, to=2-3]
	\arrow["{\bigwedge^e \Omega_{T/S}[e] \otimes \xi_f}", from=1-3, to=2-3]
\end{tikzcd}\]
The formula for the map $\psi_{g,f}$ on the bottom row as explained in \cite[p. 31 (a)]{ConradGDualityAndBaseChange} is given by the isomorphism
\[\begin{tikzcd}[row sep=-.6em]
	{\bigwedge^e \Omega_{T/S}[e] \otimes_S \bigwedge^d \Omega_{S/R}[d]} & {\bigwedge^{e+d}\Omega_{T/R}[e+d]} \\
	{\alpha \otimes \beta} & {\alpha \wedge g^* \beta}
	\arrow[from=1-2, to=1-1]
	\arrow[maps to, from=2-1, to=2-2]
\end{tikzcd}\]
With these explicit formulas for the maps it is easy to verify that the diagram commutes as claimed. 
\end{proof}

% Consider a surjective map $f : R \to S$ of regular $F$-finite domains with kernel $J = \ker f$ and where $f$ has codimension $c$. 
% Recall that, choosing $\vartheta : \Omega_{S} \to S \otimes_R \Omega_{R}$ giving local splittings of the short exact sequence \autoref{eq:lcidifferentials} gives  $\vartheta \colon \Wedge^m \Omega_S \to S \otimes_R \Wedge^m \Omega_R $ as in \autoref{examp.lciCase}.

%\todo[inline]{Does the above look right?  Explain the notation below, Brian does it on page 30 I think.}
\begin{lemma}[\autoref{defn.lciCase}]
\label{lem:lcilemma}
    Let $f: R \to S$ be a surjective map of regular $F$-finite domains with kernel $J$ of codimension $c$. Then define the isomorphism
    $\xi_f \colon \omegacan{S} \to f^! \omegacan{R}$ so that the composition with the isomorphism 
    \[
        \eta = \eta_f \colon f^\flat(\_) \to[\cong] \bigwedge^c(J/J^2)^\vee[-c] \otimes_S (S \otimes^{\myL} \_ ) = \omegacan{f} \otimes_S (S \otimes^{\myL} \_ ).
    \]   
    from \autoref{eq:etadef} on \autopageref{eq:etadef} is locally given by
\[\begin{tikzcd}[row sep=-.6em]
	{ \omegacan{S}} &&& {f^!\omegacan{R}} & {\bigwedge^c(J/J^2)^\vee[-c] \otimes_S (S \otimes_R \omegacan{R})} \\
	\alpha &&&& {r_1^* \wedge \cdots \wedge r_c^* \otimes dr_c \wedge \cdots \wedge dr_1 {\wedge} \vartheta_f(\alpha)}
	\arrow["\eta", from=1-4, to=1-5]
	\arrow["{\xi_f}", from=1-1, to=1-4]
	\arrow[maps to, from=2-1, to=2-5]
\end{tikzcd}\]
where $J = (r_1, \ldots, r_c)$ is locally given by the regular sequence $r_1, \ldots, r_c \in R$, and $r_1^*, \ldots, r_c^*$ are the corresponding dual basis of $J/J^2$, and $\vartheta_f : \Omega_{S} \to S \otimes_R \Omega_{R}$ gives a local splitting of the short exact sequence \autoref{eq:lcidifferentials}. If $g: S \to T$ is a further local complete intersection map, the diagram \autoref{eq:compatdiag} commutes.
\end{lemma}

% \todo[inline]{Kevin: I think Karl is right, the notation here is a bit misleading at the very least. The idea is the following I believe. We have from the short exact sequence of Kahler differentials that $S \otimes \omega_R = \bigwedge^c J/J^2 \otimes \omega_S$. The element we want on the right side of the equation that $\alpha$ then maps to is $(r_1^* \wedge \cdots r_c^*) \otimes r_c \wedge \cdots \wedge r_1 \otimes \theta(\alpha)$ under this isomorphism, where $\theta$ is a splitting of the Kahler sequence (and this is independent of the choice of splitting). \\
% Mnl: I agree with all your comments.}

\begin{proof}
    Consider the diagram:
\[\begin{tikzcd}[column sep=large]
	{\omegacan{T}} & {g^\flat \omegacan{S}} & {\omegacan{g} \otimes \omegacan{S}} \\
	{(g \circ f)^\flat \omegacan{R}} & {g^\flat f^\flat \omegacan{R}} & {\omegacan{g} \otimes f^\flat \omegacan{R}} \\
	{\omegacan{g \circ f} \otimes \omegacan{R}} && {\omegacan{g} \otimes \omegacan{f} \otimes \omegacan{R}}
	\arrow["{\xi_g}", from=1-1, to=1-2]
	\arrow["\eta_g", from=1-2, to=1-3]
	\arrow["{\omegacan{g} \otimes \xi_f}", from=1-3, to=2-3]
	\arrow["{g^\flat \xi_f}", from=1-2, to=2-2]
	\arrow["{\xi_{g \circ f}}"', from=1-1, to=2-1]
	\arrow["{\psi_{g,f}}", from=2-1, to=2-2]
	\arrow["\eta_{g \circ f}"', from=2-1, to=3-1]
	\arrow["{\mbox{\footnotesize\cite[(b) on p.30/31]{ConradGDualityAndBaseChange}}}"', from=3-1, to=3-3]
	\arrow["{\omegacan{g} \otimes \eta_f}", from=2-3, to=3-3]
	\arrow["{\eta_g}", from=2-2, to=2-3]
\end{tikzcd}\]
for which we have to show the commutativity of the top left square. The top right square commutes due to the naturality of the fundamental local isomorphism \autoref{dualregular}. 
The bottom rectangle commutes in light of \cite[Theorem 2.5.1]{ConradGDualityAndBaseChange}. Finally, the commutativity of the outer square can be verified using the explicit local formulas for the maps defined above, together with the formula from \cite[(b) on p.30/31]{ConradGDualityAndBaseChange} 
for the bottom arrow. 

Concretely, if $r_1, \ldots, r_c \in R$ are a regular sequence generating the kernel of $f$, and we extend so that $r_1, \ldots, r_c, s_1, \ldots, s_e$ are a regular sequence generating the kernel of $f \circ g$, both compositions from the top left corner to the bottom right is given by
\[
\alpha \mapsto \overline{s_1}^* \wedge \cdots \wedge \overline{s_e}^* \otimes r_1^* \wedge \cdots \wedge r_c^* \otimes dr_c \wedge \cdots \wedge dr_1 \wedge ds_e \wedge \cdots \wedge ds_1 \wedge  \vartheta_{g \circ f}( \alpha)
\]
where $\overline{s_1}, \ldots, \overline{s_e}$ denote the images of $s_1, \ldots, s_e$ in $S$, and $\vartheta_{g \circ f}$ is induced by the composition of local splittings
\[
\Omega_T \to[\vartheta_g] T \otimes_S \Omega_S \to[T \otimes_S \vartheta_f] T \otimes_S S \otimes_R \Omega_R = T \otimes_R \Omega_R.
\]
From this, the commutativity of the upper left square now follows.
%\todo[inline]{\emph{Karl:} Is the $(g \circ f)^*$ just a composition of $\vartheta$'s I assume?}
\end{proof}

The following lemma is a crucial tool in the proof that the definition of $\xi_f$ via factorization into smooth and lci is independent of the chosen factorization.  

\begin{lemma}
\label{lem:sectionlemma}
Suppose that $f \colon R \to S=R[x_1,\ldots,x_n]$ is a polynomial extension of relative dimension $n$, and $g \colon S \to R$ is a section of $f$, so that $g \circ f = \id_R$. Then the diagram
\[\begin{tikzcd}
	{\omegacan{R}} & {g^\flat\omegacan{S}} \\
	{\omegacan{R} = (g \circ f)^! \omegacan{R}} & {g^\flat f^\sharp \omegacan{R}}
	\arrow["{\xi_g}", from=1-1, to=1-2]
	\arrow["{\id_{\omegacan{R}} = \xi_{\id_R}}"', from=1-1, to=2-1]
	\arrow["{\psi_{g,f}}"', from=2-1, to=2-2]
	\arrow["{g^\flat \xi_f}", from=1-2, to=2-2]
\end{tikzcd}\]
commutes. 
\end{lemma}

\begin{proof}
   We must show the commutativity of the top left square of the following diagram
\[\begin{tikzcd}[column sep=large]
	{\omegacan{R}} & {g^\flat \omegacan{S}} & {\omegacan{g} \otimes \omegacan{S}} \\
	{(g \circ f)^! \omegacan{R}} & {g^\flat f^\sharp \omegacan{R}} & {\omegacan{g} \otimes f^\sharp \omegacan{R}} \\
	{\omegacan{R}} && {\omegacan{g} \otimes \omegacan{f} \otimes \omegacan{R}}
	\arrow["{\xi_g}", from=1-1, to=1-2]
	\arrow["\eta", from=1-2, to=1-3]
	\arrow["{\omegacan{g} \otimes \xi_f}", from=1-3, to=2-3]
	\arrow["{g^\flat \xi_f}", from=1-2, to=2-2]
	\arrow["{\xi_{g \circ f}}"', from=1-1, to=2-1]
	\arrow["{\psi_{g,f}}", from=2-1, to=2-2]
	\arrow[shift left=1, no head, from=2-1, to=3-1]
	\arrow[shift left=1, no head, from=2-3, to=3-3]
	\arrow["{\eta f^\sharp}", from=2-2, to=2-3]
	\arrow[no head, from=2-1, to=3-1]
	\arrow[no head, from=2-3, to=3-3]
	\arrow["{\mbox{\footnotesize \cite[2.2.3 (c) or (d)]{ConradGDualityAndBaseChange}} \otimes \omegacan{R}}"', from=3-1, to=3-3]
\end{tikzcd}.\]
Again, the top right square commutes due to the naturality of the fundamental local isomorphism \autoref{dualregular}. The bottom rectangle commutes due to the definition of $\psi_{g,f}$ as in \cite[2.7.2]{ConradGDualityAndBaseChange}. As $f \colon R \to S = R[x_1, \ldots, x_n]$ is the natural inclusion, and $g \colon S \to R$ is determined by $g(x_i) = 0$ for all $i$. Using the explicit local formulas for the maps $\xi_f$ and $\xi_g$ defined
above, together with the formula from \cite[(c) or (d) on p.30/31]{ConradGDualityAndBaseChange} for the bottom arrow, both compositions from the top left corner to the bottom right are given by
\[
\alpha \mapsto x_1^* \wedge \cdots \wedge x_n^* \otimes dx_n \wedge \cdots \wedge dx_1 \otimes \alpha.
\]
This shows the commutativity of the upper left square.
\end{proof}

Now, for $h \colon R \to[f] S \to[g] T$ a finite type morphism of $F$-finite regular rings, factored as $f$ a polynomial extension and $g$ surjective. Define $\xi_h \colon \omegacan{T} \to h^! \omegacan{R}$ to be the composition
\[
    \xi_h \colon \omegacan{T} \to[\xi_g] g^!\omegacan{S} \to[g^!\xi_f] g^!f^!\omegacan{R} \to[\psi_{g,f}^{-1}] h^!\omegacan{R}.
\]
We are now ready to prove \autoref{thm.CanDualRegPullsback}.  Namely, we will show that this definition does not depend on the choice of factorization and that it is compatible with composition. 
We repeat the statement for the convenience of the reader.

\begin{theorem*}[\autoref{thm.CanDualRegPullsback}]\label{appthm.CanDualRegPullsback}
    Let $h \colon R \to[f] S \to[g] T$ be a finite type morphism of $F$-finite regular rings, factored as $f$ a polynomial extension and $g$ surjective. Define
    $\xi_h \colon \omegacan{T} \to h^! \omegacan{R}$ to be the composition
    \[
    \xi_h \colon \omegacan{T} \to[\xi_g] g^!\omegacan{S} \to[g^!\xi_f] g^!f^!\omegacan{R} \to[\psi_{g,f}^{-1}] h^!\omegacan{R}.
    \]
    This is independent of the factorization.

    If $e : T \to U$ is a further finite type map of $F$-finite regular rings then the following diagram commutes:
    \[
    \begin{tikzcd}
	{\omegacan{U}} && {e^! \omegacan{T}} \\
	{(e \circ h)^!\omegacan{R}} && {e^!(h^!\omegacan{R})}
	\arrow["{\xi_{e\circ h}}"', from=1-1, to=2-1]
	\arrow["{\xi_e}", from=1-1, to=1-3]
	\arrow["{e^!\xi_h}", from=1-3, to=2-3]
	\arrow["{\psi_{e,h}}", from=2-1, to=2-3]
\end{tikzcd}.\]
\end{theorem*}

\begin{proof}[Proof of \autoref{thm.CanDualRegPullsback}]
    Since any two factorizations of $h$ as in \autoref{thm.CanDualRegPullsback} can be dominated by a third, and since polynomial extensions have a section, the first statement follows from the claim below.
\begin{claim}
\label{compatclaim}
    Let $R,S,S',T$ be $F$-finite regular rings. Suppose we are given maps
    \[\begin{tikzcd}
	& {S'} \\
	R & S & T
	\arrow["{f'}", from=2-1, to=1-2]
	\arrow["\gamma", shift left=1, from=1-2, to=2-2]
	\arrow["\phi", shift left=1, from=2-2, to=1-2]
	\arrow["f", from=2-1, to=2-2]
	\arrow["{g'}", shift right=1, from=1-2, to=2-3]
	\arrow["g", from=2-2, to=2-3]
\end{tikzcd}\]
where $f,f',\phi$ are polynomial extensions with $f' = \phi \circ f$, $\gamma,g,g'$ are surjective (hence also lci) with $g' = g \circ \gamma$, and $\gamma$ is a section of $\phi$ so that $\gamma \circ \phi = \id_S$. Then the following diagram commutes

% https://q.uiver.app/#q=WzAsNyxbMCwwLCJcXG9tZWdhY2Fue1R9Il0sWzAsMSwiZydeIVxcb21lZ2FjYW57Uyd9Il0sWzEsMCwiZ14hXFxvbWVnYWNhbntTfSJdLFsyLDAsImdeIWZeIVxcb21lZ2FjYW57Un0iXSxbMCwyLCJnJ14hZideIVxcb21lZ2FjYW57Un0iXSxbMiwyLCIoZydmJyleIVxcb21lZ2FjYW57Un09KGdmKV4hXFxvbWVnYWNhbntSfSJdLFsxLDJdLFswLDIsIntcXHhpX2d9Il0sWzAsMSwie1xceGlfe2cnfX0iLDJdLFsyLDMsIntnXiFcXHhpX2Z9Il0sWzEsNCwie2cnXiFcXHhpX3tmJ319IiwyXSxbNCw1LCJcXHBzaV97ZycsZid9XnstMX0iLDJdLFszLDUsIlxccHNpX3tnLGZ9XnstMX0iXV0=&macro_url=%5Cnewcommand%7B%5Comegacan%7D%5B1%5D%7B%5Comega%5E%7B%5Cmydot%7D_%7B%231%7D%7D
\[\begin{tikzcd}[cramped]
	{\omegacan{T}} & {g^!\omegacan{S}} & {g^!f^!\omegacan{R}} \\
	{g'^!\omegacan{S'}} \\
	{g'^!f'^!\omegacan{R}} & {} & {(g'f')^!\omegacan{R}=(gf)^!\omegacan{R}}
	\arrow["{{\xi_g}}", from=1-1, to=1-2]
	\arrow["{{\xi_{g'}}}"', from=1-1, to=2-1]
	\arrow["{{g^!\xi_f}}", from=1-2, to=1-3]
	\arrow["{{g'^!\xi_{f'}}}"', from=2-1, to=3-1]
	\arrow["{\psi_{g',f'}^{-1}}"', from=3-1, to=3-3]
	\arrow["{\psi_{g,f}^{-1}}", from=1-3, to=3-3]
\end{tikzcd}\]
\end{claim}

\begin{proof}[Proof of claim]   
Considering the 2-functoriality of $(\_)^!$, to verify the claim it suffices to check that the outer square in the diagram
    \[\begin{tikzcd}
	{\omegacan{T}} & {g^!\omegacan{S}} & {g^!f^!\omegacan{R}} \\
	{g'^!\omegacan{S'}} & {g^!\gamma^!\omegacan{S'}} \\
	{g'^!f'^!\omegacan{R}} && {g^!\gamma^!f'^!\omegacan{R}}
	\arrow["{\xi_{g'}}"', from=1-1, to=2-1]
	\arrow["{\xi_g}", from=1-1, to=1-2]
	\arrow["{g^!\xi_f}", from=1-2, to=1-3]
	\arrow["{g'^!\xi_{f'}}"', from=2-1, to=3-1]
	\arrow["{g^!\xi_\gamma}"', from=1-2, to=2-2]
	\arrow["{\psi_{g,\gamma}}", from=2-1, to=2-2]
	\arrow["{(1)}"{description}, shift left=4, curve={height=6pt}, draw=none, from=1-1, to=2-2]
	\arrow["{g^!\psi_{\gamma,f'}}", from=1-3, to=3-3]
	\arrow["{\psi_{g,\gamma}}"', from=3-1, to=3-3]
	\arrow["g^!\gamma^!\xi_{f'}"',from=2-2, to=3-3]
	\arrow["{(2)}"{description, pos=0.3}, shift right=2, curve={height=-12pt}, draw=none, from=1-2, to=3-3]
	\arrow["{(3)}"{description, pos=0.3}, shift right=1, curve={height=6pt}, draw=none, from=2-1, to=3-3]
\end{tikzcd}\]
commutes. Note that the square labeled (1) commutes using \autoref{lem:lcilemma}, and the square (3) commutes by the naturality of the transformation $\psi_{g,\gamma} \colon g'^! \to g^! \gamma^!$. Thus, we are left to check only that the square labeled (2) commutes.

To that end, it suffices to show the square labeled (4) below
    \[\begin{tikzcd}
	{\omegacan{S}} & {} & {f^!\omegacan{R}} \\
	\\
	{\gamma^!\omegacan{S'}} & {} & {\gamma^!f'^!\omegacan{R}} \\
	\\
	{\gamma^!\phi^!\omegacan{S}} & {} & {\gamma^!\phi^!f^!\omegacan{R}}
	\arrow["{\xi_f}", from=1-1, to=1-3]
	\arrow["{\gamma^!\xi_{f'}}", from=3-1, to=3-3]
	\arrow["{\gamma^!\phi^!\xi_f}", from=5-1, to=5-3]
	\arrow["{\xi_\gamma}"', from=1-1, to=3-1]
	\arrow["{\gamma^!\xi_\phi}"', from=3-1, to=5-1]
	\arrow["{\psi_{\gamma,f'}}", from=1-3, to=3-3]
	\arrow["{\gamma^!\psi_{\phi,f}}", from=3-3, to=5-3]
	\arrow["{\psi_{\gamma,\phi}}", shift left=5, curve={height=-40pt}, from=1-3, to=5-3]
	\arrow["{\psi_{\gamma,\phi}}"', shift right=5, curve={height=40pt}, from=1-1, to=5-1]
	\arrow["{(5)}"{description}, curve={height=40pt}, draw=none, from=1-1, to=5-1]
	\arrow["{(6)}"{description}, draw=none, from=3-2, to=5-2]
	\arrow["{(7)}"{description}, curve={height=-40pt}, draw=none, from=1-3, to=5-3]
	\arrow["{(4)}"{description}, draw=none, from=1-2, to=3-2]
\end{tikzcd}\]
commutes. First, observe the left wing labeled (5) commutes due to \autoref{lem:sectionlemma}. Next, observe the square labeled (6) commutes due to \autoref{lem:smoothlemma}. Also, the right wing
of the diagram labeled (7) commutes in light of the 2-functoriality of $(\_)!$, see \cite[(2.7.15), third nontrivial case]{ConradGDualityAndBaseChange} and also \cite[Theorem 2.7.2]{ConradGDualityAndBaseChange}. Since the outer square also commutes due to the naturality of the natural transformation $\psi_{\gamma,\phi} \colon \id \to \gamma^!\phi^!$, we conclude that square (4) commutes. Applying $g^!$, we see (2) commutes as well. This completes the verification of \autoref{compatclaim}.
\end{proof} 
The claim proves that the definition of $\xi_h$ is independent of the factorization of $h$. For the final statement in \autoref{thm.CanDualRegPullsback}, consider the following diagram which is obtained by taking a presentation of $h$ (resp. $e$) into a polynomial map $R \to S$ (resp. $T \to T[x]=T[x_1,\ldots,x_d]$) and a lci surjection $S \to T$ (resp. $T[x] \to U$): 
% https://q.uiver.app/#q=WzAsNixbMCwyLCJSIl0sWzEsMSwiUyJdLFsyLDIsIlQiXSxbNCwyLCJVIl0sWzMsMSwiVFt4XSJdLFsyLDAsIlNbeF0iXSxbMCwxLCJmIl0sWzEsNSwiYSciXSxbNSw0LCJnJyIsMCx7InN0eWxlIjp7ImhlYWQiOnsibmFtZSI6ImVwaSJ9fX1dLFs0LDMsImIiLDAseyJzdHlsZSI6eyJoZWFkIjp7Im5hbWUiOiJlcGkifX19XSxbMSwyLCJnIiwwLHsic3R5bGUiOnsiaGVhZCI6eyJuYW1lIjoiZXBpIn19fV0sWzIsNCwiYSJdLFswLDIsImgiLDJdLFsyLDMsImMiLDJdXQ==
\[\begin{tikzcd}[cramped]
	&& {S[x]} \\
	& S && {T[x]} \\
	R && T && U
	\arrow["f", from=3-1, to=2-2]
	\arrow["{a'}", from=2-2, to=1-3]
	\arrow["{g'}", two heads, from=1-3, to=2-4]
	\arrow["b", two heads, from=2-4, to=3-5]
	\arrow["g", two heads, from=2-2, to=3-3]
	\arrow["a", from=3-3, to=2-4]
	\arrow["h"', from=3-1, to=3-3]
	\arrow["e"', from=3-3, to=3-5]
\end{tikzcd}\]
Observe that the top square is Cartesian. As all upward arrows are polynomials maps, all downward arrows are lci surjections, the topmost path over the mountain is a factorization of the composition $e \circ h$. First observe, that but by \autoref{lem:smoothlemma} diagram \autoref{eq:compatdiag} holds for the composition $a'\circ f$ and as well for the composition $b\circ g'$ by \autoref{lem:lcilemma}. Hence it is enough to show equality of the isomorphisms $\xi_{g'\circ a'}=\psi^{-1}_{g',a'}\circ g'^!\xi_{a'}\circ\xi_{g'}$ and $\psi_{a,g}^{-1}\circ a^!\xi_{g}\circ\xi_{a}$ obtained by taking the top or bottom path in the cartesian square. In other words, our final goal is to prove that square (8) in Figure \ref{keydiagram} commutes (which is \autoref{eq:compatdiag} for the composition $a \circ g$). For simplicity of notation, all unlabeled tensor products for the remainder of this proof are taken over $T[x]$.

\begin{figure}[h]
% https://q.uiver.app/#q=WzAsOCxbMCwwLCJcXG9tZWdhX3tUW3hdfV5cXG15ZG90Il0sWzEsMCwiYV4hXFxvbWVnYV9UXlxcbXlkb3QiXSxbMiwwLCJhXiFnXiFcXG9tZWdhX1NeXFxteWRvdCJdLFsyLDIsImcnXiFhJ14hIFxcb21lZ2FfU15cXG15ZG90Il0sWzAsMiwiZydeISBcXG9tZWdhX3tTW3hdfV5cXG15ZG90Il0sWzAsNCwiXFxsZWZ0KCBcXGJpZ3dlZGdlXmMgXFxmcmFje0pbeF19e0pbeF1eMn1cXHJpZ2h0KV5cXHZlZVstY10gXFxvdGltZXMgZydeKiBcXG9tZWdhX3tTW3hdfV5cXG15ZG90Il0sWzMsMCwiYV4hXFxsZWZ0KFxcbGVmdCggXFxiaWd3ZWRnZV5jIFxcZnJhY3tKfXtKXjJ9IFxccmlnaHQpXlxcdmVlIFstY10gXFxvdGltZXNfVCBnXiogXFxvbWVnYV9TXlxcbXlkb3QgIFxccmlnaHQpIl0sWzMsMywiXFxsZWZ0KFxcYmlnd2VkZ2VeYyBcXGZyYWN7Slt4XX17Slt4XV4yfSBcXHJpZ2h0KV5cXHZlZVstY10gXFxvdGltZXMgZydeKiBhJ14hIFxcb21lZ2FfU15cXG15ZG90ICJdLFswLDEsIlxceGlfYSJdLFsxLDIsImFeIVxceGlfZyJdLFsyLDMsIlxccHNpX3tnJyxhJ30gXFxwc2lfe2EsZ31eey0xfSJdLFswLDQsIlxceGlfe2cnfSIsMl0sWzQsMywiZydeIVxceGlfe2EnfSIsMl0sWzQsNSwiXFxldGFfe2cnfSIsMl0sWzIsNiwiYV4hXFxldGFfZyJdLFszLDcsIlxcZXRhX3tnJ30iXSxbNiw3LCJcXGthcHBhIl0sWzUsNywiXFxsZWZ0KFxcYmlnd2VkZ2VeYyBcXGZyYWN7Slt4XX17Slt4XV4yfSBcXHJpZ2h0KV5cXHZlZVstY10gXFxvdGltZXMgZydeKlxceGlfe2EnfSIsMl0sWzEsMTIsIig4KSIsMSx7InNob3J0ZW4iOnsidGFyZ2V0IjoyMH0sInN0eWxlIjp7ImJvZHkiOnsibmFtZSI6Im5vbmUifSwiaGVhZCI6eyJuYW1lIjoibm9uZSJ9fX1dLFsxMiwxNywiKDEwKSIsMSx7InNob3J0ZW4iOnsic291cmNlIjoyMCwidGFyZ2V0IjoyMH0sInN0eWxlIjp7ImJvZHkiOnsibmFtZSI6Im5vbmUifSwiaGVhZCI6eyJuYW1lIjoibm9uZSJ9fX1dLFsxNCw3LCIoOSkiLDAseyJvZmZzZXQiOjUsInNob3J0ZW4iOnsic291cmNlIjoyMCwidGFyZ2V0IjoyMH0sInN0eWxlIjp7ImJvZHkiOnsibmFtZSI6Im5vbmUifSwiaGVhZCI6eyJuYW1lIjoibm9uZSJ9fX1dXQ==
\[\begin{tikzcd}[cramped]
	{\omega_{T[x]}^\mydot} & {a^!\omega_T^\mydot} & {a^!g^!\omega_S^\mydot} & {a^!\left(\left( \bigwedge^c \frac{J}{J^2} \right)^\vee [-c] \otimes_T g^* \omega_S^\mydot  \right)} \\
	\\
	{g'^! \omega_{S[x]}^\mydot} && {g'^!a'^! \omega_S^\mydot} \\
	&&& {\left(\bigwedge^c \frac{J[x]}{J[x]^2} \right)^\vee[-c] \otimes g'^* a'^! \omega_S^\mydot } \\
	{\left( \bigwedge^c \frac{J[x]}{J[x]^2}\right)^\vee[-c] \otimes g'^* \omega_{S[x]}^\mydot}
	\arrow["{\xi_a}", from=1-1, to=1-2]
	\arrow["{\xi_{g'}}"', from=1-1, to=3-1]
	\arrow["{a^!\xi_g}", from=1-2, to=1-3]
	\arrow[""{name=0, anchor=center, inner sep=0}, "{a^!\eta_g}", from=1-3, to=1-4]
	\arrow["{\psi_{g',a'} \psi_{a,g}^{-1}}", from=1-3, to=3-3]
	\arrow["\kappa", from=1-4, to=4-4]
	\arrow[""{name=1, anchor=center, inner sep=0}, "{g'^!\xi_{a'}}"', from=3-1, to=3-3]
	\arrow["{\eta_{g'}}"', from=3-1, to=5-1]
	\arrow["{\eta_{g'}}", from=3-3, to=4-4]
	\arrow[""{name=2, anchor=center, inner sep=0}, "{\left(\bigwedge^c \frac{J[x]}{J[x]^2} \right)^\vee[-c] \otimes g'^*\xi_{a'}}"', from=5-1, to=4-4]
	\arrow["{(8)}"{description}, draw=none, from=1-2, to=1]
	\arrow["{(9)}", shift right=5, draw=none, from=0, to=4-4]
	\arrow["{(10)}"{description}, draw=none, from=1, to=2]
\end{tikzcd}\]
\caption{We need to prove commutativity of (8).}\label{keydiagram}
\end{figure}

 To see that (8) commutes, our strategy is as follows: we shall argue that squares (9) and (10) in Figure \ref{keydiagram} commute, giving explicit references as we go, and then subsequently
 verify the commutativity of the outside diagram in Figure \ref{keydiagram} by an explicit check on elements. The first verification is straightforward: that (10) commutes is a direct consequence of the functoriality of the fundamental local isomorphism $\eta_{g'}$ from \autoref{dualregular} (\textit{cf.} \cite[III Proposition 7.2]{HartshorneResidues} and \cite[(2.5.3)]{ConradGDualityAndBaseChange}).

Next, to see that (9) in Figure \ref{keydiagram} commutes, we begin with the observation that the diagram on page~183 of \cite{HartshorneResidues} commutes, with the caveat that we are consistently using the conventions in \cite{ConradGDualityAndBaseChange} in the order of the tensor products. Explicitly, this gives a commutative diagram
% https://q.uiver.app/#q=WzAsNCxbMCwwLCJhXipnXlxcZmxhdFxcb21lZ2FfU15cXG15ZG90Il0sWzIsMCwiYV4qXFxsZWZ0KFxcbGVmdChcXGJpZ3dlZGdlXmMge0ogXFxvdmVyIEpeMn0gXFxyaWdodCleXFx2ZWUgWy1jXSBcXG90aW1lc19UIGdeKlxcb21lZ2FfU15cXG15ZG90IFxccmlnaHQpIl0sWzAsMSwiZydeXFxmbGF0IGEnXiogXFxvbWVnYV9TXlxcbXlkb3QiXSxbMiwxLCJcXGxlZnQoXFxiaWd3ZWRnZV5jIHtKW3hdIFxcb3ZlciBKW3hdXjJ9IFxccmlnaHQpXlxcdmVlWy1jXSBcXG90aW1lcyBnJ14qYSdeKlxcb21lZ2FfU15cXG15ZG90Il0sWzAsMSwiYV4qXFxldGFfZyJdLFswLDJdLFsyLDMsIlxcZXRhX3tnJ30iLDJdLFsxLDNdXQ==
\[\begin{tikzcd}[cramped]
	{a^*g^\flat\omega_S^\mydot} && {a^*\left(\left(\bigwedge^c {J \over J^2} \right)^\vee [-c] \otimes_T g^*\omega_S^\mydot \right)} \\
	{g'^\flat a'^* \omega_S^\mydot} && {\left(\bigwedge^c {J[x] \over J[x]^2} \right)^\vee[-c] \otimes g'^*a'^*\omega_S^\mydot}
	\arrow["{a^*\eta_g}", from=1-1, to=1-3]
	\arrow[from=1-1, to=2-1]
	\arrow[from=1-3, to=2-3]
	\arrow["{\eta_{g'}}"', from=2-1, to=2-3]
\end{tikzcd}\]
which commutes by compatibility of the fundamental local isomorphism with flat base change. Here the vertical arrows are the usual identifications arising from the Cartesian square. This compatibility is also asserted at the bottom of page~52 of \cite{ConradGDualityAndBaseChange}, where it is noted to hold in greater generality, and it appears again at the bottom of page~54 of \cite{ConradGDualityAndBaseChange}, which we will use momentarily. 

Applying the functor $\bigwedge^d \Omega_{T[x]/T}[d] \otimes (-)$ to the entire diagram above, we see that the top square in the diagram below commutes
% https://q.uiver.app/#q=WzAsOCxbMCwwLCJhXiFnXlxcZmxhdFxcb21lZ2FfU15cXG15ZG90Il0sWzIsMCwiYV4hXFxsZWZ0KFxcbGVmdChcXGJpZ3dlZGdlXmMge0ogXFxvdmVyIEpeMn0gXFxyaWdodCleXFx2ZWVbLWNdIFxcb3RpbWVzX1QgZ14qXFxvbWVnYV9TXlxcbXlkb3QgXFxyaWdodCkiXSxbMCwxLCJcXGJpZ3dlZGdlXmQgXFxPbWVnYV97VFt4XSBcXG92ZXIgVH1bZF0gXFxvdGltZXMgZydeXFxmbGF0IGEnXiogXFxvbWVnYV9TXlxcbXlkb3QiXSxbMiwxLCJcXGJpZ3dlZGdlXmQgXFxPbWVnYV97VFt4XSBcXG92ZXIgVH1bZF0gXFxvdGltZXNcXGxlZnQoXFxiaWd3ZWRnZV5jIHtKW3hdIFxcb3ZlciBKW3hdXjJ9IFxccmlnaHQpXlxcdmVlWy1jXSBcXG90aW1lcyBnJ14qYSdeKlxcb21lZ2FfU15cXG15ZG90Il0sWzAsNCwiZydeIWEnXiFcXG9tZWdhX1NeXFxteWRvdCJdLFsyLDMsIlxcbGVmdChcXGJpZ3dlZGdlXmMge0pbeF0gXFxvdmVyIEpbeF1eMn0gXFxyaWdodCleXFx2ZWVbLWNdIFxcb3RpbWVzIGcnXipcXGxlZnQoXFxiaWd3ZWRnZV5kIFxcT21lZ2Ffe1NbeF0gXFxvdmVyIFN9W2RdIFxcb3RpbWVzX3tTW3hdfWEnXipcXG9tZWdhX1NeXFxteWRvdFxccmlnaHQpIl0sWzIsNCwiXFxsZWZ0KFxcYmlnd2VkZ2VeYyB7Slt4XSBcXG92ZXIgSlt4XV4yfSBcXHJpZ2h0KV5cXHZlZVstY10gXFxvdGltZXMgZydeKmEnXiFcXG9tZWdhX1NeXFxteWRvdCJdLFsyLDIsIiBcXGxlZnQoXFxiaWd3ZWRnZV5jIHtKW3hdIFxcb3ZlciBKW3hdXjJ9IFxccmlnaHQpXlxcdmVlWy1jXVxcb3RpbWVzIFxcYmlnd2VkZ2VeZCBcXE9tZWdhX3tUW3hdIFxcb3ZlciBUfVtkXSBcXG90aW1lcyBnJ14qYSdeKlxcb21lZ2FfU15cXG15ZG90Il0sWzAsMSwiYV4hXFxldGFfZyJdLFswLDJdLFsyLDMsIlxcYmlnd2VkZ2VeZCBcXE9tZWdhX3tUW3hdIFxcb3ZlciBUfVtkXSBcXG90aW1lcyBcXGV0YV97Zyd9IiwyXSxbMSwzXSxbNSw2LCIiLDIseyJsZXZlbCI6Miwic3R5bGUiOnsiaGVhZCI6eyJuYW1lIjoibm9uZSJ9fX1dLFs0LDYsIlxcZXRhX3tnJ30iXSxbMiw0XSxbMyw3XSxbNyw1LCIiLDEseyJsZXZlbCI6Miwic3R5bGUiOnsiaGVhZCI6eyJuYW1lIjoibm9uZSJ9fX1dXQ==
\[\begin{tikzcd}[cramped]
	{a^!g^\flat\omega_S^\mydot} && {a^!\left(\left(\bigwedge^c {J \over J^2} \right)^\vee[-c] \otimes_T g^*\omega_S^\mydot \right)} \\
	{\bigwedge^d \Omega_{T[x] \over T}[d] \otimes g'^\flat a'^* \omega_S^\mydot} && {\bigwedge^d \Omega_{T[x] \over T}[d] \otimes\left(\bigwedge^c {J[x] \over J[x]^2} \right)^\vee[-c] \otimes g'^*a'^*\omega_S^\mydot} \\
	&& { \left(\bigwedge^c {J[x] \over J[x]^2} \right)^\vee[-c]\otimes \bigwedge^d \Omega_{T[x] \over T}[d] \otimes g'^*a'^*\omega_S^\mydot} \\
	&& {\left(\bigwedge^c {J[x] \over J[x]^2} \right)^\vee[-c] \otimes g'^*\left(\bigwedge^d \Omega_{S[x] \over S}[d] \otimes_{S[x]}a'^*\omega_S^\mydot\right)} \\
	{g'^!a'^!\omega_S^\mydot} && {\left(\bigwedge^c {J[x] \over J[x]^2} \right)^\vee[-c] \otimes g'^*a'^!\omega_S^\mydot}
	\arrow["{a^!\eta_g}", from=1-1, to=1-3]
	\arrow[from=1-1, to=2-1]
	\arrow[from=1-3, to=2-3]
	\arrow["{\bigwedge^d \Omega_{T[x] \over T}[d] \otimes \eta_{g'}}"', from=2-1, to=2-3]
	\arrow[from=2-1, to=5-1]
	\arrow[from=2-3, to=3-3]
	\arrow[equals, from=3-3, to=4-3]
	\arrow[equals, from=4-3, to=5-3]
	\arrow["{\eta_{g'}}", from=5-1, to=5-3]
\end{tikzcd}\]
and where the bottom square is roughly given by passing the tensor with the free module $\bigwedge^d \Omega_{T[x]/T} = T[x] \otimes_{S[x]}\bigwedge^d \Omega_{S[x]/S}$ inside the corresponding functors. That this resulting bottom square commutes is a special case of \cite[Equation~(2.5.7)]{ConradGDualityAndBaseChange}, up to a cohomological shift. Note that commuting the factors $\left(\bigwedge^c J[x]/J[x]^2\right)^\vee[-c]$ and $\bigwedge^d \Omega_{T[x]/T}[d]$ introduces a Koszul sign $(-1)^{cd}$ according to the sign conventions of \cite[bottom of page~11]{ConradGDualityAndBaseChange}. With this understood, the outer rectangle of the above diagram is precisely (9) in \autoref{keydiagram}. Indeed, the key point is the explicit description of the isomorphism $\psi_{g',a'} \psi_{a,g}^{-1}$ given as the composition down the left side \cite[Proposition 6.3, Corollary 6.4]{HartshorneResidues}. The composition down the right side is the map denoted $\kappa$ in the original diagram.

Thus, to complete the argument, we need to show that the exterior square of the original diagram commutes. Since all maps in sight are isomorphisms, we do this by an explicit computation on elements. We begin with an element of
\[
a^!\omega^\mydot_T = \bigwedge^d \Omega_{T[x] \over T}[d] \otimes a^*\omega_T^\mydot
\]
located in the middle of the top row of the diagram and trace its image to the bottom left corner along the two possible paths. As $\bigwedge^d \Omega_{T[x]/T}$ is free of rank one, generated by $dx = dx_1 \wedge \cdots \wedge dx_d$, it suffices to check this on elements of the form $dx \otimes a^*\beta$ for $\beta \in \omega_T$.
For convenience, and using the identification $\bigwedge^d \Omega_{T[x]/T} = T[x] \otimes_{S[x]}\bigwedge^d \Omega_{S[x]/S}$, we will also use $dx$ to denote the relevant element of $\Omega_{S[x]/S}$.

We first compute the image obtained by moving clockwise around the exterior \autoref{keydiagram}. In order to apply $\xi_g$, we fix a splitting $\theta_g$ of the short exact sequence
% https://q.uiver.app/#q=WzAsNSxbMSwwLCJKL0peMiJdLFsyLDAsIlRcXG90aW1lc19TXFxPbWVnYV9TIl0sWzMsMCwiXFxPbWVnYV9UIl0sWzAsMCwiMCJdLFs0LDAsIjAiXSxbMywwXSxbMCwxXSxbMSwyXSxbMiw0XSxbMiwxLCJcXHRoZXRhX2ciLDIseyJjdXJ2ZSI6Miwic3R5bGUiOnsiYm9keSI6eyJuYW1lIjoiZGFzaGVkIn19fV1d
\[\begin{tikzcd}
	0 & {J/J^2} & {T\otimes_S\Omega_S} & {\Omega_T} & 0
	\arrow[from=1-1, to=1-2]
	\arrow[from=1-2, to=1-3]
	\arrow[from=1-3, to=1-4]
	\arrow["{\theta_g}"', curve={height=12pt}, dashed, from=1-4, to=1-3]
	\arrow[from=1-4, to=1-5]
\end{tikzcd}\]
and working locally, we may assume that $r_1,\dots,r_c$ is a regular sequence in $S$ generating $J$. By the definition of $\xi_g$, the element
\begin{equation}
\label{eq:starelement}
\tag{$*$}
dx  \otimes a^*\beta \in a^! \omega_T^\mydot = \bigwedge^d \Omega_{T[x]/T}[d] \otimes a^*\omega_T^\mydot
\end{equation}
maps to
$$dx \otimes a^*\left(r_1^* \wedge \cdots \wedge r_c^* \otimes dr_c \wedge \cdots dr_1 \wedge \theta_g(\beta)\right)$$
in the top right entry 
$$a^! \left( \left( \bigwedge^c {J \over J^2} \right)^\vee [-c]\otimes_T g^* \omega_S^\mydot\right) = \bigwedge^d \Omega_{T[x] \over T} [d]\otimes a^*\left( \left( \bigwedge^c {J \over J^2} \right)^\vee[-c] \otimes_T g^* \omega_S^\mydot \right)$$ 
at the corner of (9) in \autoref{keydiagram}. Now, by construction we have that $\theta_g(\beta) \in T \otimes_S \omega_S^\mydot = g^* \omega_S^\mydot$. Thus, using that $S \to T$ is surjective, we may represent $\theta_g(\beta) = 1 \otimes \alpha = g^* \alpha$ for $\alpha \in \omega_S^\mydot$. Note that the choice of $\alpha$ is not canonical: however, the difference in any two choices is in $J \omega_S^\mydot$, so that this ambiguity disappears after wedging with $dr_c\wedge\cdots\wedge dr_1$ and shall not be a factor in what follows. Hence, the image of $dx  \otimes a^*\beta$ in the top right corner of  is then
$$dx \otimes a^*\left(r_1^* \wedge \cdots \wedge r_c^* \otimes dr_c \wedge \cdots dr_1 \wedge g^*\alpha\right)$$ in the top right corner of (9) in \autoref{keydiagram}.

We now apply the downward morphism $\kappa$ appearing on the right side of \autoref{keydiagram} (9). As discussed earlier, this map is obtained in two steps. First, the pullback is absorbed into the factor $\left(\bigwedge^c J/J^2\right)^\vee[-c]$; identifying $r_1,\dots,r_c$ with their images in $S[x]$, this yields an element of $\left(\bigwedge^c J[x]/J[x]^2\right)^\vee[-c]$. Second, the free module $\bigwedge^d \Omega_{T[x]/T}[d] = g'{^*} \bigwedge^d \Omega_{S[x]/S}[d]$ is commuted past the degree-$c$ factor, introducing a sign $(-1)^{cd}$. Consequently, $\kappa$ sends the above element to
$$(-1)^{cd} (r_1^* \wedge \cdots \wedge r_c^*) \otimes g'^*(dx \otimes dr_c \wedge \cdots \wedge dr_1 \wedge a'^*\alpha)$$
inside the entry
\[
    \left( \bigwedge^c {J[x] \over J[x]^2}\right)^\vee[-c] \otimes g'^* a'^! \omegacan{S} = \left( \bigwedge^c {J[x] \over J[x]^2}\right)^\vee[-c] \otimes g'^* \left(\bigwedge^d \Omega_{S[x] \over S}[d] \otimes_{S[x]} a'^* \omegacan{S}\right).
\]
at the bottom right corner of \autoref{keydiagram} (9).

Finally, we apply the inverse of the bottom horizontal arrow (i.e. the inverse of the bottom map in \autoref{keydiagram} (10)). Since $\xi_{a'}^{-1}$ is given by pullback and exterior product, this yields
\begin{equation}
    \label{eq:2starelement}
    \tag{$**$}
    (-1)^{cd} (r_1^* \wedge \cdots \wedge r_c^*) \otimes g'^*(dx \wedge dr_c \wedge \cdots \wedge dr_1 \wedge a'^*\alpha).
\end{equation}
Thus, in conclusion, starting from the element \autoref{eq:starelement} and moving clockwise around the exterior square of \autoref{keydiagram} yields \autoref{eq:2starelement} inside
the bottom left entry $\left(\bigwedge^c {J[x] \over J[x]^2}\right)^\vee[-c] \otimes g'^* \omegacan{S[x]}$ of \autoref{keydiagram} (10).

To conclude, we need to see that moving counterclockwise to the bottom left gives the same result. As before, the first step is straightforward: applying $\xi_a^{-1}$ to our starting element from \autoref{eq:starelement} gives
\[
dx \otimes a^*\beta \longmapsto dx \wedge a^*\beta \in \omega_{T[x]}^\mydot
\]
in the top left corner of \autoref{keydiagram} (8).
We next apply $\xi_{g'}$. To do so, we must choose a splitting $\theta_{g'}$ of the conormal sequence for
$g' : S[x] \twoheadrightarrow T[x]$ compatibly with the fixed splitting $\theta_g$ for $g$. Consider the diagram below.

\[\begin{tikzcd}[cramped]
	& {T[x] \otimes_S J/J^2} & {J[x]/J[x]^2} \\
	0 & {T \otimes_S S[x] \otimes_S \Omega_S} & {T\otimes_S \Omega_{S[x]}} & {T\otimes_S \Omega_{S[x]/S}} & 0 \\
	0 & {T[x]\otimes_{S[x]}S[x]\otimes_{S} \Omega_S} & {T[x]\otimes_{S[x]}\Omega_{S[x]}} & {T[x]\otimes_{S[x]} \Omega_{S[x]/S}} & 0 \\
	0 & {T[x]\otimes_T \Omega_T} & {\Omega_{T[x]}} & {\Omega_{T[x]/T}} & 0
	\arrow[equals, from=1-2, to=1-3]
	\arrow[hook, from=1-2, to=2-2]
	\arrow[hook, from=1-3, to=2-3]
	\arrow[from=2-1, to=2-2]
	\arrow[from=2-2, to=2-3]
	\arrow[equals, from=2-2, to=3-2]
	\arrow[from=2-3, to=2-4]
	\arrow[equals, from=2-3, to=3-3]
	\arrow[from=2-4, to=2-5]
	\arrow[equals, from=2-4, to=3-4]
	\arrow[from=3-1, to=3-2]
	\arrow[from=3-2, to=3-3]
	\arrow[two heads, from=3-2, to=4-2]
	\arrow[from=3-3, to=3-4]
	\arrow[two heads, from=3-3, to=4-3]
	\arrow[from=3-4, to=3-5]
	\arrow[equals, from=3-4, to=4-4]
	\arrow[from=4-1, to=4-2]
	\arrow["{a^* \theta_g}"', shift right=3, curve={height=6pt}, dashed, from=4-2, to=3-2]
	\arrow[from=4-2, to=4-3]
	\arrow["{\theta_{g'}}"', shift right=3, curve={height=6pt}, dashed, from=4-3, to=3-3]
	\arrow[from=4-3, to=4-4]
	\arrow[from=4-4, to=4-5]
\end{tikzcd}\]

\noindent
where the bottom three rows are short exact sequences. The left column is obtained by
pullback by $a^*$ from the conormal sequence for $g$, which is thus split by $a^*\theta_g = T[x] \otimes_T \theta_g$.
The map $\theta_{g'}$ is required to split the indicated surjection in the middle column, which is the conormal sequence for $g'$.

The right terms $\Omega_{T[x]/T}$ of the horizontal exact sequences are free, generated by
$dx_1,\dots,dx_d$. Choosing the splittings of the horizontal sequences determined by the polynomial coordinates $x_1,\dots,x_d$,
\[
\begin{array}{c}
0 \to T[x]\otimes_S\Omega_S \to T[x]\otimes_{S[x]}\Omega_{S[x]} \to \Omega_{T[x]/T} \to 0 \\
     0 \to T[x]\otimes_T\Omega_T \to \Omega_{T[x]} \to \Omega_{T[x]/T} \to 0
     
\end{array}
\]
these splittings are compatible with the vertical maps in the diagram, since the generators
$dx_i$ are preserved throughout. (Alternatively, since the right terms can be canonically
identified, a splitting of the top sequence induces one for the bottom.) Consequently, the middle terms admit compatible direct sum
decompositions
\[
\Omega_{T[x]} \cong (T[x]\otimes_T \Omega_T)\oplus \Omega_{T[x]/T},
\quad
T[x]\otimes_{S[x]} \Omega_{S[x]} \cong (T[x]\otimes_S \Omega_S)\oplus \Omega_{T[x]/T}.
\]
We define $\theta_{g'}$ to be the unique splitting of the middle vertical surjection such that:
\begin{itemize}
\item on the summand $T[x]\otimes_T \Omega_T$, it agrees with the base change
$a^*\theta_g=T[x]\otimes_T\theta_g$;
\item on the summand $\Omega_{T[x]/T}$, it is the identity.
\end{itemize}

With $\theta_{g'}$ fixed as above, we can now compute the image of
$dx\wedge a^*\beta$ under $\xi_{g'}$. As above, we continue to 
assume that $r_1,\dots,r_c$ is a regular sequence in $S$ generating $J$, and thus the same elements (viewed in $S[x]$) form a regular sequence generating $J[x]$.
By definition of $\xi_{g'}$ (using the chosen splitting $\theta_{g'}$), we have
\[
\xi_{g'}(dx\wedge a^*\beta)
=
(r_1^*\wedge\cdots\wedge r_c^*)
\;\otimes\;
dr_c\wedge\cdots\wedge dr_1 \wedge \theta_{g'}(dx\wedge a^*\beta)
\]
in $\left(\bigwedge^c \frac{J[x]}{J[x]^2}\right)^\vee[-c]\otimes g'^*\omega_{S[x]}^\mydot$.
By construction, $\theta_{g'}$ is the identity on the summand
$\Omega_{T[x]/T}$, hence $\theta_{g'}(dx)=dx$, and on the summand
$T[x]\otimes_T\Omega_T$ it agrees with the extension of scalars
$a^*\theta_g := T[x]\otimes_T\theta_g$.
Therefore
\[
\theta_{g'}(dx\wedge a^*\beta)=dx\wedge a^*\theta_g(\beta).
\]
Substituting this into the preceding formula yields
\[
\begin{aligned}
\xi_{g'}(dx\wedge a^*\beta)
&=
(r_1^*\wedge\cdots\wedge r_c^*)
\otimes
dr_c\wedge\cdots\wedge dr_1 \wedge dx \wedge a^*\theta_g(\beta) \\
&=
(-1)^{cd}\,
(r_1^*\wedge\cdots\wedge r_c^*)
\otimes
dx \wedge dr_c\wedge\cdots\wedge dr_1 \wedge a^*\theta_g(\beta).
\end{aligned}
\]
Using the identifications
\[
a^*\theta_g(\beta)= a^*(g^*\alpha) = g'^*\bigl(a'^* \alpha \bigr)
\qquad\text{and}\qquad
dx = g'^*(dx)\in g'^*\Omega_{S[x]/S},
\]
we may rewrite the last expression as
\begin{equation}
\label{eq:3starelement}
\tag{$***$}
\xi_{g'}(dx\wedge a^*\beta)
=
(-1)^{cd}\,
(r_1^*\wedge\cdots\wedge r_c^*)
\otimes
g'^*\!\left(dx\wedge dr_c\wedge\cdots\wedge dr_1 \wedge a'^*\alpha\right).
\end{equation}
This coincides with the element obtained by traversing the exterior of the diagram
clockwise starting from $dx\otimes a^*\beta$, computed in \autoref{eq:2starelement}.
Therefore the exterior square commutes on generators $dx\otimes a^*\beta$, and hence
commutes. This completes the proof. \end{proof}

\subsection{Canonical dualizing complex for $F$-finite rings -- a sketch}
\label{ss:CanDualFFin}
The aim of this section is to sketch how the results of \autoref{subsec.DualityForFFiniteRegular} in the regular case just obtained via explicit duality theory can be combined with our analysis of Gabber's construction in \autoref{sec:gabberconstr} to obtain a dualizing complex up to canonical isomorphism for an $F$-finite ring that is independent of the various choices.  However, checking the necessary compatibilities soon becomes unwieldy. For this reason, in the present section we limit ourselves to proving that any two dualizing complexes constructed via Gabber’s result are isomorphic. We do not attempt here to verify that this isomorphism enjoys all of the compatibilities one would want in a complete theory. Instead, we point the reader to \autoref{sec.DualizingComplexesAsUnits}, where an alternative viewpoint allows one to bypass these verifications.

By \autoref{gabbify}, the choice of a $p$-generating set $x$ of an $F$-finite Noetherian ring $A$ yields an $F$-finite regular ring $S$ with a $p$-basis $x_\infty$ which under the surjection $\pi \colon S \to A$ bijectively maps $x_\infty$ to $x$. 
Since we have defined the canonical dualizing complex for regular $F$-finite rings via differential forms in the preceding section we can use this to define dualizing complexes on any $F$-finite ring, following  \cite{Gabber_someTstructures}.  Again, we know this yields a good theory via \autoref{sec.DualizingComplexesAsUnits}, but the goal in this section is to explore the classical approach.

\begin{definition}
\label{def.GeneralDefinitionOfCanonicalForFFiniteRings}
Suppose $A$ is an $F$-finite ring and $S \to A$ is finite type with $S$ a regular $F$-finite ring, for instance surjective.  We define 
\begin{equation}
    \omegacan{A} \coloneqq \pi^! \omegacan{S}    
\end{equation}
as the canonical dualizing complex for $A$. 
\end{definition}

The difficulty is that the choice of regular ring is far from unique: there may be many regular rings \(S\) surjecting onto \(A\), or more generally admitting a finite type map to \(A\), in part because \(A\) may possess many different \(p\)-generating sets.

Our goal is therefore to prove that for any two regular rings \(R\) and \(S\) with finite type maps
\[
    R \xrightarrow{\pi} A \xleftarrow{\tau} S,
\]
the associated dualizing complexes \(\pi^! \omega_R^{\mydot}\) and \(\tau^! \omega_S^{\mydot}\) are isomorphic through an explicit isomorphism. As discussed above, a complete theory would require stronger compatibilities than this, and we will not check them here.

% Of course, there is a problem.  There are potentially many different regular rings $S$ surjecting onto (or mapping finite type to)  $A$ not least because there are many different $p$-generating sets for $A$.  

% Our goal is to show that for any two regular rings $R, S$ with finite type maps 
% \[
%     R \xrightarrow{\pi} A \xleftarrow{\tau} S
% \]
% the induced dualizing complexes $\pi^! \omega_R^{\mydot}$ and $\tau^! \omega_S^{\mydot}$ are isomorphic via an explicit isomorphism.  As mentioned above, to develop a fully suitable theory, one needs much stronger compatibilities that we will not check. % We will not develop this full theory as it is even messier than the regular case we did above, and as mentioned in the introduction we will bypass this using a more modern perspectives on dualizing complexes.  We will construct the isomorphism $\pi^! \omega_R^{\mydot}\cong \tau^! \omega_S^{\mydot}$, however.

\begin{lemma}
    \label{lem.TwoRegularRingsLocalizationCompletion}
    Suppose $\pi : R \xrightarrow{f} R' \xrightarrow{\pi'} A$ is a composition such that both $R$ and $R'$ are $F$-finite regular rings.  

    Suppose additionally that the following conditions are satisfied.
    \begin{enumerate}
        \item $f$ is flat, geometrically regular, and $\pi$ (and hence $\pi'$) is of finite type,\label{xi'.cond.MapConditions}
        \item $A \to A \otimes_R R'$ is an isomorphism.\label{xi'.cond.TensorIsom}
        \item $\Omega_R \otimes_R R' \to \Omega_{R'}$ is an isomorphism (and hence $\omegacan{R}\otimes_R R' \cong \omegacan{R'}$).\label{xi'.cond.Omega}
    \end{enumerate}
    Notice these conditions hold for example if $f$ is a completion along an ideal in the kernel of $\pi$, a localization at a multiplicative set mapping to units in $A$, or a combination thereof.
    Then we have an isomorphisms $\pi^! \omega_R^{\mydot} \cong \pi'^! \omega_{R'}^{\mydot}$ and:
    \[
        \pi^! \omega_R^{\mydot} \otimes_R R' \xrightarrow{\sim} {\pi'}^!(\omega_R^{\mydot} \otimes_R R').
    \]
\end{lemma}
\begin{proof}
    To see the first isomorphism in the result, notice that $\pi^! (\omegacan{R})$  can be viewed in $D(A)$ and use \autoref{xi'.cond.TensorIsom}.  
    
    We now prove the second isomorphism.  We observe that $\omegacan{R} \otimes_R R' \cong \omegacan{R'}$ by \autoref{prop.OmegaAndLocCompl}.  We then obtain an isomorphism 
    \[
        \pi^!(-) \otimes_R R' \to {\pi'}^!(- \otimes_R R')
    \]
    from
    \cite[Theorem 4.8.3 using Definition 4.8.2(d)]{LipmanHashimotoFoundationsOfGDualityForDiagrams} (some special cases which do not include completion can also be found in \cite[VI, Theorem 5.5]{HartshorneResidues} or \cite{ConradGDualityAndBaseChange}).
    We see that 
    \[
         \pi^! \omegacan{R} \otimes_R R' \cong \pi'^!(\omegacan{R} \otimes_R R') 
    \]
    This finishes the proof.
    %This proves part of the result.  $\pi^! (\omegacan{R})$ can be viewed in $D(A)$, we see that \autoref{xi'.cond.TensorIsom} implies that $\pi^! (\omegacan{R}) \otimes_R R' \cong ???$ as desired.
    %\todo[inline]{Karl working here.}
\end{proof}

\begin{remark}
    \label{rem.LipmanDualityOk}
    There is a subtlety in the above proof that we wish to address.  In the above proof we cited a fact about $f^!$ and base change from \cite{LipmanHashimotoFoundationsOfGDualityForDiagrams}.  
    Suppose we have a finite type map of $F$-finite rings $A \xrightarrow{f} B$.  We can form the functor $f^!$, say on $D_{\mathrm{coh}}^+$ ,in various different ways.  We can use the factorization into smooth maps and closed immersions as above (the strategy found in Hartshorne and Conrad).  Alternately, we can compactify these maps to proper maps (the strategy of Verdier, Deligne, etc. see also \cite{LipmanHashimotoFoundationsOfGDualityForDiagrams}, \cite[\href{https://stacks.math.columbia.edu/tag/0DWE}{Tag 0DWE}]{stacks-project}) to construct these functors.  
   
    It is common to see that these two $f^!$ agree in the concrete settings we needed above, see \cite[\href{https://stacks.math.columbia.edu/tag/0AA2}{Tag 0AA2}]{stacks-project} and \cite[\href{https://stacks.math.columbia.edu/tag/0AA1}{Tag 0AA1}]{stacks-project}.  However, that is not enough if we wanted to develop the full theory; see the discussion after \autoref{prop.DualizingComplexesIsomorphicForArbitraryFFiniteRing}.
\end{remark}

\begin{proposition}
    \label{prop.DualizingComplexesIsomorphicForArbitraryFFiniteRing}
    Suppose that $\pi : R \to A$ and $\tau : S \to A$ are finite type maps from $F$-finite regular rings to $A$.  Pick a $p$-generating set $\overline{x}$ for $R$, and likewise $\overline{y}$ for $S$.  Choose $\tld x$ and $\tld y$ respectively, elements of $A$, so that $x := {\pi}(\overline{x}), \tld{x}$ and $y := {\tau}(\overline{y}), \tld{y}$ are $p$-generating sets of $A$.  We form the following diagram:
\[
    \xymatrix@R=50pt{
        G(R;\overline{x}) \ar@/_1pc/[drr]_{\overline\pi} \ar@{->>}[d]_u \ar[r]^-{\tld f} & G(R;\overline{x})[\tld X][Y] \ar[dr]^{\tld \pi} \ar[r]^-{\overline{p}} & G(A;x,y) \ar@{->>}[d]^{\kappa}  & \ar[l]_-{\overline{q}} G(S; \overline{y})[\tld Y][X] \ar[dl]_{\tld \tau} & \ar[l]_-{\tld g} G(S; \overline{y}) \ar@{->>}[d]^v \ar@/^1pc/[dll]^{\overline\tau} \\
        R \ar[rr]_{{\pi}} & & A & & \ar[ll]^{{\tau}} S
    }
\]
where $\tld X \mapsto \tld x$, $\tld Y \mapsto \tld y$, and $X \mapsto x$, $Y \mapsto y$.
Note $\overline{p}$ and $\overline{q}$ are completion maps, and $\tld f$ and $\tld g$ are adjoining variables.
Then we have an isomorphism:
\[ 
    \pi^! \omegacan{R} \to \tau^! \omegacan{S}.
\]
\end{proposition}
\begin{proof}
    We know $\omegacan{R} \cong u^! \omegacan{G(R; \overline{x})}$ thanks to \autoref{thm.CanDualRegPullsback} or more precisely \autoref{defn.lciCase}.  We then obtain 
    \[
        \pi^! \omegacan{R} \cong \pi^! u^! \omegacan{G(R; \overline{x})} \cong \overline{\pi}^! \omegacan{G(R; \overline{x})}.
    \]
    Arguing similarly, as $\tld{f}$ is simply adjoining variables, we see that that this also agrees with $\tld{\pi}^! \omegacan{G(R; \overline{x})[\tld X][Y]}$ thanks to \autoref{thm.CanDualRegPullsback} or \autoref{defn:SmoothMapsUpperShriekPullbackForRegularRings}.  Finally, using \autoref{lem.TwoRegularRingsLocalizationCompletion} for the composition
    \[
        G(R; \overline{x})[\tld X][Y] \to G(A;x,y) \to A
    \]
    we see that $\pi^! \omegacan{R} \cong \tld{\pi}^! \omegacan{G(R; \overline{x})[\tld X][Y]} \to \kappa^! \omegacan{G(A;x,y)}$ is an isomorphism.  The same holds for $S$ and hence the result is proven.
\end{proof}

This is not nearly enough, unfortunately, to complete the proof that one can construct well behaved dualizing complexes globally as done in the regular case above.  First, one must prove that the isomorphism just obtained in \autoref{prop.DualizingComplexesIsomorphicForArbitraryFFiniteRing} is independent of choices.  One must also show it satisfies numerous compatibilities.  While we will sketch what we think needs to be done in \autoref{rem.WhatShouldBeDone} below, we will not carry  out the details.  Instead, using the perspective on dualizing complexes developed in the following \autoref{sec.DualizingComplexesAsUnits}, one can show by different methods that every Noetherian $F$-finite scheme $X$ has a dualizing complex, that restricts to the dualizing complex described above on affine charts, and such that for any finite type map $h : Y \to X$, one has $h^! \omegacan{X} \cong \omegacan{Y}$, as desired.

We now explain how we expect that the proof that $\omegacan{X}$ can be constructed for arbitrary $F$-finite Noetherian schemes and that the expected compatibilities hold, using the methods above, and of \cite{HartshorneResidues,ConradGDualityAndBaseChange,LipmanHashimotoFoundationsOfGDualityForDiagrams}, could be used to construct dualizing complexes on $F$-finite Noetherian schemes.  

\begin{remark}
    \label{rem.WhatShouldBeDone}
    Consider $\pi : R \xrightarrow{f} R' \xrightarrow{\pi'} A$ where $f$ is of finite type or satisfies the conditions of \autoref{lem.TwoRegularRingsLocalizationCompletion} and $\pi$ is of finite type.  One must show that the resulting isomorphism $\pi^! \omegacan{R} \to \pi'^! \omegacan{R'}$  commutes with certain base changes and behaves well with respect to extensions $R \xrightarrow{f} R'\xrightarrow{g} R'' \to A$ where $f,g$ both satisfy the conditions of \autoref{lem.TwoRegularRingsLocalizationCompletion}.  
    
    Now then, define $\chi_{x \leftarrow y}$ to be the isomorphism of \autoref{prop.DualizingComplexesIsomorphicForArbitraryFFiniteRing}, and observe it is a composition of isomorphisms 
    from \autoref{defn.lciCase} (see \autoref{thm.CanDualRegPullsback}) and \autoref{lem.TwoRegularRingsLocalizationCompletion} and their inverses.  Using the compatibilities for these maps, one should show that $\chi_{x \leftarrow y}$ is compatible with composition ($\chi_{z\leftarrow x} = \chi_{z \leftarrow y} \circ \chi_{y \leftarrow x}$) and also base changes well.  

    At that point, one can glue the resulting dualizing complexes on a given $X$ as in \cite[\href{https://stacks.math.columbia.edu/tag/0AU5}{Tag 0AU5}]{stacks-project}, see \cite[Remark on page 310]{HartshorneResidues}, \cite{BBD}. Compatibility of $\chi_{x \leftarrow y}$ with composition and localization should suffice for this.  Using the fact that $\omegacan{R}$ is defined via $R \to A$ of finite type and not necessarily surjective should allow one the necessary flexibility to prove that if $h : Y \to X$ is finite type, then $\omegacan{Y} = h^! \omegacan{X}$.   
\end{remark}

There is a technical issue that must be addressed if one is following the above outline.

\begin{remark}
    Following \autoref{rem.LipmanDualityOk}, we observe that if we are given finite type maps $A \xrightarrow{f} B \xrightarrow{g} C$ the two theories of $(-)^!$ described above produce a priori different isomorphisms
    \[  
        g^! f^! \simeq (gf)^!.
    \]  
    This is problematic since the isomorphism constructed in \cite{HartshorneResidues,ConradGDualityAndBaseChange} was only shown to be compatible with base change for the ``residually stable'' case, which completion maps do not satisfy. But we do need compatibility with completion.  Indeed, we need compatibilities of the isomorphism produced in \autoref{lem.TwoRegularRingsLocalizationCompletion} which mix the two compatibilities of upper shrieks. 

    Joseph Lipman explained to us how one should overcome this difficulty and in particular, that there is a pseudo-functorial isomorphism between the different constructions of these pseudo-functors $(-)^!$, at least when restricted to $D_{\mathrm{coh}}^+$.  
    
    We sketch his argument now although have not carried out the details.  By \cite[Theorem 4.8.1]{LipmanHashimotoFoundationsOfGDualityForDiagrams}, the pseudo-functor $(-)^!$ is uniquely determined up to isomorphism on the category of separated finite type maps, by its behavior in certain circumstances (the necessary statements for proper and \'etale maps in (i) and (ii) of \cite[Theorem 4.8.1]{LipmanHashimotoFoundationsOfGDualityForDiagrams} are readily found in \cite[VII, 3.4(c) and 3.4(a) VAR 3]{HartshorneResidues}, but (iii) requires more work).  Lipman explained to us that \emph{everything in \cite[Section 4.8]{LipmanHashimotoFoundationsOfGDualityForDiagrams} is valid after one substitutes ``open immersion'' for ``\'etale''}.  Assuming that has been verified (which we did not do), one obtains that various isomorphisms and compatibilities in \cite{HartshorneResidues} satisfy the conditions we believe are needed to carry out the proof.
\end{remark}

\bibliographystyle{alpha}
\bibliography{literatur}
\end{document}